\documentclass[11pt,a4paper,reqno]{amsart}

\usepackage[lmargin=1.2in, rmargin= 1.2in, tmargin=1.5in]{geometry}

\title{Nerves of enriched categories via necklaces}

\author{Arne Mertens}
\address[Arne Mertens]{Universiteit Antwerpen, Departement Wiskunde, Middelheimcampus,
Middelheimlaan 1,
2020 Antwerp, Belgium}
\email{arne.mertens@uantwerpen.be}

\makeatletter
\@namedef{subjclassname@2020}{
  \textup{2020} Mathematics Subject Classification}
\makeatother



\usepackage{graphicx}
\usepackage[english]{babel}
\usepackage{amsmath}
\usepackage{amssymb}
\usepackage{amsthm}
\usepackage[all,cmtip]{xy}
\usepackage{cancel}
\usepackage[alpine]{ifsym}
\usepackage{enumerate}
\usepackage{stmaryrd}
\usepackage{tikz}
\usepackage{tikz-cd}
\usepackage{quiver}
\usetikzlibrary{matrix}
\usepackage[toc,page]{appendix}
\usepackage{hyperref}
\usepackage[utf8]{inputenc}
\usepackage[style=alphabetic, maxnames = 50, maxalphanames=50, urldate=long]{biblatex}

\DeclareMathOperator{\sgn}{sgn}

\DeclareMathOperator{\Ob}{Ob}
\DeclareMathOperator{\Mor}{Mor}

\DeclareMathOperator{\id}{id}

\DeclareMathOperator{\Fun}{Fun}

\DeclareMathOperator{\Lan}{Lan}

\DeclareMathOperator*{\colim}{colim}

\DeclareMathOperator{\Set}{Set}

\DeclareMathOperator{\Mod}{Mod}

\DeclareMathOperator{\Ch}{Ch}

\DeclareMathOperator{\SSet}{SSet}
\DeclareMathOperator{\CSet}{CSet}

\DeclareMathOperator{\Colax}{Colax}
\DeclareMathOperator{\StrMon}{StrMon}

\DeclareMathOperator{\Quiv}{Quiv}
\DeclareMathOperator{\Cat}{Cat}

\DeclareMathOperator{\const}{const}
\DeclareMathOperator{\Dusk}{Dusk}
\DeclareMathOperator{\hc}{hc}
\DeclareMathOperator{\dg}{dg}
\DeclareMathOperator{\cub}{cub}

\newcommand{\fint}{\mathbf{\Delta}_{f}} 
\newcommand{\simp}{\mathbf{\Delta}} 
\newcommand{\nec}{\mathcal{N}ec} 
\newcommand{\ts}{S_{\otimes}} 
\newcommand{\Fs}{S^{Frob}_{\otimes}} 

\DeclareFontFamily{U}{min}{}
\DeclareFontShape{U}{min}{m}{n}{<-> dmjhira}{}
\newcommand{\yo}{\text{\usefont{U}{min}{m}{n}\symbol{'110}}}

\newtheorem{Thm}{Theorem}[section]
\newtheorem*{Thm*}{Theorem}
\newtheorem{Lem}[Thm]{Lemma}
\newtheorem{Prop}[Thm]{Proposition}
\newtheorem*{Prop*}{Proposition}
\newtheorem{Cor}[Thm]{Corollary}
\newtheorem*{Cor*}{Corollary}
\newtheorem{Q}{Question}

\theoremstyle{definition}
\newtheorem{Def}[Thm]{Definition}
\newtheorem{Ex}[Thm]{Example}
\newtheorem{Exs}[Thm]{Examples}
\newtheorem{Con}[Thm]{Construction}

\theoremstyle{remark}
\newtheorem{Rem}[Thm]{Remark}

\begin{document}

\begin{abstract}
We introduce necklicial nerve functors from enriched categories to simplicial sets, which include the homotopy coherent \cite{cordier1982sur}, differential graded \cite{lurie2016higher} and cubical nerves \cite{legrignou2020cubical}. It is shown that every necklicial nerve can be lifted to the templicial objects of \cite{lowen2024enriched}. Building on \cite{dugger2011rigidification}, we give sufficient conditions under which the left-adjoint of a necklicial nerve can be described more explicitly. As an application, we obtain novel and simple expressions for the left-adjoints of the dg-nerve and cubical nerve.
\end{abstract}

\maketitle

\tableofcontents

\section{Introduction}\label{section: Introduction}

\subsection{Motivation and main results}

Nerve functors have proven valuable tools in comparing different models for higher (enriched) categories and relating their homotopical properties. Frequently these nerve functors take the form of a right-adjoint functor from categories enriched over a suitable monoidal category $\mathcal{W}$ to simplicial sets. Examples of interest include the classical nerve of categories, the Duskin nerve of $2$-categories \cite{duskin2001simplicial}, the homotopy coherent nerve by Cordier \cite{cordier1982sur}, the differential graded nerve by Lurie \cite{lurie2016higher} and the cubical nerve by Le Grignou \cite{legrignou2020cubical}. Nerves from $\mathcal{W}\Cat$ can often be endowed with more structure, e.g. landing in simplicial objects instead. This is done in \cite{lack2008nerves} and \cite{moser2024homotopy} for example. In \cite{lowen2024enriched}, we introduced and studied \emph{templicial objects} $\ts\mathcal{V}$ as an enriched variant of simplicial sets over a suitable monoidal category $(\mathcal{V},\otimes,I)$ which recover simplicial objects when $\mathcal{V}$ is cartesian monoidal \cite{mertens2024discrete}. Our primary goal in this paper is therefore to study the categorical properties of (right-adjoint) functors of the form
$$
\mathcal{W}\Cat\rightarrow \ts\mathcal{V}.
$$
leaving their homotopical properties to future research. Inspired by the work of Dugger and Spivak \cite{dugger2011rigidification}, the combinatorics of \emph{necklaces} will play a critical role. Restricting to a subclass of what we call \emph{necklicial nerves}, we provide a general procedure for lifting them to templicial objects, and give conditions under which their left-adjoint can be described more explicitly.

Before outlining our main results, let us recall the classical procedure for producing nerve functors landing in simplicial sets $\SSet$. If $\mathcal{W}$ is cocomplete, then so is $\mathcal{W}\Cat$ \cite{wolff1974Vcat}\cite{kelly2001locally}. Hence, any diagram $\mathbb{D}: \simp\rightarrow \mathcal{W}\Cat$ on the simplex category $\simp$ gives rise to an adjunction
\begin{equation}\label{equation: classical adjunction}
L^{\mathbb{D}}: \SSet\leftrightarrows \mathcal{W}\Cat: N^{\mathbb{D}}
\end{equation}
The right-adjoint, i.e. the \emph{nerve}, $N^{\mathbb{D}}$ is simply defined by, for all $\mathcal{C}\in \mathcal{W}\Cat$ and $n\geq 0$:
$$
N^{\mathbb{D}}(\mathcal{C})_{n} = \mathcal{W}\Cat(\mathbb{D}(n),\mathcal{C})
$$
whereas the left-adjoint $L^{\mathbb{D}}$ is constructed by left Kan extension of $\mathbb{D}$ along the Yoneda embedding $\yo: \simp\hookrightarrow \SSet$:
$$
L^{\mathbb{D}}(X) = \Lan_{\yo}\mathbb{D} = \colim_{\substack{n\geq 0\\ \sigma\in X_{n}}}\mathbb{D}(n)
$$
This construction is invertible in the sense that it defines an equivalence of categories:
\begin{equation}\label{equation: classical equivalence}
\mathrm{Nerve}(\mathcal{W}\Cat)\simeq \Fun(\simp,\mathcal{W}\Cat)^{op}
\end{equation}
where the left hand side denotes the category of right-adjoint functors $\mathcal{W}\Cat\rightarrow \SSet$ and natural transformations between them.

This procedure raises two relevant questions. Motivated by (non-commutative) algebraic geometry, where dg-categories are considered as models for spaces, we can wonder how much of the linear structure of a dg-category $\mathcal{C}$ is retained by its dg-nerve $N^{dg}(\mathcal{C})$. In \cite{lowen2023frobenius}, we answered this by lifting $N^{dg}(\mathcal{C})$ to a templicial module. In general, if the monoidal category $\mathcal{W}$ is itself tensored and enriched over a symmetric monoidal closed category $(\mathcal{V},\otimes,I)$, we ask the following.

\begin{Q}\label{question: enriching nerves}
Can the nerve $N^{\mathbb{D}}$ be lifted to an enriched nerve $N^{D}_{\mathcal{V}}: \mathcal{W}\Cat\rightarrow \ts\mathcal{V}$ along the canonical forgetful functor $\tilde{U}: \ts\mathcal{V}\rightarrow \SSet$?
\end{Q}

\noindent Secondly, colimits of (enriched) categories are notoriously hard to compute.

\begin{Q}\label{question: explicitizing left-adjoints}
Given a simplicial set $K$, how can the $\mathcal{W}$-category $L^{\mathbb{D}}(K)$ be described more explicitly in terms of $\mathbb{D}$?
\end{Q}

Several results in the literature provide answers to these questions for specific choices of $\mathcal{W}$ and $\mathcal{V}$. In view of Question \ref{question: enriching nerves}, let us first assume $\mathcal{W}$ and $\mathcal{V}$ to be cartesian monoidal (that is, their monoidal products are given by the cartesian product). Then the nerve $N^{\mathbb{D}}$ can often be lifted to a functor $\mathcal{W}\Cat\rightarrow S\mathcal{V}$ where $S\mathcal{V} = \mathcal{V}^{\simp^{op}}$ denotes the category of simplicial objects in $\mathcal{V}$. Examples from the literature include the the \emph{$2$-nerve} of Lack and Paoli \cite{lack2008nerves} ($\mathcal{W} = \mathcal{V} = \Cat$) and the \emph{homotopy coherent nerve for $(\infty,n)$-categories} by Moser, Rasekh and Rovelli \cite{moser2024homotopy} ($\mathcal{W} = \mathcal{V} = \SSet^{\Theta^{op}}$). We will return to the latter of in Section \ref{section: Examples}. If $\mathcal{V}$ is not cartesian however, it is no longer possible to define reasonable nerve functors landing in simplicial objects, but we must pass to templicial objects $\ts\mathcal{V}$ instead. We will recall templicial objects in \S\ref{subsection: Templicial objects}. Examples of such nerves $\mathcal{W}\Cat\rightarrow \ts\mathcal{V}$ include the homotopy coherent nerve \cite{lowen2024enriched} ($\mathcal{W} = S\mathcal{V}$, $\mathcal{V}$ non-cartesian), and the dg-nerve \cite{lowen2023frobenius} ($\mathcal{W} = \Ch(k)$, $\mathcal{V} = \Mod(k)$ for $k$ a commutative ring).

Concerning Question \ref{question: explicitizing left-adjoints}, Dugger and Spivak give an explicit description of the left-adjoint $\mathfrak{C}: \SSet\rightarrow \Cat_{\Delta}$ of the homotopy coherent nerve $N^{hc}$ (i.e. in the case $\mathcal{W} = \SSet$ and $\mathcal{V} = \Set$) in \cite{dugger2011rigidification}. Essential to their approach is the use of \emph{necklaces}, which were first introduced by Baues \cite{baues1980geometry} as ``cellular strings''.

A general approach to nerves $\mathcal{W}\Cat\rightarrow \SSet$ (i.e. the case $\mathcal{V} = \Set$) was put forward by Le Grignou in \cite{legrignou2020cubical} using the category of cubes with connections $\square$. Any \emph{strong monoidal} diagram $H: \square\rightarrow \mathcal{W}$ induces a monoidal adjunction $\CSet\leftrightarrows \mathcal{W}$ by left Kan extension, where $\CSet = \Set^{\square^{op}}$ denotes the category of cubical sets. Applying this adjunction to hom-objects, one obtains a nerve functor as the following composite
\begin{equation}\label{diagram: cubical nerve procedure}
N^{H}: \mathcal{W}\Cat\rightarrow \Cat_{\square}\xrightarrow{N^{\cub}} \SSet
\end{equation}
where $\Cat_{\square} = \CSet\Cat$ and $N^{\cub}$ denotes the cubical nerve from loc. cit.

In the present paper, we address both Questions \ref{question: enriching nerves} and \ref{question: explicitizing left-adjoints} simultaneously through a general procedure analogous to \eqref{diagram: cubical nerve procedure}, where we make use of necklaces instead of cubes, and we allow arbitrary $\mathcal{V}$. Full details are given in \S\ref{subsection: A general procedure}. Let us denote the category of necklaces by $\nec$. Starting now from a \emph{colax monoidal} diagram
\begin{equation}\label{diagram: conecklicial object}
D: \nec\rightarrow \mathcal{W}
\end{equation}
which again by left Kan extension produces an adjunction $\mathcal{V}^{\nec^{op}}\leftrightarrows \mathcal{W}$, the right-adjoint of which will always be lax-monoidal (with respect to the Day convolution on $\mathcal{V}^{\nec^{op}}$). Applying this right-adjoint to hom-objects, we obtain a nerve functor as the following composite (also see Construction \ref{construction: nerve generated by necklicial diagram}):
$$
N^{D}_{\mathcal{V}}: \mathcal{W}\Cat\rightarrow \mathcal{V}^{\nec^{op}}\Cat\xrightarrow{(-)^{temp}} \ts\mathcal{V}
$$
where the functor $(-)^{temp}$ was constructed in \cite{lowen2024enriched}. Moreover, this functor turns out to have a left-adjoint $L^{D}$ when $D$ is strong monoidal (Proposition \ref{proposition: nerve gen. by strong mon. diagram has left-adj.}). This procedure recovers that of \cite{legrignou2020cubical} as follows. There is the strong monoidal functor
$$
\dim: \nec\rightarrow \square
$$
constructed by Rivera and Zeinalian in \cite{rivera2018cubical}, which we will come back to in detail in \S\ref{subsection: Necklaces versus cubes}. Then we find that for any $H$ as above, precisely $N^{H\dim}_{\Set}\simeq N^{H}$ (Corollary \ref{proposition: cubical nerve is necklicial}). The construction via necklaces is thus more general than the one via cubes. Moreover, necklaces allow for explicit descriptions of the left-adjoints, and thus provide an answer to Question \ref{question: explicitizing left-adjoints}, which cubes don't seem to do.

Our first main result addresses Question \ref{question: enriching nerves} above.

\begin{Thm*}[Proposition \ref{proposition: embedding of necklace diagrams into necklicial functors} and Theorem \ref{theorem: underlying D-nerve is cat. nerve assoc. to D}]
There is a fully faithful functor
$$
\Phi: \StrMon(\nec,\mathcal{W})\hookrightarrow \Fun(\simp,\mathcal{W}\Cat)
$$
such that for any strong monoidal functor $D: \nec\rightarrow \mathcal{W}$ and $\mathbb{D} = \Phi(D)$, we have a natural isomorphism of functors $\mathcal{W}\Cat\rightarrow \SSet$:
$$
\tilde{U}\circ N^{D}_{\mathcal{V}}\simeq N^{\mathbb{D}}
$$
\end{Thm*}

We call the nerve functors $N^{\mathbb{D}}$ arising in this way \emph{necklicial} (Definition \ref{definition: necklicial nerve}). Certainly not every nerve $\mathcal{W}\Cat\rightarrow \SSet$ is necklicial, but as we'll see in Section \ref{section: Examples}, all the examples mentioned above are. Moreover, we can identify exactly which diagrams $\mathbb{D}: \simp\rightarrow \mathcal{W}$ produce necklicial nerves and can thus be lifted in this way (see Proposition \ref{proposition: necklicial nerve charac.}).

We then move on to Question \ref{question: explicitizing left-adjoints} by identifying conditions on a strong monoidal diagram $D$ which allows to describe the left-adjoint more explicitly. By $F: \Set\rightarrow \mathcal{V}: S\mapsto \coprod_{x\in S}I$ we denote the free functor, and by $\iota: \nec_{-}\hookrightarrow \nec$ the inclusion of active surjective necklace maps (see \S\ref{subsection: Necklaces and necklace categories} for more details). Then our second main result is the following.

\begin{Thm*}[Corollary \ref{corollary: explicitation of the left-adjoint for simp. sets}]
Let $D: \nec\rightarrow \mathcal{W}$ be a strong monoidal diagram and $\pi: \mathcal{W}\rightarrow \mathcal{V}$ a colimit and tensor preserving $\mathcal{V}$-functor. Suppose there exists $D': \nec_{-}\rightarrow \mathcal{V}$ such that $\pi D\simeq \Lan_{\iota}D'$. Then for any simplicial set $K$ with $a,b\in K_{0}$, we have a canonical isomorphism in $\mathcal{V}$:
$$
\pi(L^{D}(K)(a,b))\simeq \coprod_{T\in \nec}F(K^{nd}_{T}(a,b))\otimes D'(T)
$$
where $K^{nd}_{T}(a,b)$ is the set of totally non-degenerate maps $T\rightarrow K_{a,b}$ in $\SSet_{*,*}$.
\end{Thm*}

This recovers \cite[Corollary 4.8]{dugger2011rigidification} which makes explicit the left-adjoint of the homotopy coherent nerve. What's more, it can now be applied to other nerves such as the dg-nerve and the cubical nerve.

\begin{Cor*}[Corollaries \ref{corollary: results for the dg-nerve} and \ref{corollary: results for the cubical nerve}]
Let $K$ be a simplicial set with $a,b\in K_{0}$.
\begin{enumerate}[1.]
\item Let $L^{dg}: \SSet\rightarrow k\Cat_{dg}$ denote the left-adjoint of the dg-nerve. Then for all $n\in \mathbb{Z}$, we have an isomorphism of $k$-modules
$$
L^{dg}(K)_{n}(a,b)\simeq \bigoplus_{\substack{T\in \nec\\ \dim(T) = n}}\!\!\!k.K^{nd}(a,b)
$$
\item Let $L^{cub}: \SSet\rightarrow \Cat_{\square}$ denote the left-adjoint of the cubical nerve. Then for all $n\geq 0$, we have a bijection
$$
L^{cub}(K)_{n}(a,b)\simeq \coprod_{\substack{T\in \nec\\ [1]^{n}\twoheadrightarrow [1]^{\dim(T)}\\ \text{surjective}}}\!\!\!\!\!\!K^{nd}_{T}(a,b)
$$
\end{enumerate}
\end{Cor*}

Finally, we relate our necklicial nerves to \emph{quasi-categories in $\mathcal{V}$} and \emph{Frobenius structures} introduced in \cite{lowen2024enriched} and \cite{lowen2023frobenius} respectively. Quasi-categories in $\mathcal{V}$ are templicial objects in $\mathcal{V}$ satisfying an analogue of the weak Kan condition and they precisely recover Joyal's classical quasi-categories \cite{joyal2002quasi} when $\mathcal{V} = \Set$. Frobenius structures are associative multiplications on templicial objects which a lot of nerve functors come naturally equipped with. In particular it was shown in \cite[Proposition 3.16]{lowen2023frobenius} that the lift $N^{dg}_{k}: k\Cat_{dg}\rightarrow \ts\Mod(k)$ of the dg-nerve induces an equivalence between non-negatively graded dg-categories and templicial modules with a Frobenius structure. By $\overline{\nec}$ we denote the extended necklace category, which is detailed in \S\ref{subsection: Frobenius structures}. Our final main result is the following.

\begin{Thm*}[Theorem \ref{theorem: D-nerve is quasi-cat.} and Corollary \ref{corollary: D-nerve has Frob. structure}]
Let $\mathcal{C}\in \mathcal{W}\Cat$, and $D$ as in \eqref{diagram: conecklicial object}.
\begin{enumerate}[1.]
\item Assume that for all $A,B\in \Ob(\mathcal{C})$ and $0 < j < n$ the following lifting problem in $\mathcal{W}$ has a solution:
\[\begin{tikzcd}
	{\underset{\substack{T\rightarrow (\Lambda^{n}_{j})_{0,n} \text{in }\SSet_{*,*}\\ T\in \nec}}{\colim}D(T)} & {\mathcal{C}(A,B)} \\
	{D\left(\Delta^{n}\right)}
	\arrow[from=1-1, to=2-1]
	\arrow[from=1-1, to=1-2]
	\arrow[dashed, from=2-1, to=1-2]
\end{tikzcd}\]
Then $N^{D}_{\mathcal{V}}(\mathcal{C})$ is a quasi-category in $\mathcal{V}$. In particular, the simplicial set $N^{\mathbb{D}}(\mathcal{C})$ is an ordinary quasi-category for $\mathbb{D} = \Phi(D)$.
\item Suppose that $D$ extends to a colax monoidal diagram $\overline{\nec}\rightarrow \mathcal{W}$. Then $N^{D}_{\mathcal{V}}(\mathcal{C})$ has a Frobenius structure.
\end{enumerate} 
\end{Thm*}

The advantage of the first part of this theorem is that the colimit above is now considered in $\mathcal{W}$ (instead of $\mathcal{W}\Cat$). This is often more computable, especially in cases like $\mathcal{W} = \SSet$ and $\mathcal{W} = \Ch(k)$ where colimits are calculated pointwise.

The remainder of the paper is then devoted to applying the theorems above to examples of nerves from the literature, whenever applicable. For each example we identify the diagram $D: \nec\rightarrow \mathcal{W}$ which generates the nerve. A summary of these diagrams is given below. Here, we've denoted $\const_{I}: \nec\rightarrow \mathcal{V}$ for the constant functor on $I$, $N: \Cat\rightarrow \SSet$ for the classical nerve functor, $F$ for the free functor (which is the identity if $\mathcal{V} = \Set$), $N^{\square}_{\bullet}: C\Mod(k)\rightarrow \Ch(k)$ for the cubical chains functor, and $\yo$ for the Yoneda embedding. See the relevant subsections of Section \ref{section: Examples} for more details.

\begin{figure}[h]
\begin{center}
\begin{tabular}{|c|c|c|c|}
\hline
Nerve $N^{D}_{\mathcal{V}}$ & $\mathcal{V}$ & $\mathcal{W}$ & $D$\\
\hline
\hline
ordinary nerve & $\Set$ & $\Set$ & $\const_{*}$\\
\hline
templicial nerve \cite{lowen2024enriched} & $\mathcal{V}$ & $\mathcal{V}$ & $\const_{I}$\\
\hline
Duskin nerve \cite{duskin2001simplicial} & $\Set$ & $\Cat$ & $\Dusk = \dim$\\
\hline
homotopy coherent nerve \cite{cordier1982sur} & $\Set$ & $\SSet$ & $\hc = FN\dim$\\
\hline
templicial hc-nerve \cite{lowen2024enriched} & $\mathcal{V}$ & $S\mathcal{V}$ & $\hc$\\
\hline
hc-nerve for $(\infty,n)$-categories \cite{moser2024homotopy} & $\SSet^{\Theta^{op}}$ & $\SSet^{\Theta^{op}}$ & $\hc$\\
\hline
differential graded nerve \cite{lurie2016higher} & $\Set$ & $\Ch(k)$ & $dg_{\bullet} = N^{\square}_{\bullet}F\yo_{\square}\dim$\\
\hline
templicial dg-nerve \cite{lowen2023frobenius} & $\Mod(k)$ & $\Ch(k)$ & $dg_{\bullet}$\\
\hline
cubical nerve \cite{legrignou2020cubical} & $\Set$ & $\CSet$ & $\cub = F\yo_{\square}\dim$\\
\hline 
Frobenius forgetful functor \cite{lowen2023frobenius} & $\mathcal{V}$ & $\mathcal{V}^{\overline{\nec}^{op}}$ & $F\yo_{\overline{\nec}}\vert_{\nec}$\\
\hline
\end{tabular}
\end{center}
\end{figure}

Note that the majority of them factor through $\dim$ and thus also fit into the paradigm of \cite{legrignou2020cubical}. Further note that they are interrelated, most notably when $\mathcal{V} = \Mod(k)$. It is shown throughout Section \ref{section: Examples} that the following diagram commutes everywhere up to isomorphism, except in that $N_{\bullet}\circ tr$ and $N^{\square}_{\bullet}$ are only quasi-isomorphic:

\[\begin{tikzcd}
	\nec && {S\Mod(k)} \\
	& {C\Mod(k)} && {\Mod(k)} \\
	&& {\Ch(k)}
	\arrow["\hc"{description}, curve={height=-6pt}, from=1-1, to=1-3]
	\arrow["\cub"{description}, from=1-1, to=2-2]
	\arrow["{\const_{k}}"{description, pos=0.3}, curve={height=-12pt}, from=1-1, to=2-4]
	\arrow["{\dg}"{description}, curve={height=24pt}, from=1-1, to=3-3]
	\arrow["{H_{0}}", from=1-3, to=2-4]
	\arrow[""{name=0, anchor=center, inner sep=0}, "{N_{\bullet}}"', shift right, curve={height=1pt}, from=1-3, to=3-3]
	\arrow["tr", from=2-2, to=1-3]
	\arrow["{N^{\square}_{\bullet}}"', from=2-2, to=3-3]
	\arrow[""{name=1, anchor=center, inner sep=0}, "\Gamma"', shift right=2, curve={height=1pt}, from=3-3, to=1-3]
	\arrow["{\pi_{0}}"', from=3-3, to=2-4]
	\arrow["\dashv"{anchor=center}, draw=none, from=0, to=1]
	\arrow["\sim"{description, pos=0.7}, draw=none, from=0, to=2-2]
\end{tikzcd}\]
where $tr$ takes the triangulation of a cubical object,  $N_{\bullet}\dashv \Gamma$ denotes the Dold-Kan correspondence, and $H_{0}$ and $\pi_{0}$ denote the functors taking $0$th homology functor and connected components respectively. These comparisons between the diagrams $D: \nec\rightarrow \mathcal{W}$ induce comparisons between the nerves they generate, as explained in \S\ref{subsection: Comparison maps}. In particular, we recover a result by Faonte \cite{faonte2015simplicial} and Lurie \cite{lurie2016higher} which shows that the so-called `small' and `big' dg-nerves are equivalent (see Proposition \ref{corollary: comparisons with dg-nerve}).

The list of examples above is certainly not exhaustive. Other nerves from the literature which are likely also generated by a diagram \eqref{diagram: conecklicial object} include the $2$-nerve of $2$-categories \cite{lack2008nerves} and the Street nerve of $\omega$-categories \cite{street1987algebra}. We will investigate these examples in future work.

\subsection{Overview of the paper}

Let us give an overview of the contents of the paper. In Section \ref{section: Preliminaries}, we recall the necessary preliminaries on templicial objects, necklaces and necklace categories from \cite{lowen2024enriched}, as well as the Frobenius structures from \cite{lowen2023frobenius}.

In Section \ref{section: Nerves of enriched categories}, we restrict to the case $\mathcal{V} = \Set$, that is, we consider nerve functors $\mathcal{W}\Cat\rightarrow \SSet$. In \S\ref{subsection: Necklicial nerve functors}, we define \emph{necklicial} nerve functors (Definition \ref{definition: necklicial nerve}) as those which are generated by a strong monoidal diagram $\nec\rightarrow \mathcal{W}$, and characterize them completely in Proposition \ref{proposition: necklicial nerve charac.}. Then in \S\ref{subsection: Necklaces versus cubes}, we compare necklicial nerves to ones generated by a stong monoidal diagram $\square\rightarrow \mathcal{W}$ on the cube category $\square$ with connections, from \cite{legrignou2020cubical}.

We continue in Section \ref{section: Enriched nerves of enriched categories} by generalizing the approach from Section \ref{section: Nerves of enriched categories} to monoidal categories $\mathcal{V}$ different from $\Set$, to obtain nerves of the form $\mathcal{W}\Cat\rightarrow \ts\mathcal{V}$ from a strong monoidal diagram $D: \nec\rightarrow \mathcal{W}$. Each subsection is devoted to proving one main result about such nerves under certain conditions on $D$. In \S\ref{subsection: A general procedure}, we show how they lift the $\Set$-based nerves along the forgetful functor $\tilde{U}: \ts\mathcal{V}\rightarrow \SSet$ (Theorem \ref{theorem: underlying D-nerve is cat. nerve assoc. to D}). In \S\ref{subsection: Explicitation of the left-adjoint} we show when the left-adjoint can be described more explicitly (Theorem \ref{theorem: explicitation of left-adjoint for free temp. obj.}). In \S\ref{subsection: Quasi-categories in V} we recall the quasi-categories in $\mathcal{V}$ from \cite{lowen2024enriched} and show when the nerve of a $\mathcal{W}$-category is a quasi-category in $\mathcal{V}$ (Theorem \ref{theorem: D-nerve is quasi-cat.}). In \S\ref{subsection: Frobenius structures}, we show when the nerve of a $\mathcal{W}$-category has a Frobenius structure (Corollary \ref{corollary: D-nerve has Frob. structure}). Finally, in \S\ref{subsection: Comparison maps}, we show when the natural comparison map between nerves induces a trivial Kan fibration on underlying simplicial sets (Corollary \ref{corollary: trivial fib. between nerves}.2).

The largest part of the paper is contained in Section \ref{section: Examples}, which treats several examples from the literature. For each example, we first show by which diagram $D: \nec\rightarrow \mathcal{W}$ it is generated, and then apply the theorems and corollaries from Section \ref{section: Enriched nerves of enriched categories} to them.

We end the paper with Appendix \ref{section: The generating diagram of the differential graded nerve}, to which we postponed the proof of the generating diagram of the differential graded nerve. 

\subsection{Notations and conventions}\label{subsection: Notations and conventions}

\begin{enumerate}[1.]
\item Throughout the text, we let $(\mathcal{V},\otimes,I)$ be a symmetric monoidal closed category which is cocomplete and finitely complete. Up to natural isomorphism, there is a unique colimit preserving functor $F: \Set\rightarrow \mathcal{V}$ such that $F(\{*\}) = I$. This functor is left-adjoint to the forgetful functor $U = \mathcal{V}(I,-): \mathcal{V}\rightarrow \Set$. Endowing $\Set$ with the cartesian monoidal structure, $F$ is strong monoidal and $U$ is lax monoidal. These notations will remain fixed as well.

\item Let $(\mathcal{W},\otimes_{\mathcal{W}},I_{\mathcal{W}})$ be a $\mathcal{V}$-enriched monoidal category in the sense of \cite{batanin2012centers}. That is, a pseudomonoid in the monoidal $2$-category $\mathcal{V}\Cat$. Assume moreover that
\begin{itemize}
\item $\mathcal{W}$ is tensored over $\mathcal{V}$. We denote the tensoring of $\mathcal{W}$ over $\mathcal{V}$ by $-\cdot -: \mathcal{V}\times \mathcal{W}\rightarrow \mathcal{W}$ and the $\mathcal{V}$-enrichment by $[-,-]: \mathcal{W}\times \mathcal{W}\rightarrow \mathcal{V}$.
\item the underlying category of $\mathcal{W}$ is cocomplete and that $-\otimes_{\mathcal{W}}-$ preserves colimits in each variable.
\item the canonical morphism in the underlying category of $\mathcal{W}$,
\begin{equation}\label{diagram: interchange}
(V_{1}\otimes V_{2})\cdot (W_{1}\otimes_{\mathcal{W}} W_{2})\rightarrow (V_{1}\cdot W_{1})\otimes_{\mathcal{W}} (V_{2}\cdot W_{2})
\end{equation}
is an isomorphism for all $V_{1}, V_{2}\in \mathcal{V}$ and $W_{1}, W_{2}\in \mathcal{W}$.
\end{itemize}

\item To relate enriched categories to templicial objects (see \S\ref{subsection: A general procedure}), it will be more convenient for us to consider the composition $m$ of a $\mathcal{W}$-category to be given by a collection of morphisms in $\mathcal{W}$, for all $A,B,C\in \Ob(\mathcal{C})$:
$$
m_{A,B,C}: \mathcal{C}(A,B)\otimes \mathcal{C}(B,C)\rightarrow \mathcal{C}(A,C)
$$
as opposed to the more conventional $\mathcal{C}(B,C)\otimes \mathcal{C}(A,B)\rightarrow \mathcal{C}(A,C)$. We denote the category of small $\mathcal{W}$-categories and $\mathcal{W}$-functors between them by
$$
\mathcal{W}\Cat
$$
\end{enumerate}


\vspace{0,3cm}
\noindent \emph{Acknowledgement.}
This project has received funding from the European Research Council (ERC) under the European Union’s Horizon 2020 research and innovation programme (grant agreement No. 817762).

I would like to thank Wendy Lowen and Lander Hermans for valuable feedback during the writing of this paper. I am also grateful to Miloslav \v{S}t\v{e}p\'an for pointing out the reference \cite{lack2008nerves} and to Clemens Berger for \cite{baues1980geometry}. Further thanks are extended to Bernhard Keller, Tom Leinster and Michel Van den Bergh for interesting comments and questions on the project.

\section{Preliminaries}\label{section: Preliminaries}

For this first section, we recall some preliminaries from \cite{lowen2024enriched} and \cite{lowen2023frobenius}, most notably the definitions of templicial objects, necklace categories and Frobenius structures. For more details, see loc. cit.

\subsection{Templicial objects}\label{subsection: Templicial objects}

For a set $S$, we consider the category $\mathcal{V}\Quiv_{S} = \mathcal{V}^{S\times S}$ of \emph{$\mathcal{V}$-enriched quivers}. That is, its objects are collections $Q = (Q(a,b))_{a,b\in S}$ with $Q(a,b)\in \mathcal{V}$ and a morphism $f: Q\rightarrow P$ is a collection $(f_{a,b})_{a,b\in S}$ with $f_{a,b}: Q(a,b)\rightarrow P(a,b)$ in $\mathcal{V}$. Note that $\mathcal{V}\Quiv_{S}$ is cocomplete and finitely complete since $\mathcal{V}$ is, and that it carries a monoidal structure $(\otimes_{S},I_{S})$ given by
$$
(Q\otimes_{S} P)(a,b) = \coprod_{c\in S}Q(a,c)\otimes P(c,b)\qquad \text{and}\qquad I_{S}(a,b) =
\begin{cases}
I & \text{if }a = b\\
0 & \text{if }a\neq b
\end{cases}
$$
for all $Q,P\in \mathcal{V}\Quiv_{S}$ and $a,b\in S$. Note that a monoid in $\mathcal{V}\Quiv_{S}$ is precisely a $\mathcal{V}$-enriched category with object set $S$.

Given a map of sets $f: S\rightarrow T$, we have an adjunction $f_{!}: \mathcal{V}\Quiv_{S}\leftrightarrows \mathcal{V}\Quiv_{T}: f^{*}$ where $f^{*}(Q)(a,b) = Q(f(a),f(b))$ for all $Q\in \mathcal{V}\Quiv_{T}$ and $a,b\in S$. Moreover, $f_{!}$ and $f^{*}$ have canonical colax and lax monoidal structures respectively.

Further, we let $\fint$ denote the category of \emph{finite intervals}, which is the subcategory of the simplex category $\simp$ containing all objects $[n] = \{0 < 1 < \dots  < n\}$ for integers $n\geq 0$ and all order morphisms $f: [m]\rightarrow [n]$ such that $f(0) = 0$ and $f(m) = n$. It carries a monoidal structure $(+,[0])$ given by $[m] + [n] = [m + n]$ on objects.

\begin{Def}[Definition 2.3, \cite{lowen2024enriched}]\label{definition: temp. obj.}
A \emph{tensor-simplicial} or \emph{templicial object} in $\mathcal{V}$ is a pair $(X,S)$ with $S$ a set and 
$$
X: \fint^{op}\rightarrow \mathcal{V}\Quiv_{S}
$$
a strongly unital, colax monoidal functor. A \emph{templicial morphism} $(X,S)\rightarrow (Y,T)$ is a pair $(\alpha,f)$ with $f: S\rightarrow T$ a map of sets and $\alpha: f_{!}X\rightarrow Y$ a monoidal natural transformation between colax monoidal functors $\fint^{op}\rightarrow \mathcal{V}\Quiv_{T}$. The composition of templicial morphisms is defined in the obvious way, and we denote the category of templicial objects in $\mathcal{V}$ and templicial morphisms between them by
$$
\ts\mathcal{V}
$$
\end{Def}

Given a templcial object $(X,S)$, we denote $X_{n} = X([n])\in \mathcal{V}\Quiv_{S}$ for all $n\geq 0$. The category $\fint$ contains all inner coface maps $\delta_{j}: [n-1]\hookrightarrow [n]$ for $0 < j < n$ and all codegeneracy maps $\sigma_{i}: [n+1]\twoheadrightarrow [n]$ for $0\leq i\leq n$. We denote the induced \emph{inner face morphisms} and \emph{degeneracy morphisms} in $\mathcal{V}\Quiv_{S}$ by
$$
d_{j} = X(\delta_{j}): X_{n}\rightarrow X_{n-1}\quad \text{and}\quad s_{i} = X(\sigma_{i}): X_{n}\rightarrow X_{n+1}
$$
These satisfy the usual simplicial identities. Moreover, the colax monoidal structure equips $X$ with \emph{comultiplications} and a \emph{counit} in $\mathcal{V}\Quiv_{S}$, which we'll denote by
$$
\mu_{k,l}: X_{k+l}\rightarrow X_{k}\otimes_{S} X_{l}\quad \text{and}\quad \epsilon: X_{0}\xrightarrow{\simeq} I_{S}
$$
for all $k,l\geq 0$ and where $\epsilon$ is a quiver isomorphism. The comultiplications satisfy coassociativity and counitality conditions with respect to $\epsilon$, as well as compatibility conditions with respect to $d_{j}$ and $s_{i}$.

\begin{Prop}[Proposition 2.8, \cite{lowen2024enriched} and Corollary 4.9, \cite{mertens2024discrete}]\label{proposition: temp. obj. results}
The following statements are true.
\begin{enumerate}[1.]
\item The category $\ts\mathcal{V}$ is cocomplete.
\item There is an adjunction $\tilde{F}: \SSet\leftrightarrows \ts\mathcal{V}: \tilde{U}$ where $\tilde{F}$ is induced by applying the free functor $F: \Set\rightarrow \mathcal{V}$ levelwise. This adjunction is an equivalence when $\mathcal{V} = \Set$.
\item If $\mathcal{V}$ is cartesian monoidal and satisfies (DISJ) of \cite[Condition 10.7.1]{simpson2012homotopy}, then $\ts\mathcal{V}$ is equivalent to the category $\mathbf{PC}(\mathcal{V})$ of $\mathcal{V}$-enriched precategories.
\end{enumerate}
\end{Prop}

\begin{Def}[Definition 2.11, \cite{lowen2023frobenius}]\label{definition: Frobenius temp. obj.}
A \emph{Frobenius structure} on a templicial object $(X,S)$ is a collection of quiver morphisms
$$
Z = \left(Z^{p,q}: X_{p}\otimes_{S} X_{q}\rightarrow X_{p+q}\right)_{p,q\geq 0}
$$
such that $Z^{p,q}$ is natural in $[p],[q]\in \fint$, the maps $Z^{p,q}$ are associative with unit $\epsilon^{-1}$:
\begin{equation}\label{equation: Frobenius unitality}
Z^{0,n}(\epsilon^{-1}\otimes \id_{X_{n}}) = Z^{n,0}(\id_{X_{n}}\otimes \epsilon^{-1}) = \id_{X_{n}}
\end{equation}
and the following \emph{Frobenius identities} are satisfied for all $k,l,p,q\geq 0$ with $k + l = p + q$:
\begin{equation}\label{equation: Frobenius identities}
\mu_{k,l}Z^{p,q} =
\begin{cases}
(Z^{p,k-p}\otimes \id_{X_{l}})(\id_{X_{p}}\otimes \mu_{k-p,l}) & \text{if }p\leq k\\
(\id_{X_{k}}\otimes Z^{p-k,q})(\mu_{k,p-k}\otimes \id_{X_{q}}) & \text{if }p\geq k
\end{cases}
\end{equation}
Note that in particular, $\mu_{p,q}Z^{p,q} = \id_{X_{p}\otimes_{S} X_{q}}$. A templicial object equipped with a Frobenius structure is called a \emph{Frobenius templicial object}. A \emph{Frobenius templicial morphism} is a templicial morphism which is compatible with the Frobenius structures. We denote the category of Frobenius templicial objects and Frobenius templicial morphisms by
$$
\Fs\mathcal{V}
$$
\end{Def}

\subsection{Necklaces and necklace categories}\label{subsection: Necklaces and necklace categories}

Let $\SSet_{*,*} = \SSet_{\partial\Delta^{1}/}$ denote the category of bipointed simplicial sets. We denote a bipointed simplicial set $K$ with distinguished vertices $a$ and $b$ by $K_{a,b}$. We always equip the standard simplices $\Delta^{n}$ for $n\geq 0$ with distinguished points $0$ and $n$. Given bipointed simplicial sets $K_{a,b}$ and $L_{c,d}$, we denote their \emph{wedge product} by $K\vee L$. It is the simplicial set obtained by glueing the vertices $b$ and $c$, and we equip it with the distinguished vertices $a$ and $d$.

\begin{Def}[Definition 2.3, Chapter III, \cite{baues1980geometry} and \S 3, \cite{dugger2011rigidification}]\label{definition: necklace}
A \emph{necklace} $T$ is
$$
T = \Delta^{n_{1}}\vee ...\vee \Delta^{n_{k}}\in \SSet_{*,*}
$$
for some $k\geq 0$ and $n_{1},...,n_{k} > 0$ (if $k = 0$, then $T = \Delta^{0}$). We refer to the standard simplices $\Delta^{n_{1}},...,\Delta^{n_{k}}$ as the \emph{beads} of $T$. The distinguished vertices in every bead are called the \emph{joints} of $T$. We let $\nec$ denote the full subcategory of $\SSet_{*,*}$ spanned by all necklaces. Note that $(\nec,\vee,\Delta^{0})$ is a monoidal category.
\end{Def}

It was shown in \cite[Proposition 3.4]{lowen2024enriched} that the category $\nec$ is equivalent to the category of pairs $(T,p)$ with $p\geq 0$ and $\{0 < p\}\subseteq T\subseteq [p]$, where a morphism $(T,p)\rightarrow (U,q)$ is given by a morphism $f: [p]\rightarrow [q]$ in $\fint$ such that $U\subseteq f(T)$. We will use these two presentations of $\nec$ interchangeably.

We further recall some terminology from \cite{borges2024necklaces} and \cite{lowen2024enriched}. A necklace map $f: (T,p)\rightarrow (U,q)$ is called \emph{active} if $U = f(T)$ and \emph{inert} if $f = \id_{[p]}$. It is easy to see that the subcategories of active and inert necklace maps are monoidal and form an orthogonal factorization system on $\nec$ in the sense of \cite{bousfield1976constructions}. The active maps are generated as a monoidal subcategory by the maps $\delta_{j}: \Delta^{n-1}\hookrightarrow \Delta^{n}$ and $\sigma_{i}: \Delta^{n+1}\twoheadrightarrow \Delta^{n}$ for $0 < j < n$ and $\leq i\leq n$. Similarly, the inert maps are generated as a monoidal subcategory by the maps $\nu_{k,l}: \Delta^{k}\vee \Delta^{l}\hookrightarrow \Delta^{k+l}$ for $k,l > 0$. Hence, $\nec$ is generated as a monoidal category by $\delta_{j}$, $\sigma_{i}$ and $\nu_{k,l}$.

A necklace map $f: (T,p)\rightarrow (U,q)$ is called \emph{surjective} or \emph{injective} if the underlying morphism $[p]\rightarrow [q]$ in $\fint$ is so. Further, we call $f$ \emph{spine collapsing} if it is the wedge product of identities and the map $\Delta^{1}\twoheadrightarrow \Delta^{0}$. Finally, for any necklace $(T,p)$ and $j\in [p]\setminus T$, we denote by
$$
\delta_{j}: (\delta^{-1}_{j}(T),p-1)\hookrightarrow (T,p)\qquad \text{and}\quad \nu_{j,p-j}: (T\cup \{j\},p)\hookrightarrow (T,p)
$$
the active injective necklace map given by $\delta_{j}$ in $\fint$ and the unique inert necklace map.

\begin{Def}[Definition 3.8, \cite{lowen2024enriched}]\label{definition: necklace category}
Consider the functor category $\mathcal{V}^{\nec^{op}}$ with the monoidal structure $(\otimes_{Day},\underline{I})$ of Day convolution \cite{day1970closed}. A \emph{necklace category} is a category enriched in $\mathcal{V}^{\nec^{op}}$. We denote the category of small necklace categories and $\mathcal{V}^{\nec^{op}}$-enriched functors between them by
$$
\mathcal{V}\Cat_{\nec}
$$
\end{Def}

Given a templicial object $(X,S)$ with vertices $a,b\in S$, we obtain a functor $X_{\bullet}(a,b): \nec^{op}\rightarrow \mathcal{V}$ where, for all necklaces $T = \Delta^{n_{1}}\vee \dots \Delta^{n_{k}}$, we have
$$
X_{T}(a,b) = (X_{n_{1}}\otimes_{S}\dots \otimes_{S} X_{n_{k}})(a,b)\simeq \coprod_{c_{1},\dots, c_{k-1}\in S}X_{n_{1}}(a,c_{1})\otimes\dots \otimes X_{n_{k}}(c_{k-1},b)
$$
More specifically, $X_{\bullet}(a,b)$ sends $\delta_{j}: \Delta^{n-1}\hookrightarrow \Delta^{n}$, $\sigma_{i}: \Delta^{n+1}\twoheadrightarrow \Delta^{n}$ and $\nu_{k,l}: \Delta^{k}\vee \Delta^{l}\hookrightarrow \Delta^{k+l}$ to $d_{j}$, $s_{i}$ and $\mu_{k,l}$ respectively. With compositions induced by the canonical morphisms $X_{T}(a,b)\otimes X_{U}(b,c)\rightarrow X_{T\vee U}(a,c)$, we obtain a necklace category $X^{nec}$ whose object set is $S$ and hom-objects are $X_{\bullet}(a,b)$. This construction extends to a functor $(-)^{nec}: \ts\mathcal{V}\rightarrow \mathcal{V}\Cat_{\nec}$ (see \cite[Construction 3.9]{lowen2024enriched}).

The functor $(-)^{nec}$ was independently constructed by Minichiello, Rivera and Zeinalian in \cite{minichiello2023categorical} for $\mathcal{V} = \Set$. In this case, we'll denote $\mathcal{V}\Cat_{\nec}$ by $\Cat_{\nec}$.

\begin{Thm}[Theorem 3.12, \cite{lowen2024enriched}]\label{theorem: nec-temp adjunction}
The functor $(-)^{nec}$ is fully faithful and has a right-adjoint
\begin{equation}\label{diagram: nec-temp adjunction}
(-)^{nec}: \ts\mathcal{V}\leftrightarrows \mathcal{V}\Cat_{\nec}: (-)^{temp}
\end{equation}
\end{Thm}

One can describe the functor $(-)^{temp}$ by induction on the dimension of the simplices, see \cite[Construction 3.11]{lowen2024enriched}. But in case $\mathcal{V} = \Set$, it can be described more easily.

\begin{Ex}\label{example: temp construction in Set-case}
For a necklace category $\mathcal{C}\in \Cat_{\nec}$, an $n$-simplex of $\mathcal{C}^{temp}\in \SSet$ is a collection
$$
\left((A_{i})_{i=0}^{n}, (\alpha_{i,j})_{0\leq i < j\leq n}\right)
$$
with $A_{i}\in \Ob(\mathcal{C})$ and $\alpha_{i,j}\in \mathcal{C}_{\{0 < j-i\}}(A_{i},A_{j})$, such that for all $0 < k < j$, we have $\mathcal{C}(\nu_{k-i,j-k})(\alpha_{i,j}) = m_{\mathcal{C}}(\alpha_{i,k}, \alpha_{k,j})$.
\end{Ex}

\section{Nerves of enriched categories}\label{section: Nerves of enriched categories}

In this section we focus on nerves of the form $\mathcal{W}\Cat\rightarrow \SSet$. That is, we restrict to the case where $\mathcal{V} = \Set$ with the cartesian monoidal structure. Note that in this case, $\mathcal{W}$ is simply assumed to be a cocomplete monoidal category such that $-\otimes_{\mathcal{W}}-$ preserves colimits in each variable. Indeed, $\mathcal{W}$ is automatically tensored over $\Set$ via $S\cdot W = \coprod_{x\in S}W$ for all $S\in \Set$ and $W\in \mathcal{W}$, and the canonical morphism \eqref{diagram: interchange} is always an isomorphism.

\subsection{Necklicial nerve functors}\label{subsection: Necklicial nerve functors}

We introduce \emph{necklicial} nerve functors as those right-adjoints $\mathcal{W}\Cat\rightarrow \SSet$ arising from a strong monoidal functor $D: \nec\rightarrow \mathcal{W}$. Certainly not all possible right-adjoints $\mathcal{W}\Cat\rightarrow \SSet$ are necklicial nerves, but many examples of interest are, as we will see in Section \ref{section: Examples}. Moreover, restricting to necklicial nerves will allow us to obtain the results of Section \ref{section: Enriched nerves of enriched categories}, where we use the diagram $D$ to lift the induced nerve to templicial objects, describe its left-adjoint more explicitly and detect when the nerve is a quasi-category.

\begin{Con}\label{construction: cosimplicial object generated by necklicial diagram}

We construct a functor
$$
\Phi: \StrMon(\nec,\mathcal{W})\rightarrow \Fun(\simp,\mathcal{W}\Cat)
$$
from the category of strong monoidal functors $\nec\rightarrow \mathcal{W}$ and monoidal natural transformations between them, to the category of functors $\simp\rightarrow \mathcal{W}\Cat$ and natural transformations between them.

Given a strong monoidal functor $D: \nec\rightarrow \mathcal{W}$, we define $\Phi(D): \simp\rightarrow \mathcal{W}\Cat$ as follows. For every integer $n\geq 0$, $\Phi(D)^{n}$ is the $\mathcal{W}$-category with object set $[n]$ and for all $i,j\in [n]$:
$$
\Phi(D)^{n}(i,j) =
\begin{cases}
D(\Delta^{j-i}) & \text{if }i\leq j\\
0 & \text{if }i > j
\end{cases}
$$
Given $i\leq k\leq j$ in $[n]$, the composition of $\Phi(D)^{n}$ is defined by the strong monoidal structure of $D$:
$$
\Phi(D)^{n}(i,k)\otimes \Phi(D)^{n}(k,j)\simeq D(\Delta^{k-i}\vee \Delta^{j-k})\rightarrow D(\Delta^{j-i}) = \Phi(D)^{n}(i,j)
$$
For $i\in [n]$, the identity in $\Phi(D)^{n}(i,i)$ is given by the unit $I\simeq D(\Delta^{0})$.

Further, if $f: [m]\rightarrow [n]$ is a morphism in $\simp$, we define a $\mathcal{W}$-functor $\Phi(D)(f): \Phi(D)^{m}\rightarrow \Phi(D)^{n}$ which is given on objects by the map $f$. For all $i\leq j$ in $[m]$, $f$ induces a morphism in $\fint$:
$$
f_{i,j}: [j-i]\rightarrow [f(j) - f(i)]: k\mapsto f(k + i) - f(i)
$$
which we identify with a necklace map $f_{i,j}: \Delta^{j-i}\rightarrow \Delta^{f(j)-f(i)}$. Then
$$
\Phi(D)(f)_{i,j}: \Phi(D)^{m}(i,j) = D(\Delta^{j-i})\xrightarrow{D(f_{i,j})} D(\Delta^{f(j)-f(i)}) = \Phi(D)^{n}(f(i),f(j))
$$

Finally, given a monoidal natural transformation $\alpha: D\rightarrow D'$ in $\StrMon(\nec,\mathcal{W})$ we define a natural transformation $\Phi(\alpha): \Phi(D)\rightarrow \Phi(D')$ as follows. For every integer $n\geq 0$, $\Phi(\alpha)^{n}$ is the $\mathcal{W}$-enriched functor which given by the identity on objects and for all $i,j\in [n]$:
$$
\Phi(\alpha)^{n}_{i,j} =
\begin{cases}
\alpha_{\Delta^{j-i}} & \text{if }i\leq j\\
\id_{0} & \text{if }i > j
\end{cases}
$$

It immediately follows from the definitions that this produces a well-defined functor $\Phi: \StrMon(\nec,\mathcal{W})\rightarrow \Fun(\simp,\mathcal{W}\Cat)$.
\end{Con}

\begin{Def}\label{definition: necklicial nerve}
Let $\mathbb{D}: \simp\rightarrow \mathcal{W}\Cat$ be a diagram and $N^{\mathbb{D}}: \mathcal{W}\Cat\rightarrow \SSet$ its associated right-adjoint functor under equivalence \eqref{equation: classical equivalence}. We call $N^{\mathbb{D}}$ a \emph{necklicial nerve} if $\mathbb{D}\simeq \Phi(D)$ for some strong monoidal functor $D: \nec\rightarrow \mathcal{W}$. We denote by
$$
\mathrm{Nerve}_{\nec}(\mathcal{W}\Cat)
$$
the full subcategory of $\Fun_{RAdj}(\mathcal{W}\Cat,\SSet)$ spanned by all necklicial nerves.
\end{Def}

The remainder of this subsection is devoted to characterising necklicial nerves.

\begin{Def}
Let us call a morphism $\delta: [m]\rightarrow [n]$ of $\simp$ \emph{inert} if it is given by $\delta(i) = \delta(0) + i$ for all $i\in [m]$. It is easy to see that a morphism in $\simp$ is inert if and only if it is a composition of outer coface maps $\delta_{0}, \delta_{n}: [n-1]\rightarrow [n]$.
\end{Def}

\begin{Prop}\label{proposition: embedding of necklace diagrams into necklicial functors}
The functor $\Phi: \StrMon(\nec,\mathcal{W})\rightarrow \Fun(\simp,\mathcal{W}\Cat)$ of Construction \ref{construction: cosimplicial object generated by necklicial diagram} is fully faithful. Hence, we have  an equivalence of categories
$$
\mathrm{Nerve}_{\nec}(\mathcal{W}\Cat)\simeq \StrMon(\nec,\mathcal{W})^{op}
$$
\end{Prop}
\begin{proof}
Let $D,D': \nec\rightarrow \mathcal{W}$ be strong monoidal functors and $\beta: \Phi(D)\rightarrow \Phi(D')$ a natural transformation. Note that since $\Phi(D)^{0} = \Phi(D')^{0}$ has a single object with hom-object given by $I$, the naturality of $\beta$ implies that $\beta^{n}$ is given by the identity on objects for all $n\geq 0$. Further, since for any inert map $\delta: [m]\rightarrow [n]$ in $\simp$, the induced morphism $\Phi(D)(\delta)_{0,m}$ is the identity, we have that $\beta^{n}_{i,j} = \beta^{j-i}_{0,j-i}$ for all $i\leq j$ in $[n]$. For a necklace $T = \{0 = t_{0} < t_{1} <  ... < t_{k} = p\}$, we have an isomorphism
$$
\mu: \Phi(D)^{p}(0,t_{1})\otimes ...\otimes \Phi(D)^{p}(t_{k-1},p)\xrightarrow{\sim} D(T)
$$
by the strong monoidality of $D$, and a similar isomorphism $\mu'$ for $D'$. Then define $\alpha_{T}: D(T)\rightarrow D'(T)$ as $\alpha_{T} = \mu'(\beta^{p}_{0,t_{1}} \otimes ...\otimes \beta^{p}_{t_{k-1},p})\mu^{-1}$. Then $\alpha$ is compatible with active necklace maps by the naturality of $\mu$, $\mu'$ and $\beta$. Further $\alpha$ is compatible with inert necklace maps by the naturality and associativity of $\mu$ and $\mu'$, and by the functoriality of $\beta^{p}$. Thus $\alpha$ is a natural transformation between functors $\nec\rightarrow \mathcal{W}$. Moreover, $\alpha$ is monoidal by the associativity of $\mu$ and $\mu'$. Finally, by definition we have $\Phi(\alpha)^{p}_{i,j} = \alpha_{\{0 < j-i\}} = \beta^{j-i}_{0,j-i} = \beta^{p}_{i,j}$ for all $i\leq j$ in $[p]$. So $\Phi(\alpha) = \beta$ and $\alpha$ is clearly unique with this property.

The equivalence then follows from \eqref{equation: classical equivalence}.
\end{proof}

\begin{Prop}\label{proposition: necklicial nerve charac.}
The essential image of the functor $\Phi$ of Construction \ref{construction: cosimplicial object generated by necklicial diagram} consists of all diagrams $\mathbb{D}: \simp\rightarrow \mathcal{W}\Cat$ satisfying the following properties.
\begin{enumerate}[1.]
\item The diagram of functors
\[\begin{tikzcd}
	\simp && {\mathcal{W}\Cat} \\
	& \Set
	\arrow["{\mathbb{D}}", from=1-1, to=1-3]
	\arrow["\Ob", from=1-3, to=2-2]
	\arrow["\Ob"', from=1-1, to=2-2]
\end{tikzcd}\]
commutes up to natural isomorphism.
\item For any integers $0\leq i < j\leq n$, $\mathbb{D}^{n}(j,i)$ is an initial object of $\mathcal{W}$.
\item The unit morphism $I\rightarrow \mathbb{D}^{0}(0,0)$ of $\mathbb{D}^{0}$ is an isomorphism.
\item For any inert morphism $\delta: [m]\rightarrow [n]$ in $\simp$, the induced morphism in $\mathcal{W}$
$$
\mathbb{D}(\delta)_{0,m}: \mathbb{D}^{m}(0,m)\rightarrow \mathbb{D}^{n}(\delta(0),\delta(m))
$$
is an isomorphism.
\end{enumerate}
Hence, for a diagram $\mathbb{D}: \simp\rightarrow \mathcal{W}\Cat$, the associated right-adjoint functor $N^{\mathbb{D}}$ under \eqref{equation: classical equivalence} is a necklicial nerve if and only if $\mathbb{D}$ satisfies properties $1$-$4$.
\end{Prop}
\begin{proof}
Clearly, for any strong monoidal $D: \nec\rightarrow \mathcal{W}$, $\Phi(D)$ satisfies $1$-$4$ and these properties are invariant under isomorphism in $\Fun(\simp,\mathcal{W}\Cat)$. Conversely, it is easy to see that any functor $\mathbb{D}: \simp\rightarrow \mathcal{W}\Cat$ satisfying properties $1$-$4$ is isomorphic to a functor $\tilde{\mathbb{D}}$ for which these properties hold strictly. That is, $\Ob\circ \tilde{\mathbb{D}}$ is precisely the forgetful functor $\simp\rightarrow \Set$, all $\tilde{\mathbb{D}}^{n}(j,i)$ with $i < j$ are equal to the same initial object $0$, $\tilde{\mathbb{D}}^{0}(0,0) = I$ and for any inert map $\delta: [m]\rightarrow [n]$ in $\simp$, the induced map $\tilde{\mathbb{D}}(\delta)_{0,m}$ is the identity. Thus we assume that properties $1$-$4$ hold strictly for $\mathbb{D}$. 
Given a necklace $T = \{0 = t_{0} < t_{1} < \dots < t_{k} = p\}$, define
$$
D(T) = \mathbb{D}^{p}(0,t_{1})\otimes \dots\otimes \mathbb{D}^{p}(t_{k-1},p)
$$
In particular, that is $D(\{0\}) = I$. Let $f: (T,p)\rightarrow (U,q)$ be a necklace map
\begin{itemize}
\item If $f$ is active, let $\mathbb{D}(f): \mathbb{D}^{p}\rightarrow \mathbb{D}^{q}$ denote the functor induced by the underlying morphism $[p]\rightarrow [q]$ in $\fint\subseteq \simp$. Then define $D(f): D(T)\rightarrow D(U)$ as
$$
D(f) = \mathbb{D}(f)_{0,t_{1}}\otimes \dots\otimes \mathbb{D}(f)_{t_{k-1},p}
$$
where we used that $D(f(T)) = D(U)$, which follows from property $3$ and the fact that $f(T) = U$ as subsets.
\item If $f$ is inert, define $D(f): D(T)\rightarrow D(U)$ as
$$
D(f) = m_{T_{1}}\otimes \dots\otimes m_{T_{l}}
$$
where $m_{T'}: \mathbb{D}^{p'}(0,t'_{1})\otimes \dots\otimes \mathbb{D}^{p'}(t'_{k'-1},t'_{p'})\rightarrow \mathbb{D}^{p'}(0,p')$ denotes the composition in $\mathbb{D}^{p'}$ for a given necklace $(T',p')$. Writing $U = \{0 = u_{0} < u_{1} < \dots < u_{l} = p\}\subseteq T$, the $(T_{i},u_{i}-u_{i-1})$ are the unique necklaces such that $T = T_{1}\vee \dots\vee T_{l}$.
\end{itemize}
Then this defines a functor $D: \nec\rightarrow \mathcal{W}$. Indeed, $D$ is functorial on active necklace maps by the functoriality of $\mathbb{D}$, $D$ is functorial on inert maps by the associativity of $m$, and then $D$ is functorial on all necklace maps by the naturality of $m$. Finally note that $\mathbb{D}(T\vee U) = \mathbb{D}(T)\otimes \mathbb{D}(U)$ by property $4$ for any two necklaces $T$ and $U$. Thus we have a strong monoidal structure on $D$. It then follows straightforwardly from the definitions, and properties $1$ and $2$, that $\Phi(D) = \mathbb{D}$.
\end{proof}

\subsection{Necklaces versus cubes}\label{subsection: Necklaces versus cubes}

We denote by $\square$ the category of \emph{cubes with connections}. Its objects are all categories $[1]^{n} = \left\lbrace(\epsilon_{1},\dots,\epsilon_{n})\mid \epsilon_{i}\in \{0,1\}\right\rbrace$ with $n\geq 0$ an integer. Then $\square$ is the subcategory of $\Cat$ generated by the following morphisms
\begin{align*}
\delta^{\epsilon}_{i}&: [1]^{n-1}\rightarrow [1]^{n}: (\epsilon_{1},\dots,\epsilon_{n-1})\mapsto (\epsilon_{1},\dots,\epsilon_{i-1},\epsilon,\epsilon_{i},\dots,\epsilon_{n}) &&\text{for }1\leq i\leq n, \epsilon\in \{0,1\}\\
\sigma_{i}&: [1]^{n+1}\rightarrow [1]^{n}: (\epsilon_{1},\dots,\epsilon_{n+1})\mapsto (\epsilon_{1},\dots,\epsilon_{i-1},\epsilon_{i+1},\dots,\epsilon_{n+1}) &&\text{for }1\leq i\leq n+1\\
\gamma_{i}&: [1]^{n+1}\rightarrow [1]^{n}: (\epsilon_{1},\dots,\epsilon_{n+1})\mapsto (\epsilon_{1},\dots,\max(\epsilon_{i},\epsilon_{i+1}),\dots,\epsilon_{n}) &&\text{for }1\leq i\leq n-1
\end{align*}
These satisfy certain idenities. For more details we refer to the literature (see \cite{brown1981algebra} for example). We denote by $\CSet = \Set^{\square^{op}}$ the category of cubical sets. For any $n\geq 0$, we denote $\square^{n} = \square(-,[1]^{n})$ for the standard $n$-cube.

The category $\square$ is also monoidal with monoidal product induced by the cartesian product on $\Cat$:
$$
[1]^{m}\otimes [1]^{n} = [1]^{m+n}
$$
for all $m,n\geq 0$. However, this monoidal structure is not symmetric. Equipped with the Day convolution \cite{day1970closed}, the category $(\CSet,\otimes_{Day},\square^{0})$ is cocomplete biclosed (non-symmetric) monoidal, and we have $\square^{m}\otimes \square^{n}\simeq \square^{m+n}$ for all $m,n\geq 0$.

Let $\Cat_{\square}$ denote the category of small \emph{cubical categories}, i.e. categories enriched in $\CSet$. Then Le Grignou defines a nerve functor $\Cat_{\square}\rightarrow \SSet$ as follows.

\begin{Def}[Definitions 27 and 28 and Proposition 21, \cite{legrignou2020cubical}]\label{definition: cubical nerve}
We define a diagram
$$
W_{c}: \simp\rightarrow \Cat_{\square}.
$$
For any $n\geq 0$, the cubical category $W^{n}_{c}$ has object set $[n] = \{0,\dots,n\}$ and its hom-objects are given by, for all  $i,j\in [n]$:
$$
W^{n}_{c}(i,j) =
\begin{cases}
\square^{j-i-1} & \text{if }i < j\\
\square^{0} & \text{if }i = j\\
\emptyset & \text{if }i > j
\end{cases}
$$
The identities are given by the unique vertex of $W^{n}_{c}(i,i)$ for all $i\in [n]$, and the composition is given by the following map, for all $i < j < k$ in $[n]$:
$$
W^{n}_{c}(i,j)\otimes W^{n}_{c}(j,k)\simeq \square^{k-i-2}\xrightarrow{(\delta^{1}_{j-i})^{*}} \square^{k-i-1} = W^{n}_{c}(i,k)
$$

Given a morphism $f: [m]\rightarrow [n]$ in $\simp$, the induced cubical functor $W^{m}_{c}\rightarrow W^{n}_{c}$ is given on objects by $f$, and the map of cubical sets $W_{c}(f)_{i,j}: W^{m}_{c}(i,j)\rightarrow W^{n}_{c}(f(i),f(j))$ for $i < j$ in $[m]$ is defined for coface and codegeneracy maps seperately:
$$
W_{c}(\delta_{k})_{i,j} =
\begin{cases}
\delta^{0}_{k-i} & \text{if }i < k\leq j\\
\id_{\square^{j-i-1}} & \text{otherwise}
\end{cases}
\quad \text{and}\quad
W_{c}(\sigma_{k})_{i,j} =
\begin{cases}
\sigma_{1} & \text{if }k = i\\
\gamma_{k-i} & \text{if }i < k < j-1\\
\sigma_{j-i-1} & \text{if } k = j-1\\
\id_{\square^{j-i-1}} & \text{otherwise}
\end{cases}
$$
We call the nerve associated to $W^{n}_{c}$ under \eqref{equation: classical equivalence} the \emph{cubical nerve} $N^{\cub}: \Cat_{\square}\rightarrow \SSet$.
\end{Def}

It is  easy to see that $W^{n}_{c}$ satisfies properties $1$-$4$ of Proposition \ref{proposition: necklicial nerve charac.} and so the cubical nerve $N^{cub}$ is necklicial. To identify the associated diagram $\nec\rightarrow \CSet$, we consider a comparison map between necklaces and cubes which was first considered by Rivera and Zeinalian.

\begin{Def}[\S 3.6, \cite{dugger2011rigidification} and \S 4, \cite{rivera2018cubical}]\label{definition: dimension of a necklace}
\begin{enumerate}[1.]
\item We define a functor
\begin{equation}\label{diagram: partition functor}
\nec\rightarrow \Cat: T\mapsto \mathcal{P}_{T} = \left\lbrace U\subseteq [p]\mid T\subseteq U\right\rbrace
\end{equation}
where we consider $\mathcal{P}_{T}$ as a poset ordered by inclusion. Note that we may identify $\mathcal{P}_{T}$ with $(\nec_{inert})^{op}/T$ where $\nec_{inert}$ is the subcategory of $\nec$ containing all inert necklace maps. For any necklace map $f: (T,p)\rightarrow (U,q)$, the induced functor is defined as
$$
f^{*}: \mathcal{P}_{T}\rightarrow \mathcal{P}_{U}: V\mapsto f(V)
$$

\item Given a necklace $(T,p)$, we write $T^{c} = [p]\setminus T$ and define the \emph{dimension} of $T$ as
$$
\dim(T) = \vert T^{c}\vert
$$
It is the number of vertices of $T$ minus the number of joints of $T$. Let us write $T^{c} = \{i_{1} < \dots < i_{n}\}$ where $n = \dim(T)$. We may identify $[1]^{n}$ with the power set $2^{T^{c}}$ of $T^{c}$ and thus the map $\mathcal{P}_{T}\rightarrow 2^{T^{c}}: U\mapsto U\setminus T$ induces an isomorphism
\begin{equation}\label{diagram: equiv. descriptions of cubes}
\mathcal{P}_{T}\xrightarrow{\sim} [1]^{n}: U\mapsto (\epsilon_{1},\dots,\epsilon_{n})\qquad \text{with }\epsilon_{j} =
\begin{cases}
0 & \text{if }i_{j}\not\in U\\
1 & \text{if }i_{j}\in U
\end{cases}
\end{equation}
Then \eqref{diagram: partition functor} factors through the inclusion $\square\subseteq \Cat$ as a strong monoidal functor
\begin{equation}\label{diagram: dimension functor}
\dim: \nec\rightarrow \square: T\mapsto [1]^{\dim(T)}.
\end{equation}
which is given on monoidal generating morphisms by
$$
\dim(\delta_{k}) = \delta^{0}_{k},\,\, \dim(\nu_{k,p-k}) = \delta^{1}_{k},\,\, \dim(\sigma_{k}) =
\begin{cases}
\sigma_{1} & \text{if }k = 0\\
\gamma_{k} & \text{if }0 < k < n\\
\sigma_{n} & \text{if }k = n
\end{cases},\,\,
\dim(\sigma_{0}) = \id_{[0]}
$$
where $\delta_{k}: \Delta^{p-1}\rightarrow \Delta^{p}$, $\nu_{k,p-k}: \Delta^{k}\vee \Delta^{p-k}\rightarrow \Delta^{p}$, $\sigma_{k}: \Delta^{p+1}\rightarrow \Delta^{p}$ with $p > 0$, and $\sigma_{0}: \Delta^{1}\rightarrow \Delta^{0}$.
\end{enumerate}
\end{Def}

The spine collapsing maps are precisely the active surjective ones which preserve the dimension. In this sense, necklaces contain slightly more information than cubes.

\begin{Lem}\label{lemma: spine collapsing equiv.}
Let $\sigma: U\rightarrow U'$ be an active surjective necklace map. The following are equivalent:
\begin{enumerate}
\item $\sigma$ is spine collapsing,
\item $\sigma$ induces a bijection $\sigma\vert_{U^{c}}: U^{c}\xrightarrow{\sim} (U')^{c}$,
\item $\dim(U) = \dim(U')$.
\end{enumerate}
\end{Lem}
\begin{proof}
All three statements are invariant under taking wedge products and thus we may assume that $\sigma$ is a surjective necklace map $\Delta^{n}\rightarrow \Delta^{m}$ with $n\geq m\geq 0$. Then note that all three statements are true if $\sigma$ is the identity or $\Delta^{1}\twoheadrightarrow \Delta^{0}$, and false otherwise.
\end{proof}

\begin{Prop}\label{proposition: cubical nerve is necklicial}
The cubical nerve functor $N^{\cub}: \Cat_{\square}\rightarrow \SSet$ of \cite{legrignou2020cubical} is necklicial with associated diagram given by the composite
$$
\nec\xrightarrow{\dim} \square\xrightarrow{\yo} \CSet
$$
where $\yo$ is the Yoneda embedding.
\end{Prop}
\begin{proof}
Set $D = \yo\circ \dim$. By the definition of $N^{\cub}$ it suffices to show that $\Phi(D)$ of Construction \ref{construction: cosimplicial object generated by necklicial diagram} coincides with $W_{c}$. Note that for all $n\geq 0$, both $\Phi(D)^{n}$ and $W^{n}_{c}$ have $[n]$ as their set of objects. Then for all $i < j$ in $[n]$ we have
$$
\Phi(D)^{n}(i,j) = \square^{\dim\Delta^{j-i}} = \square^{j-i-1} = W^{n}_{c}(i,j)
$$
while $\Phi(D)^{n}(j,i) = 0 = W^{n}_{c}(j,i)$ and $\Phi(D)^{n}(i,i) = * = W^{n}_{c}(i,i)$. Moreover, the composition of $\Phi(D)^{n}$ is induced by the inert map $\nu_{j-i,k-j}: \Delta^{j-i}\vee \Delta^{k-j}\hookrightarrow \Delta^{k-i}$ for $i < j < k$ in $[n]$, which is mapped to $\delta^{1}_{j-i}: \square^{k-i-2}\rightarrow \square^{k-i-1}$ by $\dim$. It follows that $W^{n}_{c}$ and $\Phi(D)^{n}$ are isomorphic as categories.

It remains to verify that these isomorphims are natural in $n$. Consider the coface map $\delta_{k}: [n-1]\rightarrow [n]$ for $0\leq k\leq n$ and $i < j$ in $[n-1]$. It follows from Construction \ref{construction: cosimplicial object generated by necklicial diagram} that $\Phi(D)(\delta_{k})_{i,j}: \Phi(D)^{n}(i,j)\rightarrow \Phi(D)^{n+1}(\delta_{k}(i),\delta_{k}(j))$ is given by
$$
\Phi(D)(\delta_{k})_{i,j} =
\begin{cases}
\yo\dim(\delta_{k-i}) & \text{if }i < k\leq j\\
\yo\dim(\id_{\Delta^{j-i}}) & \text{otherwise}
\end{cases}
=
\begin{cases}
\delta^{0}_{k-i} & \text{if }i < k\leq j\\
\id_{\square^{j-i-1}} & \text{otherwise}
\end{cases}
$$
which coincides with $W_{c}(\delta_{k})_{i,j}: W^{n}_{c}(i,j)\rightarrow W^{n+1}_{c}(\delta_{k}(i),\delta_{k}(j))$. Similarly, consider the codegeneracy map $\sigma_{k}: [n+1]\rightarrow [n]$ for $0\leq k\leq n$ and $i < j$ in $[n+1]$. Then
$$
\Phi(D)(\sigma_{k})_{i,j} =
\begin{cases}
\yo\dim(\sigma_{k-i}) & \text{if }i\leq k < j\\
\yo\dim(\id_{\Delta^{j-i}}) & \text{otherwise}
\end{cases}
=
\begin{cases}
\sigma_{1} & \text{if }k = i < j-1\\
\gamma_{k-i} & \text{if }i < k < j-1\\
\sigma_{j-i-1} & \text{if } i < k = j-1\\
\id_{\square^{j-i-1}} & \text{otherwise}
\end{cases}
$$
which coincides with $W_{c}(\sigma_{k})_{i,j}$.
\end{proof}

In \cite{legrignou2020cubical}, Le Grignou provides a general procedure for constructing right-adjoint functors $\mathcal{W}\Cat\rightarrow \SSet$ from a given strong monoidal functor $\square\rightarrow \mathcal{W}$. In fact, it is shown in \cite{cisinski2006prefaisceaux} that such strong monoidal functors are equivalent to so-called ``monoidal segments'' in $\mathcal{W}$, but we will not go further into them here.

\begin{Con}\label{construction: cubical nerve procedure}
Let $H: \square\rightarrow \mathcal{W}$ be a given monoidal functor. By left Kan extension, $H$ induces a monoidal adjunction $L^{H}: \CSet\leftrightarrows \mathcal{W}: R^{H}$, i.e. the left-adjoint $L^{H}$ is strong monoidal as well, whereby the right-adjoint $R^{H}$ is lax monoidal. Applying this adjunction to hom-objects, we obtain an adjunction
$$
L^{H}: \Cat_{\square}\leftrightarrows \mathcal{W}\Cat: R^{H}
$$
Then consider the composite
$$
N^{H}: \mathcal{W}\Cat\xrightarrow {R^{H}} \Cat_{\square}\xrightarrow{N^{\cub}} \SSet
$$
\end{Con}

\begin{Cor}\label{corollary: cubical vs necklicial procedure}
Let $H: \square\rightarrow \mathcal{W}$ be a strong monoidal functor. Then the nerve $N^{H}$ of Construction \ref{construction: cubical nerve procedure} is necklicial with associated diagram given by the composite
$$
\nec\xrightarrow{\dim} \square\xrightarrow{H} \mathcal{W}
$$
\end{Cor}
\begin{proof}
Let $L^{cub}$ denote the left-adjoint of $N^{cub}$. Then it suffices to note that we have an isomorphism of $\mathcal{W}$-categories which is natural in $n\geq 0$:
$$
L^{H}\left(L^{cub}(\Delta^{n})\right)\simeq L^{H}(W_{c}^{n})\simeq L^{H}(\Phi(\yo\circ \dim)^{n})\simeq \Phi(L^{H}\circ \yo\circ \dim)^{n}\simeq \Phi(H\circ \dim)^{n}.
$$
\end{proof}

\section{Enriched nerves of enriched categories}\label{section: Enriched nerves of enriched categories}

For the remainder of the text we return to the general case where $\mathcal{W}$ is a $\mathcal{V}$-enriched monoidal category as described in \S\ref{subsection: Notations and conventions}. In other words, we consider nerves of the form $\mathcal{W}\Cat\rightarrow \ts\mathcal{V}$, generated by a strong monoidal diagram $D: \nec\rightarrow \mathcal{W}$. This section contains the main results of the paper, one for each subsection.

\subsection{A general procedure}\label{subsection: A general procedure}

In this subsection, we describe a general procedure for producing enriched nerve functors $N^{D}_{\mathcal{V}}: \mathcal{W}\Cat\rightarrow \ts\mathcal{V}$, landing in templicial objects (see \S\ref{subsection: Templicial objects}), from a given strong monoidal functor $D: \nec\rightarrow \mathcal{W}$. In fact, we will generalize to when $D$ is merely \emph{colax} monoidal (see Construction \ref{construction: nerve generated by necklicial diagram}) as this will come into play in some examples in Section \ref{section: Examples}. In this case, the resulting functor $N^{D}_{\mathcal{V}}$ is not guaranteed have a left-adjoint however.

This procedure further allows to lift any necklicial nerve to an enriched version landing in $\ts\mathcal{V}$ (Theorem \ref{theorem: underlying D-nerve is cat. nerve assoc. to D}). We end the subsection with a description of the simplices of a necklicial nerve in terms of $D$ (Corollary \ref{corollary: n-simplex of underlying D-nerve}).

\begin{Def}\label{definition: adjunction induced by strong monoidal necklace diagram}
Given a functor $D: \nec\rightarrow \mathcal{W}$, we define an adjunction
\begin{equation}\label{diagram: adjunction induced by strong monoidal necklace diagram}
\begin{tikzcd}
	{\mathcal{V}^{\nec^{op}}} & {\mathcal{W}}
	\arrow[""{name=0, anchor=center, inner sep=0}, "{\mathfrak{l}^{D}_{\mathcal{V}}}", shift left=2, from=1-1, to=1-2]
	\arrow[""{name=1, anchor=center, inner sep=0}, "{\mathfrak{n}^{D}_{\mathcal{V}}}", shift left=2, from=1-2, to=1-1]
	\arrow["\dashv"{anchor=center, rotate=-90}, draw=none, from=0, to=1]
\end{tikzcd}
\end{equation}
between the category $\mathcal{V}^{\nec^{op}}$ of functors $\nec^{op}\rightarrow \mathcal{V}$ and $\mathcal{W}$, as follows. Define:
\begin{itemize}
\item $\mathfrak{l}^{D}_{\mathcal{V}}$ by the following weighted colimit in $\mathcal{W}$, which can be realized as a coequalizer, for any functor $X: \nec^{op}\rightarrow \mathcal{V}$:
$$
\mathfrak{l}^{D}_{\mathcal{V}}(X) = {\colim_{T\in \nec}}^{X_{T}}D(T)\in \mathcal{W}\simeq \mathrm{coeq}\left(\coprod_{\substack{T\rightarrow U\\ \text{in }\nec}}X_{U}\cdot D(T)\underset{\beta}{\overset{\alpha}{\rightrightarrows}} \coprod_{T\in \nec}X_{T}\cdot D(T)\right)
$$
where $\alpha$ and $\beta$ apply $X$ and $D$ to a map $T\rightarrow U$ respectively, and
\item $\mathfrak{n}^{D}_{\mathcal{V}}$ by the $\mathcal{V}$-enrichment of $\mathcal{W}$, for any $T\in \nec$ and $W\in \mathcal{W}$:
$$
\mathfrak{n}^{D}_{\mathcal{V}}(W) = [D(-),W]\in \mathcal{V}^{\nec^{op}}.
$$
\end{itemize}
It is clear that these are well-defined functors and they are adjoint by definition of the weighted colimit (see \cite[Definition 7.4.1]{riehl2014categorical} for example).

If $\mathcal{V} = \Set$, then we also write $\mathfrak{l}^{D} = \mathfrak{l}^{D}_{\Set}$ and $\mathfrak{n}^{D} = \mathfrak{n}^{D}_{\Set}$. Note that $\mathfrak{l}^{D}\simeq \mathfrak{l}^{D}_{\mathcal{V}}\circ F$ and $\mathfrak{n}^{D}\simeq U\circ \mathfrak{n}^{D}_{\mathcal{V}}$, where we used the free-forgetful adjunction $F: \Set^{\nec^{op}}\leftrightarrows \mathcal{V}^{\nec^{op}}: U$.
\end{Def}

The main reason for requiring the isomorphism \eqref{diagram: interchange} is the following lemma.

\begin{Lem}\label{lemma: weighted colimit colax monoidal}
Let $D: \nec\rightarrow \mathcal{W}$ be a colax (respectively strong) monoidal functor. Then $\mathfrak{l}^{D}_{\mathcal{V}}: \mathcal{V}^{\nec^{op}}\rightarrow \mathcal{W}$ is colax (respectively strong) monoidal with respect to the Day convolution on $\mathcal{V}^{\nec^{op}}$.
\end{Lem}
\begin{proof}
Let $X,Y\in \mathcal{V}^{\nec^{op}}$. Since $X\otimes_{Day} Y\simeq \Lan_{\vee}(X(-)\otimes Y(-))$, the left Kan extension of $\nec^{op}\times \nec^{op}\rightarrow \mathcal{V}: (T,U)\mapsto X_{T}\otimes Y_{U}$ along $\vee: \nec^{op}\times \nec^{op}\rightarrow \nec^{op}$, we have:
$$
\mathfrak{l}^{D}_{\mathcal{V}}(X\otimes_{Day} Y) = {\colim_{T\in \nec}}^{(X\otimes_{Day} Y)_{T}}D(T)\simeq {\colim_{U,V\in \nec}}^{X_{U}\otimes Y_{V}}D(U\vee V)
$$
Now by the colax monoidal structure of $D$, we have an induced morphism
$$
\mathfrak{l}^{D}_{\mathcal{V}}(X\otimes_{Day} Y)\rightarrow {\colim_{U,V\in \nec}}^{X_{U}\otimes Y_{V}}D(U)\otimes_{\mathcal{W}} D(V)\rightarrow \mathfrak{l}^{D}_{\mathcal{V}}(X)\otimes_{\mathcal{W}} \mathfrak{l}^{D}_{\mathcal{V}}(Y)
$$
It follows from \eqref{diagram: interchange} that the second morphism is always an isomorphism. If $D$ is moreover strong monoidal, then the first morphism is an isomorphism as well. Further, the monoidal unit of $\mathcal{V}^{\nec^{op}}$ is $\underline{I} = F(\nec(-,\{0\}))$, the constant functor on $I$. Now, the colax monoidal structure on $D$ induces a morphism
$$
\mathfrak{l}^{D}_{\mathcal{V}}(\underline{I}) = {\colim_{T\in \nec}}^{I}D(T)\simeq D(\Delta^{0})\rightarrow I
$$
which is again an isomorphism if $D$ is strong monoidal. The coassociativity and counitality immediately follow from those of $D$, so that these morphisms indeed equip $\mathfrak{l}^{D}_{\mathcal{V}}$ with a colax monoidal structure, which is strong monoidal when $D$ is.
\end{proof}

\begin{Con}\label{construction: nerve generated by necklicial diagram}
We construct a functor
$$
\Colax(\nec,\mathcal{W})^{op}\rightarrow \Fun(\mathcal{W}\Cat,S_{\otimes}\mathcal{V})
$$
from the opposite of the category of colax monoidal functors $\nec\rightarrow \mathcal{W}$ and monoidal natural transformations between them.

Given a colax monoidal functor $D: \nec\rightarrow \mathcal{W}$, consider the adjunction $\mathfrak{l}^{D}_{\mathcal{V}}\dashv \mathfrak{n}^{D}_{\mathcal{V}}$. By Lemma \ref{lemma: weighted colimit colax monoidal}, the left-adjoint $\mathfrak{l}^{D}_{\mathcal{V}}$ is colax monoidal and thus $\mathfrak{n}^{D}_{\mathcal{V}}$ is lax monoidal. Its lax structure is induced by (for $U,V\in \nec$ and $W,W'\in \mathcal{W}$):
$$
\mathfrak{n}^{D}_{\mathcal{V}}(W)_{U}\otimes \mathfrak{n}^{D}_{\mathcal{V}}(W')_{V} = [D(U),W]\otimes [D(V),W']\rightarrow [D(U\vee T), W\otimes_{\mathcal{W}} W'] = \mathfrak{n}^{D}_{\mathcal{V}}(W\otimes_{\mathcal{W}} W')_{U\vee V}
$$
where we used the fact that $\otimes_{\mathcal{W}}$ is a $\mathcal{V}$-functor and the colax monoidal structure of $D$.

Therefore, applying $\mathfrak{n}^{D}_{\mathcal{V}}$ to hom-objects induces a functor
$$
\mathfrak{n}^{D}_{\mathcal{V}}: \mathcal{W}\Cat\rightarrow \mathcal{V}\Cat_{\nec}
$$
Finally, we obtain a functor as the composite 
$$
N^{D}_{\mathcal{V}}: \mathcal{W}\Cat\xrightarrow{\mathfrak{n}^{D}_{\mathcal{V}}} \mathcal{V}\Cat_{\nec}\xrightarrow{(-)^{temp}} \ts\mathcal{V}
$$
where $(-)^{temp}$ is the right-adjoint of \eqref{diagram: nec-temp adjunction}. It is clear that the construction $D\mapsto N^{D}_{\mathcal{V}}$ is functorial in $D$.
\end{Con}

\begin{Def}
Let $D: \nec\rightarrow \mathcal{W}$ be a colax monoidal functor. We call the induced functor $N^{D}_{\mathcal{V}}$ of Construction \ref{construction: nerve generated by necklicial diagram} the \emph{nerve generated by $D$} or \emph{$D$-nerve}.

If $\mathcal{V} = \Set$, we will also write $N^{D} = N^{D}_{\Set}$.
\end{Def}

\begin{Ex}
Let $\mathcal{C}$ be a small $\mathcal{W}$-category and $D: \nec\rightarrow \mathcal{W}$ a colax monoidal functor. We describe the $D$-nerve $N^{D}_{\mathcal{V}}(\mathcal{C})$ in low dimensions, using the inductive description of the functor $(-)^{temp}$ (see \cite[Construction 3.11]{lowen2024enriched}).
\begin{itemize}
\item The vertices of $N^{D}_{\mathcal{V}}(\mathcal{C})$ are given by the objects of $\mathcal{C}$.
\item Take objects $A,B\in \Ob(\mathcal{C})$. Then
$$
N^{D}_{\mathcal{V}}(\mathcal{C})_{1}(A,B) = [D(\Delta^{1}),\mathcal{C}(A,B)]\in \mathcal{V}
$$
\item Take objects $A,B\in \Ob(\mathcal{C})$. Then $N^{D}_{\mathcal{V}}(\mathcal{C})_{2}(A,B)$ is given by the following pullback in $\mathcal{V}$:
\[\begin{tikzcd}
	{N^{D}_{\mathcal{V}}(\mathcal{C})_{2}(A,B)} & {\coprod_{C\in \Ob(\mathcal{C})}[D(\Delta^{1}),\mathcal{C}(A,C)]\otimes [D(\Delta^{1}),\mathcal{C}(C,B)]} \\
	{[D(\Delta^{2}),\mathcal{C}(A,B)]} & {[D(\Delta^{1}\vee \Delta^{1}),\mathcal{C}(A,B)]}
	\arrow[from=1-2, to=2-2]
	\arrow["\mu_{1,1}", from=1-1, to=1-2]
	\arrow["p_{2}"', from=1-1, to=2-1]
	\arrow["{[D(\nu_{1,1}),\mathcal{C}(A,B)]}"', from=2-1, to=2-2]
\end{tikzcd}\]
where the right vertical map is induced by $D(\Delta^{1}\vee \Delta^{1})\rightarrow D(\Delta^{1})\otimes D(\Delta^{1})$ and the composition in $\mathcal{C}$.

Then the induced morphism $\mu_{1,1}$ is the comultiplication of $N^{D}_{\mathcal{V}}(\mathcal{C})$ and the face morphism $d_{1}: N^{D}_{\mathcal{V}}(\mathcal{C})_{2}\rightarrow N^{D}_{\mathcal{V}}(\mathcal{C})_{1}$ is given by $[D(\delta_{1}),\mathcal{C}(A,B)]\circ p_{2}$.
\end{itemize}
\end{Ex}

\begin{Prop}\label{proposition: nerve gen. by strong mon. diagram has left-adj.}
Let $D: \nec\rightarrow \mathcal{W}$ be a strong monoidal functor. Then the $D$-nerve $N^{D}_{\mathcal{V}}$ has a left-adjoint $L^{D}_{\mathcal{V}}: \ts\mathcal{V} \rightarrow \mathcal{W}\Cat$.
\end{Prop}
\begin{proof}
Since $D$ is strong monoidal, it follows that $\mathfrak{l}^{D}_{\mathcal{V}}: \mathcal{V}^{\nec^{op}}\rightarrow \mathcal{W}$ is strong monoidal as well by Lemma \ref{lemma: weighted colimit colax monoidal}, i.e. \eqref{diagram: adjunction induced by strong monoidal necklace diagram} is a monoidal adjunction. Thus applying $\mathfrak{l}^{D}_{\mathcal{V}}$ on hom-objects defines a functor $\mathfrak{l}^{D}_{\mathcal{V}}: \mathcal{V}\Cat_{\nec}\rightarrow \mathcal{W}\Cat$ which is left-adjoint to $\mathfrak{n}^{D}_{\mathcal{V}}: \mathcal{W}\Cat\rightarrow \mathcal{V}\Cat_{\nec}$. Hence, $N^{D}_{\mathcal{V}}$ has a left-adjoint given by the composite
$$
L^{D}_{\mathcal{V}}: \ts\mathcal{V}\xrightarrow{(-)^{nec}} \mathcal{V}\Cat_{\nec}\xrightarrow{\mathfrak{l}^{D}_{\mathcal{V}}} \mathcal{W}\Cat
$$
\end{proof}

\begin{Rem}
Note that by definition of the $D$-nerve, the left-adjoint $L^{D}_{\mathcal{V}}$ of $N^{D}_{\mathcal{V}}$ is given as follows. For a templicial object $(X,S)$, its associated $\mathcal{W}$-category $L^{D}_{\mathcal{V}}(X)$ has object set $S$ and for all $a,b\in S$, the hom-object is given by the weighted colimit $$
L^{D}_{\mathcal{V}}(X)(a,b)\simeq {\colim_{T\in \nec}}^{X_{T}(a,b)}D(T)
$$
The composition law of $L^{D}_{\mathcal{V}}(D)$ is induced by the isomorphisms $D(T)\otimes D(U)\simeq D(T\vee U)$ for $T,U\in \nec$.
\end{Rem}

Let us compare the $D$-nerve to the necklicial nerves of Section \ref{section: Nerves of enriched categories}.

\begin{Thm}\label{theorem: underlying D-nerve is cat. nerve assoc. to D}
Let $D: \nec\rightarrow \mathcal{W}$ be strong monoidal functor. Then
$$
\tilde{U}\circ N^{D}_{\mathcal{V}}\simeq N^{\Phi(D)}
$$
where $N^{\Phi(D)}: \mathcal{W}\Cat\rightarrow \SSet$ is the right-adjoint associated to $\Phi(D)$ under \eqref{equation: classical equivalence}, and this isomorphism is natural in $D$. In particular, we may identify $N^{\Phi(D)}$ with $N^{D}$.
\end{Thm}
\begin{proof}
It suffices to construct an isomorphism $L^{D}_{\mathcal{V}}\tilde{F}(\Delta^{n})\simeq \Phi(D)^{n}$ which is natural in $n\geq 0$. First note that both $\mathcal{W}$-categories have $[n]$ as their set of objects. Further, for all $i\leq j$ in $[n]$, we have
$$
L^{D}_{\mathcal{V}}\tilde{F}(\Delta^{n})(i,j) = {\colim_{T\in \nec}}^{F(\Delta^{n}_{T}(i,j))}D(T)\simeq D(\Delta^{j-i}) = \Phi(D)^{n}(i,j)
$$
where the isomorphism is induced by the fact that $\Delta^{n}_{\bullet}(i,j)\simeq \Delta^{j-i}_{\bullet}(0,j-i)$ in $\Set^{\nec^{op}}$. It is clear that these isomorphisms are compatible with the compostition and identities of both $\mathcal{W}$-categories. The naturality in $D$ immediately follows from the definitions. Finally if $\mathcal{V} = \Set$, then $\tilde{U}$ is an equivalence and thus we may identify $N^{\Phi(D)}$ with $N^{D}$.
\end{proof}

\begin{Cor}\label{corollary: n-simplex of underlying D-nerve}
Let $D: \nec\rightarrow \mathcal{W}$ be a colax monoidal functor, $\mathcal{C}$ a small $\mathcal{W}$-category and $n\geq 0$ an integer. An $n$-simplex of $N^{D}(\mathcal{C})$ is equivalent to a pair
$$
\left((A_{i})_{i=0}^{n},(\alpha_{i,j})_{0\leq i < j\leq n}\right)
$$
with $A_{i}\in \Ob(\mathcal{C})$ and $\alpha_{i,j}: D(\Delta^{j-i})\rightarrow \mathcal{C}(A_{i},A_{j})$ in $\mathcal{W}$ such that for all $i < k < j$ in $[n]$ the following diagram commutes in $\mathcal{W}$:
\[\begin{tikzcd}
	{D(\Delta^{k-i}\vee \Delta^{j-k})} & {D(\Delta^{j-k})\otimes D(\Delta^{j-k})} & {\mathcal{C}(A_{i},A_{k})\otimes \mathcal{C}(A_{k},A_{j})} \\
	{D(\Delta^{j-i})} & & {\mathcal{C}(A_{i},A_{j})}
	\arrow[from=1-1, to=1-2]
	\arrow["{m_{\mathcal{C}}}", from=1-3, to=2-3]
	\arrow["{\alpha_{i,k}\otimes \alpha_{k,j}}", from=1-2, to=1-3]
	\arrow["{D(\nu_{k-i,j-k})}"', from=1-1, to=2-1]
	\arrow["{\alpha_{i,j}}"', from=2-1, to=2-3]
\end{tikzcd}\]
\end{Cor}
\begin{proof}
From Theorem \ref{theorem: underlying D-nerve is cat. nerve assoc. to D} and \cite[Proposition 3.14]{lowen2024enriched}, we have $N^{D}\simeq \tilde{U}\circ N^{D}_{\mathcal{V}}\simeq (-)^{temp}\circ \mathcal{U}\circ \mathfrak{n}^{D}_{\mathcal{V}}$ where $\mathcal{U}: \mathcal{V}\Cat_{\nec}\rightarrow \Cat_{\nec}$ is the forgetful functor induced by $U: \mathcal{V}\rightarrow \Set$. So this follows from Example \ref{example: temp construction in Set-case}.
\end{proof}

\subsection{Explicitation of the left-adjoint}\label{subsection: Explicitation of the left-adjoint}

For this subsection, we fix a strong monoidal diagram $D: \nec\rightarrow \mathcal{W}$. Then we have an induced left-adjoint $L^{D}_{\mathcal{V}}: \ts\mathcal{V}\rightarrow \mathcal{W}\Cat$ (Proposition \ref{proposition: nerve gen. by strong mon. diagram has left-adj.}). In Theorem \ref{theorem: explicitation of left-adjoint for free temp. obj.} we provide sufficient conditions on $D$ so that $L^{D}_{\mathcal{V}}$ can be described more explicitly. In particular, this also applies to the left-adjoints of some classical necklicial nerves such as the differential graded nerve (see \S\ref{subsection: The differential graded nerve}).

Given a simplicial set $K$ with vertices $a$ and $b$, Dugger and Spivak give an explicit description of the $n$-simplices of the mapping spaces $\mathfrak{C}[K](a,b)$ where $\mathfrak{C}$ is the left-adjoint of the homotopy coherent nerve $N^{hc}$ \cite[Corollary 4.8]{dugger2011rigidification}. We will return to this example in \S\ref{subsection: Homotopy coherent nerves}, where $N^{hc}$ is shown to be generated by a certain diagram $hc: \nec\rightarrow \SSet$ \eqref{diagram: hc}. In this subsection, we thus extend their result in two ways: we replace $hc$ by more general diagrams $D$, allowing different types of nerves; and we replace simplicial sets by templicial objects in $\mathcal{V}$.

\begin{Def}
Let $\nec_{-}$ and $\nec_{+}$ denote the the wide subcategories of $\nec$ consisting of all active surjective necklace maps and all injective necklace maps respectively. We denote the inclusion $\nec_{-}\hookrightarrow \nec$ by $\iota$.
\end{Def}

\begin{Rem}\label{remark: act. surj. - inj. factorization system}
It is easy to see that the subcategories $(\nec_{-},\nec_{+})$ form an (orthogonal) factorization system on $\nec$ in the sense of \cite{bousfield1976constructions}.
\end{Rem}

\begin{Lem}\label{lemma: left Kan ext. of surj. necklace maps}
Let $D': \nec_{-}\rightarrow \mathcal{V}$ be a functor. Then for any $T$, we have an isomorphism
$$
(\Lan_{\iota}D')(T)\simeq \coprod_{\substack{U\hookrightarrow T\\ \text{in }\nec_{+}}}D'(U)
$$
\end{Lem}
\begin{proof}
From Remark \ref{remark: act. surj. - inj. factorization system} it is easy to see that the discrete subcategory of $(\iota\downarrow T)$ spanned by all injective necklace maps $U\hookrightarrow T$ is reflective. In other words, its inclusion into $(\iota\downarrow T)$ is a right-adjoint and thus a final functor. Consequently,
$$
(\Lan_{\iota}D')(T)\simeq \colim_{(U\rightarrow T) \in(\iota\downarrow T)}D'(U)\simeq \coprod_{\substack{U\hookrightarrow T\\ \text{in }\nec_{+}}}D'(U)
$$
\end{proof}

Since the monoidal category $\mathcal{V}$ is assumed to be closed, it is canonically tensored over itself by its monoidal product $\otimes$. We call a $\mathcal{V}$-enriched functor $\pi: \mathcal{W}\rightarrow \mathcal{V}$ \emph{tensor preserving} if the canonical morphism $V\otimes \pi(W)\rightarrow \pi(V\cdot W)$ is an isomorphism for all $V\in \mathcal{V}$ and $W\in \mathcal{W}$.

\begin{Prop}\label{proposition: explicitation of left-adjoint to active surjective necklace maps}
Let $\pi: \mathcal{W}\rightarrow \mathcal{V}$ a be colimit and tensor preserving $\mathcal{V}$-functor. Suppose that $\pi D\simeq \Lan_{\iota}D'$ for some functor $D': \nec_{-}\rightarrow \mathcal{V}$. Then for any templicial object $(X,S)$ with $a,b\in S$,
$$
\pi(L^{D}_{\mathcal{V}}(X)(a,b))\simeq {\colim_{T\in \nec_{-}}}^{X_{T}(a,b)}D'(T)
$$
\end{Prop}
\begin{proof}
This is immediate since
$$
\pi(L^{D}_{\mathcal{V}}(X)(a,b))\simeq {\colim_{T\in \nec}}^{X_{T}(a,b)}\pi D(T)\simeq {\colim_{T\in \nec}}^{X_{T}(a,b)}(\Lan_{\iota}D')(T)\simeq {\colim_{T\in \nec_{-}}}^{X_{T}(a,b)}D'(T)
$$
\end{proof}

Before stating our main theorem in this subsection, we recall a what it means for a templicial object to have non-degenerate simplices.

\begin{Def}[Definition 4.14, \cite{lowen2024enriched}]\label{definition: non-deg. simp.}
Let $(X,S)$ be a templicial object and denote the quiver $X^{deg}_{n} = \colim_{\substack{\sigma: [n]\twoheadrightarrow [k]\text{ surj.}\\ 0\leq k < n}}X_{k}$ for any integer $n\geq 0$. We say that $X$ \emph{has non-degenerate simplices} if for every $n\geq 0$, there exists some $X^{nd}_{n}\in \mathcal{V}\Quiv_{S}$ such that the canonical quiver morphism $X^{deg}_{n}\rightarrow X_{n}$ is isomorphic to the coprojection
$$
X^{deg}_{n}\rightarrow X^{deg}_{n}\amalg X^{nd}_{n}
$$
\end{Def}

It was shown in \cite[Lemma 2.19]{lowen2024enriched} that for any templicial object that has non-degenerate simplicies, we have isomorphism for every $n\geq 0$:
$$
X_{n}\simeq \coprod_{\substack{\sigma: [n]\twoheadrightarrow [k]\\ \text{surjective}}}X^{nd}_{k}
$$

\begin{Thm}\label{theorem: explicitation of left-adjoint for free temp. obj.}
Let $\pi: \mathcal{W}\rightarrow \mathcal{V}$ be a colimit and tensor preserving $\mathcal{V}$-functor. Suppose that $\pi D\simeq \Lan_{\iota}D'$ for some functor $D': \nec_{-}\rightarrow \mathcal{V}$. Then for any templicial object $(X,S)$ that has non-degenerate simplices and $a,b\in S$, 
$$
\pi(L^{D}_{\mathcal{V}}(X)(a,b))\simeq \coprod_{T\in \nec}X^{nd}_{T}(a,b)\cdot D'(T)
$$
where $X^{nd}_{T} = X^{nd}_{t_{1}}\otimes_{S} ...\otimes_{S} X^{nd}_{p-t_{k-1}}$ for every necklace $T = \{0 < t_{1} < ... < t_{k-1} < p\}$.
\end{Thm}
\begin{proof}
Let $T = \{0 < t_{1} < ... < t_{k-1} < p\}$ be a necklace. Then it follows that
$$
X_{T}\simeq \coprod_{\substack{f_{i}: [t_{i}-t_{i-1}]\twoheadrightarrow [n_{i}]\\ i\in \{1,...,k\}}}X^{nd}_{n_{1}}\otimes_{S} ...\otimes_{S} X^{nd}_{n_{k}}\simeq \coprod_{\substack{T\twoheadrightarrow U\\ \text{in }\nec_{-}}}X^{nd}_{U}
$$
where we used that any active surjective necklace map $T\twoheadrightarrow U$ is determined by its underlying map $f: [p]\rightarrow [q]$ in $\fint$, which in turn can be uniquely decomposed as $f = f_{1} + ... + f_{k}$ with $f_{i}: [t_{i}-t_{i-1}]\twoheadrightarrow [n_{i}]$ surjective. Hence, by Proposition \ref{proposition: explicitation of left-adjoint to active surjective necklace maps},
\begin{align*}
\pi(L^{D}_{\mathcal{V}}(X)(a,b)) & \simeq {\colim_{T\in \nec_{-}}}^{X_{T}(a,b)}D'(T)\simeq \int^{T\in \nec_{-}}X_{T}(a,b)\cdot D'(T)\\
&\simeq \int^{T\in \nec_{-}}\coprod_{\substack{T\twoheadrightarrow U\\ \text{in }\nec_{-}}}X^{nd}_{U}(a,b)\cdot D'(T)\\
&\simeq  \coprod_{U\in \nec_{-}}X^{nd}_{U}(a,b)\cdot \int^{T\in \nec_{-}}F(\nec_{-}(T,U))\cdot D'(T)
\end{align*}
By the coYoneda lemma, this is further isomorphic to $\coprod_{U\in \nec_{-}}X^{nd}_{U}(a,b)\cdot D'(U)$.
\end{proof}

\begin{Rem}
In case $\mathcal{V} = \Set$, then $X^{nd}_{T}(a,b)$ for a given necklace $T$ can also be described as the set of all \emph{totally non-degenerate maps} $T\rightarrow X_{a,b}$ in $\SSet_{*,*}$, in the sense of \cite{dugger2011rigidification}. That is, a map $T\rightarrow X_{a,b}$ is totally non-degenerate if it maps every bead of $T$ to a non-degenerate simplex of $X$.
\end{Rem}

\begin{Cor}\label{corollary: explicitation of the left-adjoint for simp. sets}
Let $\pi: \mathcal{W}\rightarrow \mathcal{V}$ be a colimit and tensor preserving $\mathcal{V}$-functor. Suppose there exists $D': \nec_{-}\rightarrow \mathcal{V}$ such that $\pi D\simeq \Lan_{\iota}D'$. Then for any simplicial set $K$ with $a,b\in K_{0}$,
$$
\pi(L^{D}(K)(a,b))\simeq \coprod_{T\in \nec}F(K^{nd}_{T}(a,b))\otimes D'(T)
$$
\end{Cor}
\begin{proof}
Apply Theorem \ref{theorem: explicitation of left-adjoint for free temp. obj.} to the templicial object $\tilde{F}(K)$ since $L^{D}_{\mathcal{V}}\circ \tilde{F}\simeq L^{D}$ (Theorem \ref{theorem: underlying D-nerve is cat. nerve assoc. to D}) and $\tilde{F}(K)$ always has non-degenerate simplices by \cite[Example 2.17]{lowen2024enriched}.
\end{proof}

\subsection{Quasi-categories in $\mathcal{V}$}\label{subsection: Quasi-categories in V}

In \cite{lowen2024enriched} we introduced quasi-categories in $\mathcal{V}$ as an enriched generalization of the classical quasi-categories by Joyal \cite{joyal2002quasi}. They are templicial objects satisfying an analogue of the weak Kan condition, as we will recall shortly. In this subsection, we provide conditions in terms of the generating diagram $D: \nec\rightarrow \mathcal{W}$ such that the $D$-nerve of a $\mathcal{W}$-category is a quasi-category in $\mathcal{V}$. In particular, this can is applicable for classical quasi-categories as well.

\begin{Def}[Definition 5.4, \cite{lowen2024enriched}]
We say a functor $Y: \nec^{op}\rightarrow \mathcal{V}$ \emph{lift inner horns} if for all integers $0 < j < n$, the following lifting problem in $\mathcal{V}^{\nec^{op}}$ has a solution in $\mathcal{V}^{\nec^{op}}$:
\[\begin{tikzcd}
	{\tilde{F}(\Lambda^{n}_{j})_{\bullet}(0,n)} & {X_{\bullet}(A,B)} \\
	{\tilde{F}\left(\Delta^{n}\right)_{\bullet}(0,n)}
	\arrow[from=1-1, to=2-1]
	\arrow[from=1-1, to=1-2]
	\arrow[dashed, from=2-1, to=1-2]
\end{tikzcd}\]

A $(X,S)$ templicial object in $\mathcal{V}$ is called a \emph{quasi-category in $\mathcal{V}$} if $X_{\bullet}(a,b)$ lifts inner horns for all $a,b\in S$.
\end{Def}

It is shown in \cite[Proposition 5.8 and Corollary 5.13]{lowen2024enriched} that this definition recovers the classical notion of a quasi-category when $\mathcal{V} = \Set$ and that the underlying simplicial set $\tilde{U}(X)$ of a quasi-category $X$ in $\mathcal{V}$ is a classical quasi-category. Moreover, for any necklace category $\mathcal{C}\in \mathcal{V}\Cat_{\nec}$ such that $\mathcal{C}(A,B)$ lifts inner horns for all $A,B\in \Ob(\mathcal{C})$, we have that $\mathcal{C}^{temp}\in \ts\mathcal{V}$ is a quasi-category in $\mathcal{V}$ (see \cite[Proposition 5.10]{lowen2024enriched}).

Recall the adjunction $\mathfrak{l}^{D}\dashv \mathfrak{n}^{D}$ of Definition \ref{definition: adjunction induced by strong monoidal necklace diagram}.

\begin{Lem}\label{lemma: weighted colimit for simp. sets}
Let $K$ be a simplicial set with $a,b\in K_{0}$, and $D: \nec\rightarrow \mathcal{W}$ a colax monoidal diagram. Then we have a canonical bijection
$$
\mathfrak{l}^{D}(K_{\bullet}(a,b))\simeq \colim_{\substack{T\rightarrow K_{a,b}\text{ in }\SSet_{*,*}\\ T\text{ in }\nec}}D(T)
$$
\end{Lem}
\begin{proof}
Applying Yoneda's lemma to every bead, we find a bijection $\SSet_{*,*}(T,K_{a,b})\simeq K_{T}(a,b)$ for all $T\in \nec$. Hence, the right-hand side is isomorphic to $(\Lan_{\yo}D)(K_{\bullet}(a,b))$ with $\yo: \nec \hookrightarrow \Set^{\nec^{op}}$ the Yoneda embedding. Now note that $\Lan_{\yo}D\simeq \mathfrak{l}^{D}$.
\end{proof}

\begin{Thm}\label{theorem: D-nerve is quasi-cat.}
Let $D: \nec\rightarrow \mathcal{W}$ be a colax monoidal functor and $\mathcal{C}$ be a small $\mathcal{W}$-category. Assume that for all $A,B\in \Ob(\mathcal{C})$ and $0 < j < n$ the following lifting problem in $\mathcal{W}$ has a solution:
\[\begin{tikzcd}
	{\underset{\substack{T\rightarrow (\Lambda^{n}_{j})_{0,n} \text{in }\SSet_{*,*}\\ T\in \nec}}{\colim}D(T)} & {\mathcal{C}(A,B)} \\
	{D\left(\Delta^{n}\right)}
	\arrow[from=1-1, to=2-1]
	\arrow[from=1-1, to=1-2]
	\arrow[dashed, from=2-1, to=1-2]
\end{tikzcd}\]
Then $N^{D}_{\mathcal{V}}(\mathcal{C})$ is a quasi-category in $\mathcal{V}$. In particular, the simplicial set $N^{D}(\mathcal{C})$ is an ordinary quasi-category.
\end{Thm}
\begin{proof}
Since $\mathfrak{l}^{D}\simeq \mathfrak{l}^{D}_{\mathcal{V}}\circ F$, we have by Lemma \ref{lemma: weighted colimit for simp. sets} and the adjunction \eqref{diagram: adjunction induced by strong monoidal necklace diagram} that the above lifting problem has a solution for all $0 < j < n$ if and only if $\mathfrak{n}^{D}_{\mathcal{V}}(\mathcal{C})_{\bullet}(A,B)$ lifts inner horns in $\mathcal{V}^{\nec^{op}}$. Hence by \cite[Proposition 5.10]{lowen2024enriched}, this implies that $N^{D}_{\mathcal{V}}(\mathcal{C}) = \mathfrak{n}^{D}_{\mathcal{V}}(\mathcal{C})^{temp}$ is a quasi-category in $\mathcal{V}$. Then also $N^{D}(\mathcal{C})$ is a quasi-category by Theorem \ref{theorem: underlying D-nerve is cat. nerve assoc. to D} and \cite[Corollary 5.13]{lowen2024enriched}.
\end{proof}

\begin{Ex}
From \cite[Proposition 5.1]{lowen2024enriched} we have:
\begin{equation}\label{equation: horn as colimit}
\left(\Lambda^{n}_{j}\right)_{\bullet}(0,n)\simeq \bigcup_{\substack{i=1\\ i\neq j}}^{n-1}\delta_{i}(\Delta^{n-1}_{\bullet}(0,n))\cup \bigcup_{k=1}^{n-1}(\Delta^{k}\vee \Delta^{n-k})_{\bullet}(0,n)
\end{equation}
as subfunctors of $\Delta^{n}_{\bullet}(0,n)$ in $\Set^{\nec^{op}}$. This allows us to write out the colimit appearing Theorem \ref{theorem: D-nerve is quasi-cat.} more concretely. For example, in low dimensions we have the following.
\begin{itemize}
\item For $n = 2$ and $j = 1$,
$$
\mathfrak{l}^{D}\left((\Lambda^{2}_{1})_{\bullet}(0,2)\right)\simeq D(\Delta^{1}\vee \Delta^{1})
$$
\item For $n = 3$ and $j = 1$,
$$
\mathfrak{l}^{D}\left((\Lambda^{3}_{1})_{\bullet}(0,3)\right)\simeq D(\Delta^{2}\vee \Delta^{1})\amalg_{D(\Delta^{1}\vee \Delta^{1}\vee \Delta^{1})} D(\Delta^{1}\vee \Delta^{2})\amalg_{D(\Delta^{1}\vee \Delta^{1})} D(\Delta^{2})
$$
\end{itemize}
\end{Ex}

\subsection{Frobenius structures}\label{subsection: Frobenius structures}

In \cite{lowen2023frobenius}, we introduced Frobenius structures on templicial objects (see \S\ref{subsection: Templicial objects}). In this subsection, we provide a sufficient condition on the diagram $D: \nec\rightarrow \mathcal{W}$ such that the induced nerve has a Frobenius structure. Recall that the comultiplication maps of a templicial object can be paremetrised by inert necklace maps $\Delta^{k}\vee \Delta^{l}\hookrightarrow \Delta^{k+l}$ via the functor $(-)^{nec}$ of \eqref{diagram: nec-temp adjunction}. Similarly, the multiplication maps of a Frobenius structure can be parametrised by necklace maps in the oppposite direction $\Delta^{k+l}\rightarrow \Delta^{k}\vee \Delta^{l}$. We call such maps \emph{coinert} and the first step is to extend the category $\nec$ to also include these coinert maps.

Next, we show that adjunction \eqref{diagram: nec-temp adjunction} extends to one between $\Fs\mathcal{V}$ and categories enriched in $\mathcal{V}^{\overline{\nec}^{op}}$ (Theorem \ref{theorem: Frob. temp. nec. adjunction}), and from this show that the $D$-nerve carries a Frobenius structure whenever $D$ extends to $\overline{\nec}$. We end the subsection by showing that any nerve arising from a strong monoidal diagram $\square\rightarrow \mathcal{W}$ (see \S\ref{subsection: Necklaces versus cubes}) carries a Frobenius structure.

\begin{Def}\label{definition: extended necklace category}
We define a monoidal category $\overline{\nec}$ as follows:
\begin{itemize}
\item The objects of $\overline{\nec}$ are the same as those of $\nec$.
\item Given two necklaces $(T,p)$ and $(U,q)$, a morphism $(T,p)\rightarrow (U,q)$ in $\overline{\nec}$ is a pair $(f,U')$ with $f: [p]\rightarrow [q]$ in $\fint$ and $f(T)\cup U\subseteq U'\subseteq [q]$.
\end{itemize}
The composition of two morphisms $(f,U'): T\rightarrow U$ and $(g,V'): U\rightarrow V$ is given by the pair $(gf, V'\cup g(U'))$ and the identity on a necklace $T$ is given by the pair $(\id_{[p]},T)$.

The category $\overline{\nec}$ has a monoidal structure given on morphisms by
$$
(f,U')\vee (g,V') = (f\vee g, U'\vee T')
$$
with monoidal unit given by the necklace $(\{0\},0)$.
\end{Def}

\begin{Rem}
Note that we can identify $\nec$ with the non-full monoidal subcategory of $\overline{\nec}$ that consists of all morphisms $(f,U'): T\rightarrow U$ with $U' = f(T)$.

Moreover, letting $\nec_{inert}$ denote the subcategory of $\nec$ consisting of all inert necklace maps, we can also consider $(\nec_{inert})^{op}$ as a non-full monoidal subcategory of $\overline{\nec}$ as follows. An inert map $f: (T,q)\hookrightarrow (U,q)$ in $\nec$ can be identified with the pair $f^{co} = (\id_{[q]},T): U\rightarrow T$ in $\overline{\nec}$ (this is well-defined as $U\subseteq T$). We call such a morphism a \emph{coinert map}.
\end{Rem}

\begin{Rem}\label{remark: unique decomp. of extended necklace map}
Every morphism $f: T\rightarrow U$ in $\overline{\nec}$ can be uniquely decomposed as
$$
T\xrightarrow{f_{1}} T_{1}\xrightarrow{f^{co}_{2}} T_{2}\xrightarrow{f_{3}} U
$$
with $f_{1}$ an active necklace map, $f^{co}_{2}$ a coinert map and $f_{3}$ an inert necklace map.

Note that any coinert map is the composition of wedges $\vee$ of maps $\nu^{co}_{k,l}: \Delta^{k+l}\rightarrow \Delta^{k}\vee \Delta^{l}$ for $k,l > 0$.
\end{Rem}

As for $\nec$, we can consider the category $\mathcal{V}^{\overline{\nec}^{op}}$ of functors $\overline{\nec}^{op}\rightarrow \mathcal{V}$ equipped with the Day convolution (see \cite{day1970closed}). Then consider the category of small categories enriched in $\mathcal{V}^{\overline{\nec}^{op}}$:
$$
\mathcal{V}\Cat_{\overline{\nec}} = \mathcal{V}^{\overline{\nec}^{op}}\Cat
$$
The inclusion functor $i: \nec\hookrightarrow \overline{\nec}$ is by definition strong monoidal. Thus left Kan extension along $i$ provides a monoidal adjunction:
\[\begin{tikzcd}
	\mathcal{V}^{\nec^{op}} & {\mathcal{V}^{\overline{\nec}^{op}}}
	\arrow[""{name=0, anchor=center, inner sep=0}, "{\Lan_{i}}", shift left=2, from=1-1, to=1-2]
	\arrow[""{name=1, anchor=center, inner sep=0}, "{res_{i}}", shift left=2, from=1-2, to=1-1]
	\arrow["\dashv"{anchor=center, rotate=-90}, draw=none, from=0, to=1]
\end{tikzcd}\] 
which in turn induces an adjunction on enriched categories:
\[\begin{tikzcd}
	\mathcal{V}\Cat_{\nec} & {\mathcal{V}\Cat_{\overline{\nec}}}
	\arrow[""{name=0, anchor=center, inner sep=0}, "{\Lan_{i}}", shift left=2, from=1-1, to=1-2]
	\arrow[""{name=1, anchor=center, inner sep=0}, "{res_{i}}", shift left=2, from=1-2, to=1-1]
	\arrow["\dashv"{anchor=center, rotate=-90}, draw=none, from=0, to=1]
\end{tikzcd}\]

\begin{Con}
We construct a functor
$$
(-)^{temp}: \mathcal{V}\Cat_{\overline{\nec}}\rightarrow \Fs\mathcal{V}
$$
which lifts $(-)^{temp}$ of \eqref{diagram: nec-temp adjunction} along $\Fs\mathcal{V}\rightarrow \ts\mathcal{V}$ and $res_{i}: \mathcal{V}\Cat_{\overline{\nec}}\rightarrow \mathcal{V}\Cat_{\nec}$.

Let $\mathcal{C}$ be a $\mathcal{V}^{\overline{\nec}^{op}}$-enriched category. We construct a Frobenius structure on $res_{i}(\mathcal{C})^{temp}$. Denote the composition of $\mathcal{C}$ by $m$ and the comultiplication of $res_{i}(\mathcal{C})^{temp}$ by $\mu$. We construct a Frobenius structure:
$$
\left(Z^{k,l}: \mathcal{C}^{temp}_{k}\otimes_{S} \mathcal{C}^{temp}_{l}\rightarrow \mathcal{C}^{temp}_{k+l}\right)_{k,l\geq 0}
$$
by induction on the pairs $(k,l)$. If $k = 0$, we set $Z^{0,l}$ to be the left unit isomorphism $Z^{0,l}: I_{S}\otimes_{S} \mathcal{C}^{temp}_{l}\xrightarrow{\sim} \mathcal{C}^{temp}_{l}$. Similarly, if $l = 0$, we set $Z^{k,0}$ to be the right unit isomorphism. This forces that condition \eqref{equation: Frobenius unitality} of Definition \ref{definition: Frobenius temp. obj.} holds. Assume further that $k,l > 0$ and set $n = k + l$. For all $p,q > 0$ with $p + q = n$, define a morphism $\xi_{p,q}: \mathcal{C}^{temp}_{k}\otimes_{S} \mathcal{C}^{temp}_{l}\rightarrow \mathcal{C}^{temp}_{p}\otimes_{S} \mathcal{C}^{temp}_{q}$ by
\begin{equation}
\xi_{p,q} =
\begin{cases}
(Z^{k,l-q}\otimes \id_{\mathcal{C}^{temp}_{q}})(\id_{\mathcal{C}^{temp}_{k}}\otimes \mu_{p-k,q}) & \text{if }k < p\\
\id_{\mathcal{C}^{temp}_{k}\otimes_{S} \mathcal{C}^{temp}_{l}} & \text{if }k = p\\
(\id_{\mathcal{C}^{temp}_{k}}\otimes Z^{k-p,l})(\mu_{p,q-l}\otimes \id_{\mathcal{C}^{temp}_{l}}) & \text{if }k > p
\end{cases}
\end{equation}
If $k < p$, we have a commutative diagram in $\overline{\nec}$:
\[\begin{tikzcd}
	{\Delta^{k}\vee \Delta^{p-k}\vee \Delta^{n-p}} & {\Delta^{k}\vee \Delta^{n-k}} \\
	{\Delta^{p}\vee \Delta^{n-p}} & {\Delta^{n}}
	\arrow["{\nu^{co}_{k,n-k}}"', from=2-2, to=1-2]
	\arrow["{\id_{\Delta^{k}}\vee \nu_{p-k,q}}"', from=1-1, to=1-2]
	\arrow["{\nu^{co}_{k,p-k}\vee\id_{\Delta^{q}}}", from=2-1, to=1-1]
	\arrow["{\nu_{p,q}}"', from=2-1, to=2-2]
\end{tikzcd}\]
Moreover, we have that $\nu_{k,p-k}\circ \nu^{co}_{k,p-k} = \id_{\Delta^{p}}$.

For any integer $r > 0$, let $p_{r}: \mathcal{C}^{temp}_{r}\rightarrow \mathcal{C}_{\{0 < r\}}$ denote the canonical quiver morphism. It now follows from the definition of $(-)^{temp}$ (see \cite[Construction 3.11]{lowen2024enriched}) and the induction hypothesis that
\begin{align*}
&\mathcal{C}(\nu_{p,q})\mathcal{C}(\nu^{co}_{k,n-k})m(p_{k}\otimes p_{l})\\
&= \mathcal{C}(\nu^{co}_{k,p-k} \vee \id_{\Delta^{q}})\mathcal{C}(\id_{\Delta^{k}} \vee \nu_{p-k,n-p})m(p_{k}\otimes p_{l})\\
&= \mathcal{C}(\nu^{co}_{k,p-k} \vee  \id_{\Delta^{q}})m(\id_{\mathcal{C}_{\{0 < k\}}}\otimes \mathcal{C}(\nu_{p-k,n-p}))(p_{k}\otimes p_{l})\\
&= \mathcal{C}(\nu^{co}_{k,p-k} \vee  \id_{\Delta^{q}})m(p_{k}\otimes m(p_{p-k}\otimes p_{q})\mu_{p-k,q})\\
&= \mathcal{C}(\nu^{co}_{k,p-k} \vee  \id_{\Delta^{q}})m(m(p_{k}\otimes p_{p-k})\mu_{k,p-k}\otimes p_{q})\xi_{p,q}\\
&= \mathcal{C}(\nu^{co}_{k,p-k} \vee  \id_{\Delta^{q}})m(\mathcal{C}(\nu_{k,p-k})\otimes \id_{\mathcal{C}_{\{0 < q\}}})(p_{p}\otimes p_{q})\xi_{p,q}\\
&= \mathcal{C}(\nu^{co}_{k,p-k} \vee  \id_{\Delta^{q}})\mathcal{C}(\nu_{k,p-k} \vee \id_{\mathcal{C}_{\{0 < q\}}})m(p_{p}\otimes p_{q})\xi_{p,q}\\
&= m(p_{p}\otimes p_{q})\xi_{p,q}
\end{align*}
Similarly, $\mathcal{C}(\nu_{p,q})\mathcal{C}(\nu^{co}_{k,n-k})m(p_{k}\otimes p_{l}) = m(p_{p}\otimes p_{q})\xi_{p,q}$ also holds when $k > p$ or $k = p$. Hence, by the construction of $\mathcal{C}^{temp}_{n}$, there is a unique morphism
$$
Z^{k,l}: \mathcal{C}^{temp}_{k}\otimes \mathcal{C}^{temp}_{l}\rightarrow \mathcal{C}^{temp}_{n}
$$
such that $p_{n}Z^{k,l} = \mathcal{C}(\nu^{co}_{k,n-k})m(p_{k}\otimes p_{l})$ and $\mu_{p,q}Z^{k,l} = \xi_{p,q}$ for all $p,q > 0$ with $p + q = n$. In particular, the Frobenius identities \eqref{equation: Frobenius identities} are satisfied.

A similar argument from induction shows that the morphisms $Z^{k,l}$ are natural in $k,l\geq 0$ and satisfy associativity. Hence we obtain a Frobenius templicial object $\mathcal{C}^{temp}$.

Finally, a similar argument shows that for any functor $H: \mathcal{C}\rightarrow \mathcal{D}$ of $\mathcal{V}^{\overline{\nec}^{op}}$-categories, the templicial morphism $res_{i}(H)^{temp}: res_{i}(\mathcal{C})^{temp}\rightarrow res_{i}(\mathcal{D})^{temp}$ respects the Frobenius structures and thus lifts to a morphism in $\Fs\mathcal{V}$.
\end{Con}

\begin{Con}
We construct a functor
$$
(-)^{nec}: \Fs\mathcal{V}\rightarrow \mathcal{V}^{\overline{\nec}^{op}}\Cat
$$
which lifts $(-)^{nec}$ of \eqref{diagram: nec-temp adjunction} along $\Fs\mathcal{V}\rightarrow \ts\mathcal{V}$ and $res_{i}: \mathcal{V}\Cat_{\overline{\nec}}\rightarrow \mathcal{V}\Cat_{\nec}$.

Let $X$ be a Frobenius templicial object. Let $S$ denote the set of vertices of $X$ and $Z$ denote the Frobenius structure. In view of Remark \ref{remark: unique decomp. of extended necklace map}, it suffices to define a quiver morphism $X^{nec}_{T}\rightarrow X^{nec}_{U}$ for any coinert map $(\id_{[p]},T): U\rightarrow T$ in $\overline{\nec}$. Let $(T_{i},u_{i}-u_{i-1})$ be unique such that $T = T_{1}\vee \dots \vee T_{l}$, where $U = \{0 = u_{0} < u_{1} < \dots < u_{l} = p\}\subseteq T$. Then define
$$
X^{nec}(\id_{[p]},U): X_{T}\xrightarrow{Z^{T_{1}}\otimes_{S} ...\otimes_{S} Z^{T_{l}}} X_{U}
$$
Then it follows from \cite[Proposition 2.17]{lowen2023frobenius} that we have a well-defined functor $X^{nec}: \overline{\nec}^{op}\rightarrow \mathcal{V}\Quiv_{S}$. Moreover, with compositions determined by the quiver isomorphisms $m_{T,U}: X^{nec}_{T}\otimes_{S} X^{nec}_{U}\xrightarrow{\sim} X^{nec}_{T\vee U}$, it is clear that we obtain a $\mathcal{V}^{\overline{\nec}^{op}}$-enriched category $X^{nec}$ with object set $S$.

Moreover, given a Frobenius templicial morphism $\alpha: X\rightarrow Y$, it immediately follows from the definitions that $\alpha^{nec}$ lifts to a $\mathcal{V}^{\overline{\nec}^{op}}$-enriched functor $X^{nec}\rightarrow Y^{nec}$.
\end{Con}

\begin{Thm}\label{theorem: Frob. temp. nec. adjunction}
The adjunction $(-)^{nec}\dashv (-)^{temp}$ \eqref{diagram: nec-temp adjunction} lifts to an adjunction
$$
(-)^{nec}: \Fs\mathcal{V}\leftrightarrows \mathcal{V}\Cat_{\overline{\nec}}: (-)^{temp}
$$
along the forgetful functors $\Fs\mathcal{V}\rightarrow \ts\mathcal{V}$ and $res_{i}: \mathcal{V}\Cat_{\overline{\nec}}\rightarrow \mathcal{V}\Cat_{\nec}$.
\end{Thm}
\begin{proof}
Note that both the forgetful functors $\Fs\mathcal{V}\rightarrow \ts\mathcal{V}$ and $res_{i}: \mathcal{V}\Cat_{\overline{\nec}}\rightarrow \mathcal{V}\Cat_{\nec}$ are clearly faithful. Thus to verify the adjunction, it suffices to show that both the unit and counit of the adjunction \eqref{diagram: nec-temp adjunction} are morphisms in $\Fs\mathcal{V}$ and $\mathcal{V}\Cat_{\overline{\nec}}$ respectively. This follows from the constructions above.
\end{proof}

\begin{Prop}\label{proposition: nec functor for Frob. temp. obj. fully faithful}
The functor $(-)^{nec}: \Fs\mathcal{V}\rightarrow \mathcal{V}\Cat_{\overline{\nec}}$ of Theorem \ref{theorem: Frob. temp. nec. adjunction} is fully faithful.
\end{Prop}
\begin{proof}
Let $X$ and $Y$ be Frobenius templicial objects and $H: X^{nec}\rightarrow Y^{nec}$ a morphism in $\mathcal{V}\Cat_{\overline{\nec}}$. As $(-)^{nec}: \ts\mathcal{V}\rightarrow \mathcal{V}\Cat_{\nec}$ is fully faithful (Theorem \ref{theorem: nec-temp adjunction}), there exists a unique templicial morphism $\alpha: X\rightarrow Y$ such that $\alpha^{nec} = res_{i}(H)$. Thus it suffices to show that $\alpha$ preserves the Frobenius structures of $X$ and $Y$. But this immediately follows from the compatibility of $H$ with the coinert maps $\nu^{co}_{k,l}: \Delta^{k+l}\rightarrow \Delta^{k}\vee \Delta^{l}$ for all $k,l > 0$.
\end{proof}

\begin{Cor}\label{corollary: D-nerve has Frob. structure}
Let $D: \nec\rightarrow \mathcal{W}$ be a colax monoidal functor. If $D$ extends to a colax monoidal functor $\overline{\nec}\rightarrow \mathcal{W}$, then $N^{D}_{\mathcal{V}}$ factors through the forgetful functor $\Fs\mathcal{V}\rightarrow \ts\mathcal{V}$. In particular, $N^{D}_{\mathcal{V}}(\mathcal{C})$ has a Frobenius structure for every $\mathcal{W}$-category $\mathcal{C}$.
\end{Cor}
\begin{proof}
Let $\overline{D}: \overline{\nec}\rightarrow \mathcal{W}$ be a colax monoidal functor extending $D$. Similarly to \S\ref{subsection: A general procedure}, we have an induced lax monoidal right-adjoint $\mathfrak{n}^{\overline{D}}_{\mathcal{V}}: \mathcal{W}\rightarrow \mathcal{V}^{\overline{\nec}^{op}}$ given by $\mathfrak{n}^{D}_{\mathcal{V}}(W)_{T} = [\overline{D}(T),W]$ for all $T\in \nec$ and $W\in \mathcal{W}$. Then clearly $res_{i}\circ \mathfrak{n}^{\overline{D}}_{\mathcal{V}} = \mathfrak{n}^{D}_{\mathcal{V}}$. Hence, it follows from Theorem \ref{theorem: Frob. temp. nec. adjunction} that $N^{D}_{\mathcal{V}} = (-)^{temp}\circ \mathfrak{n}^{D}_{\mathcal{V}}$ factors through the forgetful functor $\Fs\mathcal{V}\rightarrow \ts\mathcal{V}$.
\end{proof}

\begin{Cor}\label{corollary: cubical nerve has Frob. structure}
Let $D: \nec\rightarrow \mathcal{W}$ be a colax monoidal functor. If $D$ factors through $\dim: \nec\rightarrow \square$, then $N^{D}_{\mathcal{V}}$ factors through the forgetful functor $\Fs\mathcal{V}\rightarrow \ts\mathcal{V}$. In particular, for any small $\mathcal{W}$-category $\mathcal{C}$, $N^{D}_{\mathcal{V}}(\mathcal{C})$ has a Frobenius structure.
\end{Cor}
\begin{proof}
By Corollary \ref{corollary: D-nerve has Frob. structure}, it suffices to show that $\dim: \nec\rightarrow \square$ extends to $\overline{\nec}$. We use the presentation of Definition \ref{definition: dimension of a necklace}.1. Given a morphism $(f,U'): T\rightarrow U$ in $\overline{\nec}$, set
$$
\dim(f,U'): \mathcal{P}_{T}\rightarrow \mathcal{P}_{U}: T'\mapsto f(T')\cup U'
$$
which clearly extends $\dim$ and is still strong monoidal. It remains to show that $\dim(f,U')$ belongs to $\square$. By Remark \ref{remark: unique decomp. of extended necklace map}, it suffices to check this for $\nu^{co}_{k,n-k}: \Delta^{n}\rightarrow \Delta^{k}\vee \Delta^{n-k}$ with $0 < k < n$. But under \eqref{diagram: equiv. descriptions of cubes}, we have $\dim(\nu^{co}_{k,n-k}) = \sigma_{k}: [1]^{n-1}\rightarrow [1]^{n-2}$ .
\end{proof}

\subsection{Comparison maps}\label{subsection: Comparison maps}

In this subsection, we construct comparisons between nerves where we let the enriching category $\mathcal{W}$ (Proposition \ref{proposition: basic comparison of nerves}), the generating diagram $D: \nec\rightarrow \mathcal{W}$, and the $\mathcal{W}$-category $\mathcal{C}$ (Theorem \ref{theorem: pull-back product of nerves contractible}) vary.

We define the templicial analogue of a trivial Kan fibration in Definition \ref{definition: tivial fibration} and show that the forgetful functor $\tilde{U}: \ts\mathcal{V}\rightarrow \SSet$ preserves such trivial fibrations in Corollary \ref{corollary: underlying simp. set functor preserves trivial fibrations}. In Theorem \ref{theorem: pull-back product of nerves contractible}, we provide conditions in terms of the generating diagrams $D$ such that the induced comparison templicial morphism is a trivial fibration.

\begin{Prop}\label{proposition: basic comparison of nerves}
Let $\mathcal{W}'$ be another $\mathcal{V}$-enriched monoidal category as in \S\ref{subsection: Notations and conventions}.2, and let $L: \mathcal{W}\leftrightarrows \mathcal{W}': R$ be a $\mathcal{V}$-enriched adjunction such that $L$ is colax monoidal. For any colax monoidal diagram $D: \nec\rightarrow \mathcal{W}$, we have a natural isomorphism:
$$
N^{D}_{\mathcal{V}}\circ (-)^{R}\simeq N^{LD}_{\mathcal{V}}
$$
where $(-)^{R}: \mathcal{W}'\Cat\rightarrow \mathcal{W}\Cat$ denotes the functor applying the $R$ to hom-objects.
\end{Prop}
\begin{proof}
By Construction \ref{construction: nerve generated by necklicial diagram}, it suffices to show $\mathfrak{n}^{D}_{\mathcal{V}}\circ R\simeq \mathfrak{n}^{LD}_{\mathcal{V}}$.  Given $W\in \mathcal{W}'$, we have
$$
\mathfrak{n}^{D}_{\mathcal{V}}(R(W)) = [D(-),R(W)]\simeq [LD(-),W]' = \mathfrak{n}^{LD}_{\mathcal{V}}(W)
$$
where $[-,-]': \mathcal{W}'\times \mathcal{W}'\rightarrow \mathcal{V}$ denotes the $\mathcal{V}$-enrichment of $\mathcal{W}'$.
\end{proof}

Before we can discuss further comparison maps and explain how they lift trivial Kan fibrations of simplicial sets, we require some lemmas.

Consider the functor $\partial\Delta_{\bullet}(0,n)\in \Set^{\nec^{op}}$ for $n > 0$, where $\partial\Delta^{n}$ denotes the boundary of the standard simplex. Similar to \cite[Proposition 5.1]{lowen2024enriched}, we have
\begin{equation}\label{equation: cells as union of necklaces}
\partial\Delta_{\bullet}(0,n) = \bigcup_{i=1}^{n-1}\delta_{i}(\Delta^{n-1})_{\bullet}(0,n)\cup \bigcup_{k=1}^{n-1}(\Delta^{k}\vee \Delta^{n-k})_{\bullet}(0,n)
\end{equation}
as subfunctors of $\Delta^{n}_{\bullet}(0,n)$.

\begin{Lem}\label{lemma: horn subclass of cell}
The closure of $\left\lbrace\partial\Delta^{n}_{\bullet}(0,n)\hookrightarrow \Delta^{n}_{\bullet}(0,n)\,\middle\vert\, n > 0\right\rbrace$ under pushouts and compositions contains all inner horn inclusions $(\Lambda^{n}_{j})_{\bullet}(0,n)\hookrightarrow \Delta^{n}_{\bullet}(0,n)$ for $0 < j < n$.
\end{Lem}
\begin{proof}
Note that $\delta_{j}(\partial\Delta^{n-1})_{\bullet}(0,n) = (\Lambda^{n}_{j})_{\bullet}(0,n)\cap (\delta_{j}(\Delta^{n-1}))_{\bullet}(0,n)$ and thus from \eqref{equation: cells as union of necklaces} we have a pushout:
\[\begin{tikzcd}
	{\partial\Delta^{n-1}_{\bullet}(0,n)} & {(\Lambda^{n}_{j})_{\bullet}(0,n)} \\
	{\Delta^{n-1}_{\bullet}(0,n)} & {\partial\Delta^{n}_{\bullet}(0,n)}
	\arrow[from=1-1, to=2-1]
	\arrow["{\delta_{j}}", from=1-1, to=1-2]
	\arrow[from=1-2, to=2-2]
	\arrow["{\delta_{j}}"', from=2-1, to=2-2]
\end{tikzcd}\]
It now suffices to note that the morphism $(\Lambda^{n}_{j})_{\bullet}(0,n)\hookrightarrow \Delta^{n}_{\bullet}(0,n)$ is the composition of the right vertical morphism with $\partial\Delta^{n}_{\bullet}(0,n)\hookrightarrow \Delta^{n}_{\bullet}(0,n)$.
\end{proof}

The following lemma is a strict improvement of \cite[Lemma 5.9]{lowen2024enriched}.

\begin{Lem}\label{lemma: necklace temp. adjunction counit trivial fibration}
Let $\mathcal{C}$ be a necklace category with objects $A$ and $B$. Consider the canonical morphism $\epsilon: \mathcal{C}^{temp}_{\bullet}(A,B)\rightarrow \mathcal{C}_{\bullet}(A,B)$ induced by the counit of the adjunction of \eqref{diagram: nec-temp adjunction}. Given integers $0 < j < n$, any lifting problem in $\mathcal{V}^{\nec^{op}}$:
\[\begin{tikzcd}
	{\tilde{F}(\partial\Delta^{n})_{\bullet}(0,n)} & {\mathcal{C}^{temp}_{\bullet}(A,B)} \\
	{\tilde{F}(\Delta^{n})_{\bullet}(0,n)} & {\mathcal{C}_{\bullet}(A,B)}
	\arrow[from=1-1, to=2-1]
	\arrow[from=1-1, to=1-2]
	\arrow[from=2-1, to=2-2]
	\arrow["\epsilon", from=1-2, to=2-2]
	\arrow[dashed, from=2-1, to=1-2]
\end{tikzcd}\]
has a unique solution.
\end{Lem}
\begin{proof}
Let us denote the composition of $\mathcal{C}$ by $m$, and the inner face, degeneracy and comultiplication morphisms of $\mathcal{C}^{temp}$ by $d_{j}$, $s_{i}$ and $\mu_{k,l}$ respectively. Now by \eqref{equation: cells as union of necklaces}, the top horizontal morphism in the lifting problem above corresponds to some collections of elements $(x_{k})_{k=1}^{n-1}$ and $(y_{i})_{i=1}^{n-1}$ with $x_{k}\in U((\mathcal{C}^{temp}_{k}\otimes \mathcal{C}^{temp}_{n-k})(a,b))$ and $y_{i}\in U(\mathcal{C}^{temp}_{n-1}(a,b))$, satisfying:
$$
d_{j-1}(y_{i}) = d_{i}(y_{j}),\quad \text{and}\quad (\id_{X_{k}}\otimes \mu_{l-k,n-l})(x_{k}) = (\mu_{k,l-k}\otimes \id_{X_{n-l}})(x_{l})
$$
for all $0 < i < j < n$ and $0 < k < l < n$, as well as
$$
\mu_{k,n-k-1}(y_{j}) =
\begin{cases}
(d_{j}\otimes \id_{X_{n-k-1}})(x_{k+1}) & \text{if }j\leq k\\
(\id_{X_{k}}\otimes d_{j-k})(x_{k}) & \text{if }j > k
\end{cases}
$$
for all $0< j < n$ and $0 < k < n-1$. Moreover, the bottom horizontal morphism corresponds to an element $z'\in U(\mathcal{C}_{\{0 < n\}}(a,b))$ and the commutativity of the diagram comes down to the condition that $\mathcal{C}(\nu_{k,n-k})(z') = m(p_{k}\otimes p_{n-k})(x_{k})$ and $\mathcal{C}(\delta_{i})(z') = p_{n-1}(y_{i})$ for all $0 < k,i < n$.

Then by definition of the functor $(-)^{temp}$ (see \cite[Construction 3.11]{lowen2024enriched}), there exists a unique element $z\in U(\mathcal{C}^{temp}_{n}(a,b))$ such that $\mu_{k,n-k}(z) = x_{k}$ for all $0 < k < n$, and $p_{n}(z) = z'$. Moreover, we have that $d_{i}(z) = y_{i}$ for all $0 < i < n$. Indeed, again by the definition of $(-)^{temp}$, it suffices to note that for all $0 < k,j < n$:
\begin{align*}
&\mu_{k,n-1-k}(d_{j}(z)) =
\begin{cases}
(d_{i}\otimes \id_{\mathcal{C}^{temp}_{n-k-1}})(\mu_{k+1,n-k}(z)) & \text{if }j\leq k\\
(\id_{\mathcal{C}^{temp}_{k}}\otimes d_{j-k})(\mu_{k,n-k}(z)) & \text{if }j > k
\end{cases}
= \mu_{k,n-1-k}(y_{i})\\
&p_{n-1}(d_{j}(z)) = \mathcal{C}(\delta_{j})p_{n}(z) = \mathcal{C}(\delta_{j})(z') = p_{n-1}(y_{j})
\end{align*}

Hence, the element $z$ determines a morphism $\tilde{F}(\Delta^{n})_{\bullet}(0,n)\rightarrow \mathcal{C}^{temp}_{\bullet}(a,b)$ which is a lift of the above diagram.
\end{proof}

\begin{Def}\label{definition: tivial fibration}
We say that a morphism in $\mathcal{V}^{\nec^{op}}$ \emph{lifts cells} if it has the right lifting property with respect to all boundary inclusions $\partial\Delta^{n}_{\bullet}(0,n)\hookrightarrow \Delta^{n}_{\bullet}(0,n)$ for $n > 0$.

A templicial morphism $(\alpha,f): (X,S)\rightarrow (Y,T)$ is a \emph{trivial fibration} if
\begin{enumerate}[(a)]
\item the map $f: S\rightarrow T$ is surjective,
\item for all $a,b\in S$, the induced morphism $X_{\bullet}(a,b)\rightarrow Y_{\bullet}(f(a),f(b))$ in $\mathcal{V}^{\nec^{op}}$ lifts cells.
\end{enumerate}
\end{Def}

\begin{Prop}
A simplicial map in $\SSet$ is a trivial Kan fibration if and only if it is a trivial fibration (in the sense of Definition \ref{definition: tivial fibration}).
\end{Prop}
\begin{proof}
Let $K$ be a simplicial set, considered as a templicial set, and let $n > 0$ be an integer. Then the assignment $\mathcal{S}_{1}\rightarrow \mathcal{S}_{2}: (x_{k})_{k}\mapsto (\pi_{1}(x_{n-1}),\pi_{2}(x_{1}))$ between the sets
\begin{align*}
\mathcal{S}_{1} &= \left\lbrace (x_{k}\in X_{k}\times_{X_{0}} X_{n-k})_{k=1}^{n-1}\,\middle\vert\, \forall 0 < k < l < n: (\id\times \mu_{l-k,n-l})(x_{k}) = (\mu_{k,l-k}\times \id)(x_{l})\right\rbrace,\\
\mathcal{S}_{2} &= \left\lbrace (y_{n},y_{0})\in X_{n-1}\times X_{n-1}\,\middle\vert\, d_{n-1}(y_{0}) = d_{0}(y_{n})\right\rbrace
\end{align*}
defines a bijection. Then it follows by \eqref{equation: cells as union of necklaces} that for any $a,b\in K_{0}$ a map $\partial\Delta^{n}_{\bullet}(0,n)\rightarrow K_{\bullet}(a,b)$ in $\Set^{\nec^{op}}$ is equivalent to a map $\partial\Delta^{n}_{0,n}\rightarrow K_{a,b}$ in $\SSet_{*,*}$. By Yoneda's lemma, a map $\Delta^{n}_{\bullet}(0,n)\rightarrow K_{\bullet}(a,b)$ is also equivalent to an element of $K_{n}(a,b)$ and thus to a map $\Delta^{n}_{0,n}\rightarrow K_{a,b}$ in $\SSet_{*,*}$.

From this it follows that a simplicial map has the right lifting property with respect to all cell inclusions $\partial\Delta^{n}\rightarrow \Delta^{n}$ in $\SSet$ for $n > 0$ if and only if it satisfies $(b)$ of Definition \ref{definition: tivial fibration}. Further it has the right lifting property with respect to $\emptyset\rightarrow \Delta^{0}$ if and only if it satisfies $(a)$ of Definition \ref{definition: tivial fibration}.
\end{proof}

\begin{Prop}\label{proposition: temp preserves triv. fibs.}
Let $H: \mathcal{C}\rightarrow \mathcal{D}$ be a necklace functor in $\mathcal{V}\Cat_{\nec}$. Suppose $H$ is surjective on objects and for all $A,B\in \Ob(\mathcal{C})$, the morphism
$$
H_{A,B}: \mathcal{C}(A,B)\rightarrow \mathcal{D}(f(A),f(B))
$$
in $\mathcal{V}^{\nec^{op}}$ lifts cells. Then $H^{temp}: \mathcal{C}^{temp}\rightarrow \mathcal{D}^{temp}$ in $\ts\mathcal{V}$ is a trivial fibration.
\end{Prop}
\begin{proof}
This is an immediate consequence of Lemma \ref{lemma: necklace temp. adjunction counit trivial fibration}.
\end{proof}

\begin{Cor}\label{corollary: underlying simp. set functor preserves trivial fibrations}
If a templicial morphism $\alpha$ is a trivial fibration, then the simplicial map $\tilde{U}(\alpha)$ is a trivial Kan fibration.
\end{Cor}
\begin{proof}
From \cite[Proposition 3.14]{lowen2024enriched}, we have that $\tilde{U}\simeq (-)^{temp}\circ \mathcal{U}\circ (-)^{nec}$, where $\mathcal{U}: \mathcal{V}\Cat_{\nec}\rightarrow \Cat_{\nec}$ is the forgetful functor. Thus it suffices to check that $\mathcal{U}(\alpha^{nec})$ satisfies the condition of Proposition \ref{proposition: temp preserves triv. fibs.}, which is true by definition.
\end{proof}

Recall the functor
$$
\Colax(\nec,\mathcal{W})^{op}\rightarrow \Fun(\mathcal{W}\Cat,\ts\mathcal{V}): D\mapsto N^{D}_{\mathcal{V}}
$$
of Construction \ref{construction: nerve generated by necklicial diagram}. It provides for any $\mathcal{W}$-category $\mathcal{C}$ and any monoidal natural transformation $D'\rightarrow D$, a comparison morphism $N^{D}_{\mathcal{V}}(\mathcal{C})\rightarrow N^{D'}_{\mathcal{V}}(\mathcal{C})$ in $\ts\mathcal{V}$.

\begin{Thm}\label{theorem: pull-back product of nerves contractible}
Let $\alpha: D'\rightarrow D$ be a monoidal natural transformation between colax monoidal functors $D,D': \nec\rightarrow \mathcal{W}$ and let $H: \mathcal{C}\rightarrow \mathcal{D}$ be a $\mathcal{W}$-enriched functor between small $\mathcal{W}$-categories. Assume that $H$ is surjective on objects and that the following lifting problem in $\mathcal{W}$ has a solution for all $n > 0$ and $A,B\in \Ob(\mathcal{C})$:
\[\begin{tikzcd}
	{\mathfrak{l}^{D}(\partial\Delta^{n}_{\bullet}(0,n))\amalg_{\mathfrak{l}^{D'}(\partial\Delta^{n}_{\bullet}(0,n))}D'(\Delta^{n})} & {\mathcal{C}(A,B)} \\
	{D(\Delta^{n})} & {\mathcal{D}(H(A),H(B))}
	\arrow[from=1-1, to=2-1]
	\arrow[from=1-2, to=2-2]
	\arrow[from=1-1, to=1-2]
	\arrow[from=2-1, to=2-2]
	\arrow[dashed, from=2-1, to=1-2]
\end{tikzcd}\]
Then the induced templicial morphism $N^{D'}_{\mathcal{\mathcal{V}}}(\mathcal{C})\rightarrow N^{D}_{\mathcal{V}}(\mathcal{C})\times_{N^{D}_{\mathcal{V}}(\mathcal{D})} N^{D'}_{\mathcal{V}}(\mathcal{D})$ is a trivial fibration. In particular, $N^{D'}(\mathcal{C})\rightarrow N^{D}(\mathcal{C})\times_{N^{D}(\mathcal{D})} N^{D'}(\mathcal{D})$ is a trivial Kan fibration.
\end{Thm}
\begin{proof}
By the adjunction \eqref{diagram: adjunction induced by strong monoidal necklace diagram}, the above lifting problem is equivalent to
\[\begin{tikzcd}
	{F(\partial\Delta^{n}_{\bullet}(0,n))} & {\mathfrak{n}^{D'}_{\mathcal{V}}(\mathcal{C})(A,B)} \\
	{F(\Delta^{n}_{\bullet}(0,n))} & {\mathfrak{n}^{D}_{\mathcal{V}}(\mathcal{C})(A,B)\times_{\mathfrak{n}^{D}_{\mathcal{V}}(\mathcal{D})(H(A),H(B))} \mathfrak{n}^{D'}_{\mathcal{V}}(\mathcal{D})(H(A),H(B))}
	\arrow[from=1-1, to=2-1]
	\arrow[from=1-1, to=1-2]
	\arrow[from=1-2, to=2-2]
	\arrow[from=2-1, to=2-2]
	\arrow[dashed, from=2-1, to=1-2]
\end{tikzcd}\]
Thus the necklace functor $\mathfrak{n}^{D'}_{\mathcal{V}}(\mathcal{C})\rightarrow \mathfrak{n}^{D}_{\mathcal{V}}(\mathcal{C})\times_{\mathfrak{n}^{D}_{\mathcal{V}}(\mathcal{D})}\mathfrak{n}^{D'}_{\mathcal{V}}(\mathcal{D})$ satisfies the conditions of  Proposition \ref{proposition: temp preserves triv. fibs.}. So the templicial morphism $N^{D'}_{\mathcal{\mathcal{V}}}(\mathcal{C})\rightarrow N^{D}_{\mathcal{V}}(\mathcal{C})\times_{N^{D}_{\mathcal{V}}(\mathcal{D})} N^{D'}_{\mathcal{V}}(\mathcal{D})$ is a trivial fibration as $(-)^{temp}$ preserves limits. The final statement follows from Corollary \ref{corollary: underlying simp. set functor preserves trivial fibrations}.
\end{proof}

\begin{Cor}\label{corollary: trivial fib. between nerves}
Let $\alpha: D'\rightarrow D$ be a monoidal natural transformation between colax monoidal functors $D,D': \nec\rightarrow \mathcal{W}$ and let $H: \mathcal{C}\rightarrow \mathcal{D}$ be a $\mathcal{W}$-enriched functor between small $\mathcal{W}$-categories.
\begin{enumerate}[1.]
\item Assume that $H$ is surjective on objects and the following lifting problem in $\mathcal{W}$ has a solution for all $n > 0$ and $A,B\in \Ob(\mathcal{C})$:
\[\begin{tikzcd}
	{\mathfrak{l}^{D}(\partial\Delta^{n}_{\bullet}(0,n))\amalg_{\mathfrak{l}^{D'}(\partial\Delta^{n}_{\bullet}(0,n))}D'(\Delta^{n})} & {\mathcal{C}(A,B)} \\
	{D(\Delta^{n})} & {\mathcal{D}(H(A),H(B))}
	\arrow[from=1-1, to=2-1]
	\arrow[from=1-2, to=2-2]
	\arrow[from=1-1, to=1-2]
	\arrow[from=2-1, to=2-2]
	\arrow[dashed, from=2-1, to=1-2]
\end{tikzcd}\]
Then the induced templicial morphism $N^{D}_{\mathcal{V}}(H): N^{D}_{\mathcal{V}}(\mathcal{C})\rightarrow N^{D}_{\mathcal{V}}(\mathcal{D})$ is a trivial fibration. In particular, $N^{D}(H): N^{D}(\mathcal{C})\rightarrow N^{D}(\mathcal{D})$ is a trivial Kan fibration.
\item Assume that $\mathcal{W}$ has a terminal object and that the following lifting problem in $\mathcal{W}$ has a solution for all $n > 0$ and $A,B\in \Ob(\mathcal{C})$:
\[\begin{tikzcd}
	{\mathfrak{l}^{D}(\partial\Delta^{n}_{\bullet}(0,n))\amalg_{\mathfrak{l}^{D'}(\partial\Delta^{n}_{\bullet}(0,n))}D'(\Delta^{n})} & {\mathcal{C}(A,B)} \\
	{D(\Delta^{n})}
	\arrow[from=1-1, to=2-1]
	\arrow[from=1-1, to=1-2]
	\arrow[dashed, from=2-1, to=1-2]
\end{tikzcd}\]
Then the induced templicial morphism $N^{D}_{\mathcal{V}}(\mathcal{C})\rightarrow N^{D'}_{\mathcal{V}}(\mathcal{C})$ is a trivial fibration. In particular, $N^{D}(\mathcal{C})\rightarrow N^{D'}(\mathcal{C})$ is a trivial Kan fibration.
\end{enumerate}
\end{Cor}
\begin{proof}
This follows from Theorem \ref{theorem: pull-back product of nerves contractible} by choosing $D' = 0$ (the constant functor on the initial object in $\mathcal{W}$) and $\mathcal{D} = 1$ (the terminal object in $\mathcal{W}\Cat$) respectively.
\end{proof}

\section{Examples}\label{section: Examples}

In this section we discuss several examples of nerves from the literature, as well as two examples of interest for the study of general templicial objects. In each subsection, we first identify the generating diagram $D: \nec\rightarrow \mathcal{W}$ and then apply the results from Section \ref{section: Enriched nerves of enriched categories} whenever possible.

It should be noted that many of the results in this section have already appeared in the literature, and we do not claim originality for them. Often we will still reprove them however, to show how they follow by using the generating diagram $D$. Our main novel results are Corollaries \ref{corollary: results for the dg-nerve}.2 and \ref{corollary: results for the cubical nerve}.2, which give explicit descriptions of the left-adjoints of the differential graded and cubical nerves.

\subsection{The nerve}\label{subsection: The nerve}

Let us start with an easy example. Recall the classical nerve functor $N: \Cat\rightarrow \SSet$. In \cite[Definition 2.11]{lowen2024enriched} an enriched variant $N_{\mathcal{V}}: \mathcal{V}\Cat\rightarrow \ts\mathcal{V}$ was constructed, called the \emph{templicial nerve}. Let us now discuss how it fits into the general procedure of \S\ref{subsection: A general procedure}. Consider the constant functor
\begin{equation}
\const_{I}: \nec\rightarrow \mathcal{V}: T\mapsto I
\end{equation}
which is clearly strong monoidal.

\begin{Prop}\label{proposition: templicial nerve}
The nerve $\mathcal{V}\Cat\rightarrow \ts\mathcal{V}$ generated by $\const_{I}: \nec\rightarrow \mathcal{V}$ is naturally isomorphic to the templicial nerve $N_{\mathcal{V}}$ of \cite{lowen2024enriched}.
\end{Prop}
\begin{proof}
From \cite[Proposition 3.16]{lowen2024enriched}, we have a natural isomorphism $N_{\mathcal{V}}\simeq (-)^{temp}\circ \const$, where $\const: \mathcal{V}\Cat\rightarrow \mathcal{V}\Cat_{\nec}$ is determined on hom-objects by the functor $\mathcal{V}\rightarrow \mathcal{V}^{\nec^{op}}$ sending every object $V$ of $\mathcal{V}$ to the constant functor on $V$. It is clear from the definitions that in fact $\mathfrak{n}^{\const_{I}}_{\mathcal{V}}\simeq \const$ and thus the result follows.
\end{proof}

The following results were already shown in \cite{lowen2024enriched} and \cite{lowen2023frobenius}, but they are now also simple consequences of the results of Section \ref{section: Enriched nerves of enriched categories}.

\begin{Cor}
The following statements are true.
\begin{enumerate}[1.]
\item There is a natural isomorphism $\tilde{U}\circ N_{\mathcal{V}}\simeq N\circ \mathcal{U}$ where $\mathcal{U}: \mathcal{V}\Cat\rightarrow \Cat$ is the forgetful functor. In particular, if $\mathcal{V} = \Set$, then $N_{\mathcal{V}}$ coincides with $N$.
\item $N_{\mathcal{V}}$ has a left-adjoint $h_{\mathcal{V}}: \ts\mathcal{V}\rightarrow \mathcal{V}\Cat$.
\item Let $\mathcal{C}$ be small $\mathcal{V}$-category. Then $N_{\mathcal{V}}(\mathcal{C})$ is a quasi-category in $\mathcal{V}$.
\item Let $\mathcal{C}$ be small $\mathcal{V}$-category. Then $N_{\mathcal{V}}(\mathcal{C})$ has a Frobenius structure.
\end{enumerate}
\end{Cor}
\begin{proof}
\begin{enumerate}[1.]
\item In view of Theorem \ref{theorem: underlying D-nerve is cat. nerve assoc. to D}, it suffices to show that the unenriched nerve $\mathcal{V}\Cat\rightarrow \SSet$ produced from $\const_{I}$ by Construction \ref{construction: cosimplicial object generated by necklicial diagram} coincides with $N\circ \mathcal{U}$. To this end, note that we have isomorphisms $\Phi(\const_{I})^{n}\simeq \mathcal{F}([n])$ which are natural in $n\geq 0$, where $\mathcal{F}: \Cat\rightarrow \mathcal{V}\Cat$ applies $F$ to hom-sets. Then we have for all small $\mathcal{V}$-categories $\mathcal{C}$:
$$
N(\mathcal{U}(\mathcal{C}))_{n}\simeq \Cat([n],\mathcal{U}(\mathcal{C}))\simeq \mathcal{V}\Cat(\mathcal{F}([n]),\mathcal{C})\simeq \mathcal{V}\Cat(\Phi(\const_{I})^{n},\mathcal{C})
$$
since $\mathcal{F}$ is left-adjoint to $\mathcal{U}$.
\item This follows from Proposition \ref{proposition: nerve gen. by strong mon. diagram has left-adj.} since $\const_{I}$ is strong monoidal.
\item By \eqref{equation: horn as colimit}, $\mathfrak{l}^{\const_{I}}((\Lambda^{n}_{j})_{\bullet}(0,n))$ is a connected colimit of copies of $I$ and thus itself isomorphic to $I$. The lifting diagram in Theorem \ref{theorem: D-nerve is quasi-cat.} for $D = \const_{I}$ thus has a trivial solution, whereby $N^{D}(\mathcal{C})$ is a quasi-category in $\mathcal{V}$ for any $\mathcal{V}$-category $\mathcal{C}$.
\item This immediately follows from Corollary \ref{corollary: D-nerve has Frob. structure}.
\end{enumerate}
\end{proof}

\begin{Rem}
Note that there is no functor $D': \nec_{-}\rightarrow \mathcal{V}$ such that $\const_{I}\simeq \Lan_{\iota}D'$ and thus the results of \S\ref{subsection: Explicitation of the left-adjoint} do not apply.
\end{Rem}

\subsection{The Duskin nerve}\label{subsection: The Duskin nerve}

Let us denote the category of small $2$-categories, that is strictly $\Cat$-enriched categories, by $\Cat_{2} = \Cat\Cat$. Consider the Duskin nerve of \cite[\S 6]{duskin2001simplicial}:
$$
N^{\Dusk}: \Cat_{2}\rightarrow \SSet
$$
In fact, this nerve is defined for bicategories, but our approach is limited to strictly enriched categories and thus we necessarily have to restrict to $2$-categories. We show how the Duskin nerve fits into the general procedure of \S\ref{subsection: A general procedure}.

Define a strong monoidal functor $\Dusk$ as the functor \eqref{diagram: partition functor}:
\begin{equation}
\Dusk: \nec\xrightarrow{\dim} \square\subseteq \Cat: T\mapsto \mathcal{P}_{T}
\end{equation}

\begin{Prop}
The nerve functor $\Cat_{2}\rightarrow \SSet$ generated by $\Dusk: \nec\rightarrow \Cat$ is naturally isomorphic to the Duskin nerve of \cite{duskin2001simplicial}. 
\end{Prop}
\begin{proof}
It was shown in \cite[Tag 00BF]{kerodon} that the restriction of the Duskin nerve to strict $2$-categories has a left-adjoint $\mathrm{Path}_{(2)}: \SSet\rightarrow 2\Cat$. For an integer $n\geq 0$, the $2$-category $\mathrm{Path}_{(2)}(\Delta^{n})$ has $[n]$ as its object set and for $i,j\in [n]$ its hom-object is the following poset ordered by inclusion:
$$
\mathrm{Path}_{(2)}(\Delta^{n})(i,j) = \left\lbrace\{i = i_{0} < \dots < i_{k} = j\}\subseteq [n]\,\middle\vert\, k\geq 0\right\rbrace
$$ 
with composition given by the union of subsets of $[n]$. Recall the functor $\Phi$ of Construction \ref{construction: cosimplicial object generated by necklicial diagram}. Clearly, $\mathrm{Path}_{(2)}(\Delta^{n})(i,j)\simeq \mathcal{P}_{\Delta^{j-i}} = \Dusk(\Delta^{j-i})$ and we have an isomorphism $\mathrm{Path}_{(2)}(\Delta^{n})\simeq \Phi(\Dusk)^{n}$ which is natural in $n$. The result follows from Theorem \ref{theorem: underlying D-nerve is cat. nerve assoc. to D}.
\end{proof}

\begin{Rem}\label{remark: convention for partitions}
In fact, Lurie defines the poset $\mathrm{Path}_{(2)}(\Delta^{n})(i,j)$ by \emph{reverse} inclusion. This is merely a convention, but it does play a role when comparing to other nerves. Here we have chosen to use the ordinary inclusion relation on $\mathrm{Path}_{(2)}(\Delta^{n})(i,j)$ and thus on $\mathcal{P}_{T}$, to bring it in accordance with the conventions of \cite{duskin2001simplicial}, \cite{dugger2011rigidification} and \cite{rivera2018cubical}.
\end{Rem}

\begin{Lem}\label{lemma: Duskin nerve quasi-cat.}
Let $\mathcal{G}$ be a groupoid. Then for all $0 < j < n$, the following lifting problem has a solution. Moreover, the solution is unique for $n\geq 3$.
\[\begin{tikzcd}
	{\mathfrak{l}^{\Dusk}((\Lambda^{n}_{j})_{\bullet}(0,n))} & {\mathcal{G}} \\
	{\Dusk(\Delta^{n})}
	\arrow[from=1-1, to=1-2]
	\arrow[from=1-1, to=2-1]
	\arrow[dashed, from=2-1, to=1-2]
\end{tikzcd}\]
\end{Lem}
\begin{proof}
By \eqref{diagram: equiv. descriptions of cubes}, we may identify $\Dusk(\Delta^{n})$ with the cube $[1]^{n-1}$ and by \eqref{equation: horn as colimit}, $H_{j,n} = \mathfrak{l}^{\Dusk}((\Lambda^{n}_{j})_{\bullet}(0,n))$ is the subcategory of $[1]^{n-1}$ which is precisely missing the face $\delta^{0}_{j}$. We distinguish some cases:
\begin{itemize}
\item If $n = 2$, then the lifting problem is equivalent to finding, for any $A\in \mathcal{G}$, a morphism in $\mathcal{G}$ with target $A$. For this we can always choose the identity on $A$.
\item If $n = 3$, then the lifting problem is equivalent to finding, for any
$$
A\xrightarrow{f_{1}} B_{1}\xrightarrow{g_{1}} C\qquad \text{and}\qquad B_{2}\xrightarrow{g_{2}} C
$$
in $\mathcal{G}$, a morphism $f_{2}$ in $\mathcal{G}$ such that $g_{2}\circ f_{2} = g_{1}\circ f_{1}$. Since $\mathcal{G}$ is a groupoid, $f_{2} = g_{2}^{-1}\circ g_{1}\circ f_{1}$ is the unique solution.
\item If $n = 4$, then $H_{j,4}$ contains all morphisms of the cube $[1]^{3}$, so any extension will automatically be unique. For the extension to exist, it suffices to show that the map $H_{j,4}\rightarrow \mathcal{G}$ sends the square of the missing face $\delta^{0}_{j}$ to a commuting square in $\mathcal{G}$. This easily follows because the other morphisms in the image of $H_{j,4}\rightarrow \mathcal{G}$ are invertible, since $\mathcal{G}$ is a groupoid.
\item If $n > 4$, then $H_{j,n}$ contains all morphisms of the cube $[1]^{n-1}$ and any commuting diagram in $[1]^{n-1}$ also commutes in $H_{j,n}$. Hence the the inclusion $H_{j,n}\hookrightarrow [1]^{n}$ is an isomorphism and thus the lifting problem has a unique solution.
\end{itemize}
\end{proof}

\begin{Rem}
Alternatively, to prove Lemma \ref{lemma: Duskin nerve quasi-cat.}, one can apply the nerve $N: \Cat\rightarrow \SSet$ to the lifting problem and note that the map $N(H_{j,n})\hookrightarrow N([1]^{n-1})$ is a monomorphism and a weak homotopy equivalence since both the target and source are weakly contractible. Then as $\mathcal{G}$ is a groupoid, $N(\mathcal{G})$ is a Kan complex and thus the lift exists in the Quillen model structure on $\SSet$. Since $N$ is fully faithful, this also provides a lift in $\Cat$.
\end{Rem}

In the following corollary, statement 1 was already shown in \cite[Tag 00JL]{kerodon} for directed graphs of posets, but is now extended to arbitrary simplicial sets. Statement 2 was originally shown in \cite[Theorem 8.6]{duskin2001simplicial}, but it is now also simple consequences of the Theorem \ref{theorem: D-nerve is quasi-cat.}.

\begin{Cor}
The following statements are true.
\begin{enumerate}[1.]
\item The Duskin nerve $N^{\Dusk}: \Cat_{2}$ has a left-adjoint $L^{Dusk}: \SSet\rightarrow \Cat_{2}$. Moreover, for any simplicial set $X$, the underlying category of $L^{\Dusk}(X)$ has hom-sets given by, for all $a,b\in X_{0}$:
$$
\Ob(L^{\Dusk}(X)(a,b))\simeq \coprod_{\substack{T\in \nec\\ \dim(T) = 0}}X^{nd}_{T}(a,b)\simeq \coprod_{p\geq 0}(X^{nd}_{1})^{\times_{X_{0}} p}(a,b)
$$
\item Let $\mathcal{C}$ be a small $2$-category such that all $2$-morphisms are invertible. Then $N^{\Dusk}(\mathcal{C})$ is a quasi-category.
\item Let $\mathcal{C}$ be a small $2$-category. Then $N^{\Dusk}(\mathcal{C})$ has a Frobenius structure.
\end{enumerate}
\end{Cor}
\begin{proof}
\begin{enumerate}[1.]
\item By Theorem \ref{theorem: explicitation of left-adjoint for free temp. obj.}, it suffices to show that $\Ob\circ \Dusk\simeq \Lan_{\iota}D'$ where $\iota: \nec_{-}\hookrightarrow \nec$ is the inclusion, and $D'(T) = \{*\}$ when $\dim(T) = 0$ and $D'(T) = \emptyset$ otherwise. Fix a necklace $(T,p)$. Note that by Lemma \ref{lemma: left Kan ext. of surj. necklace maps}, we can identify $(\Lan_{\iota}D')(T)$ with the set of all injective maps $g: [q]\hookrightarrow [p]$ in $\fint$ such that $T\subseteq g([q])$.

Let us define a bijection, where we denote $U\in \mathcal{P}_{T}$ by $\{0 = u_{0} < u_{1} < \dots < u_{l} = p\}$:
$$
\varphi_{T}: \Ob(\mathcal{P}_{T})\rightarrow (\Lan_{\iota}D')(T): U\mapsto \left([l]\hookrightarrow [p]: j\mapsto u_{j}\right)
$$
Its inverse sends an element $g: [q]\hookrightarrow [p]$ of $(\Lan_{\iota}D')(T)$ to the subset $g([q])\in \mathcal{P}_{T}$. It is an easy verification to see that this bijection is also natural in $T\in \nec$. 
\item This immediately follows from Lemma \ref{lemma: Duskin nerve quasi-cat.} and Theorem \ref{theorem: D-nerve is quasi-cat.}.
\item This immediately follows from Corollary \ref{corollary: cubical nerve has Frob. structure}.
\end{enumerate}
\end{proof}

\begin{Rem}
Note that we cannot apply Theorem \ref{theorem: explicitation of left-adjoint for free temp. obj.} to describe the $2$-morphisms of $L^{\Dusk}(X)$, that is $\Mor(L^{\Dusk}(X)(a,b))$, since $\Mor: \Cat\rightarrow \Set$ doesn't preserve colimits.
\end{Rem}

\begin{Cor}\label{corollary: comparison of Duskin and ordinary nerve}
Let $\mathrm{disc}: \Set\hookrightarrow \Cat$ be the inclusion of discrete categories. Then there is a natural isomorphism $N^{Dusk}\circ (-)^{\mathrm{disc}}\simeq N$.
\end{Cor}
\begin{proof}
Note that $\mathrm{disc}$ is right-adjoint to the functor taking connected components $\pi_{0}: \Cat\rightarrow \Set$. By Proposition \ref{proposition: basic comparison of nerves}, it suffices to note that for every $T\in \nec$, $Dusk(T) = \mathcal{P}_{T}$ is connected and thus $\pi_{0}\circ Dusk\simeq \const_{*}$.
\end{proof}

\subsection{Homotopy coherent nerves}\label{subsection: Homotopy coherent nerves}

Several different homotopy coherent nerve functors exist in the literature. The original was constructed by Cordier \cite[p.111]{cordier1982sur}:
\begin{equation}\label{diagram: categorification-homotopy coherent nerve adjunction}
\mathfrak{C}: \SSet\leftrightarrows \Cat_{\Delta}: N^{\hc}
\end{equation}
where $\Cat_{\Delta}$ denotes the category of small simplicial categories. More generally, we'll denote by $\mathcal{V}\Cat_{\Delta}$ the category of small categories enriched in $(S\mathcal{V},\otimes,F(\Delta^{0}))$, i.e. simplicial objects in $\mathcal{V}$ with the pointwise tensor product.

We introduced an enriched version in \cite{lowen2024enriched} which specializes to the classical one. Moreover, Moser, Rasekh and Rovelli constructed an enriched version in \cite{moser2024homotopy} to compare different models of $(\infty,n)$-categories. Below we show how all of these fit in the general procedure of \S\ref{subsection: A general procedure}.
 
Define a strong monoidal functor $\hc$ by the composite
\begin{equation}\label{diagram: hc}
\hc: \nec\xrightarrow{\dim} \square\subseteq \Cat\xrightarrow{N} \SSet\xrightarrow{F} S\mathcal{V}
\end{equation}
where $N$ is the ordinary nerve and $F$ is the strong monoidal functor induced by post-composition with $F: \Set\rightarrow \mathcal{V}$.

\subsubsection*{\textbf{The homotopy coherent nerve}}

In \cite[Definition 4.9]{lowen2024enriched}, we constructed the templicial homotopy coherent nerve $N^{hc}_{\mathcal{V}}: \mathcal{V}\Cat_{\Delta}\rightarrow \ts\mathcal{V}$ as an enriched variant of $N^{hc}$.

\begin{Prop}
The nerve functor $\mathcal{V}\Cat_{\Delta}\rightarrow \ts\mathcal{V}$ generated by $\hc: \nec\rightarrow S\mathcal{V}$ is naturally isomorphic to the templicial homotopy coherent nerve $N^{\hc}_{\mathcal{V}}$ of \cite{lowen2024enriched}.
\end{Prop}
\begin{proof}
This is immediate from the definition of $N^{\hc}_{\mathcal{V}}$.
\end{proof}

Given a necklace $(T,p)$ and an integer $n\geq 0$, Dugger and Spivak \cite[\S 4]{dugger2011rigidification} define a \emph{flag of length $n$} on $T$ as a sequence of subsets $\vec{T} = (T\subseteq T_{0}\subseteq T_{1}\subseteq \dots \subseteq T_{n}\subseteq [p])$. The flag $\vec{T}$ is called \emph{flanked} if $T_{0} = T$ and $T_{n} = [p]$.

\begin{Cor}
\begin{enumerate}[1.]
\item There is a natural isomorphism $\tilde{U}\circ N^{\hc}_{\mathcal{V}}\simeq N^{\hc}\circ \mathcal{U}$ where $\mathcal{U}: \mathcal{V}\Cat_{\Delta}\rightarrow \Cat_{\Delta}$ is the forgetful functor. In particular, if $\mathcal{V} = \Set$, then $N^{\hc}_{\mathcal{V}}$ coincides with the classical homotopy coherent nerve $N^{\hc}$.
\item $N^{\hc}_{\mathcal{V}}$ has a left-adjoint $\mathfrak{C}_{\mathcal{V}}: \ts\mathcal{V}\rightarrow \mathcal{V}\Cat_{\Delta}$. Moreover, for any templicial object $(X,S)$ with non-degenerate simplices, and $a,b\in S$ and $n\geq 0$, we have an isomorphism
$$
\mathfrak{C}_{\mathcal{V}}[X]_{n}(a,b)\simeq \coprod_{\substack{T\in \nec\\ \vec{T}\text{ flanked flag on }T\\ \text{of length }n}}\!\!\!\!\!\!\!\!X^{nd}_{T}(a,b)
$$

\item Let $\mathcal{C}$ be a small $S\mathcal{V}$-category such that for all $A,B\in \Ob(\mathcal{C})$, the underlying simplicial set of $\mathcal{C}(A,B)$ is a Kan complex. Then $N^{\hc}_{\mathcal{V}}(\mathcal{C})$ is a quasi-category in $\mathcal{V}$.
\item Let $\mathcal{C}$ be a small $S\mathcal{V}$-enriched category. Then $N^{\hc}_{\mathcal{V}}(\mathcal{C})$ has a Frobenius structure.
\end{enumerate}
\end{Cor}
\begin{proof}
Statements $1$-$3$ were already shown in \cite{lowen2024enriched} using the techniques of Section \ref{section: Enriched nerves of enriched categories}, applied ad hoc to the case $D = \hc: \nec\rightarrow S\mathcal{V}$. In particular, the second statement follows from Theorem \ref{theorem: explicitation of left-adjoint for free temp. obj.} by setting $D'(T)$ to be the set of all flanked flags of length $n$ on $T$. Further, statement $4$ immediately follows from Corollary \ref{corollary: cubical nerve has Frob. structure}.
\end{proof}

The following results comparing the homotopy coherent to the ordinary nerve and Duskin nerve were already known in the literature (see \cite[Tag 00KY]{kerodon} for instance). We recover them swiftly, using the diagrams generating these nerves.

\begin{Cor}
There are natural isomorphisms
\begin{enumerate}[1.]
\item $N^{\hc}_{\mathcal{V}}\circ (-)^{\const}\simeq N_{\mathcal{V}}$
\item $N^{\hc}\circ (-)^{N}\simeq N^{\Dusk}$ (for $\mathcal{V} = \Set$)
\end{enumerate}
where $\const: \mathcal{V}\rightarrow S\mathcal{V}$ sends every object $V\in \mathcal{V}$ to the constant functor on $V$, and $N: \Cat\rightarrow\SSet$ is the ordinary nerve.
\end{Cor}
\begin{proof}
By Proposition \ref{proposition: basic comparison of nerves}, it suffices to note the following.
\begin{enumerate}[1.]
\item Let $\pi_{0}: S\mathcal{V}\rightarrow \mathcal{V}$ be the left-adjoint of $\const$. For $\mathcal{V} = \Set$, this coincides with the functor taking connected components. Then $\pi_{0}\circ \hc = \pi_{0}\circ F\circ N\circ \Dusk\simeq F\circ \pi_{0}\circ N\circ \Dusk\simeq \const_{I}$ since $N(\Dusk(T))$ only has a single connected component.
\item Let $h: \SSet\rightarrow \Cat$ denote the homotopy category functor, which is left-adjoint to the nerve $N$. Then $h\circ \hc = h\circ N\circ \Dusk\simeq \Dusk$ (for $\mathcal{V} = \Set$).
\end{enumerate}
\end{proof}

\subsubsection*{\textbf{The homotopy coherent nerve of $(\infty,n)$-categories}}

Let $\Theta$ be a small category. In what follows, we consider $\mathcal{W} = \mathcal{V} = \SSet^{\Theta^{op}}$ with the cartesian monoidal closed structure. In \cite[Definition 2.3.1]{moser2024homotopy}, an enriched variant of the homotopy coherent nerve is constructed:
$$
\mathfrak{N}: \SSet^{\Theta^{op}}\Cat\rightarrow \mathbf{PC}(\SSet^{\Theta^{op}})
$$
where $\mathbf{PC}(\SSet^{\Theta^{op}})$ denotes the category of $\SSet^{\Theta^{op}}$-enriched precategories, which we may identify with $S_{\times}(\SSet^{\Theta^{op}})$ by Proposition \ref{proposition: temp. obj. results}.3. Specifically, they consider the case where $\Theta = \Theta_{n-1}$ is Joyal's cell category \cite{joyal1997disks}. Below we first recall the definition of $\mathfrak{N}$ and then show it is generated through the general procedure of \S\ref{subsection: A general procedure} by the diagram $\hc: \nec\rightarrow \SSet^{\Theta^{op}}$ \eqref{diagram: hc}.

The category $\mathbf{PC}(\SSet^{\Theta^{op}})\simeq S_{\times}(\SSet^{\Theta^{op}})$ is also equivalent to the full subcategory of $(\SSet^{\Theta^{op}})^{\simp^{op}}$ of simplicial objects $X: \simp^{op}\rightarrow \SSet^{\Theta^{op}}$ such that $X_{0}$ is a set (i.e. a constant functor $\simp^{op}\times \Theta^{op}\rightarrow \Set$). Moreover, we have an isomorphism $(\SSet^{\Theta^{op}})^{\simp^{op}}\simeq \SSet^{\simp^{op}\times \Theta^{op}}$ which sends a functor $X: \simp^{op}\rightarrow \SSet^{\Theta^{op}}$ to the functor
$$
\hat{X}: \simp^{op}\times \Theta^{op}\rightarrow \SSet: ([k],\theta)\mapsto \hat{X}_{k,\theta}\quad \text{with }(\hat{X}_{k,\theta})_{m} = ((X_{m})_{\theta})_{k}
$$
We thus obtain a fully faithful embedding:
$$
\hat{(-)}: S_{\times}(\SSet^{\Theta^{op}})\hookrightarrow \SSet^{\simp^{op}\times \Theta^{op}}
$$

Secondly, consider the category of enriched categories $\mathrm{SSSet}^{\Theta^{op}}\Cat$. We may identify it with the full subcategory of $\Cat_{\Delta}^{\simp^{op}\times \Theta^{op}}$ of functors $\mathcal{C}: \simp^{op}\times \Theta^{op}\rightarrow \Cat_{\Delta}$ such that $\Ob(\mathcal{C}_{k,\theta})$ is constant in $k\geq 0$ and $\theta\in \Theta$. We thus obtain another fully faithful embedding:
$$
\hat{(-)}: \mathrm{SSSet}^{\Theta^{op}}\Cat\hookrightarrow \Cat_{\Delta}^{\simp^{op}\times \Theta^{op}}
$$
sending every $\mathrm{SSSet}^{\Theta^{op}}$-enriched category $\mathcal{C}$ to the functor $\hat{\mathcal{C}}: \simp^{op}\times \Theta^{op}\rightarrow \Cat_{\Delta}$ with $(\hat{\mathcal{C}}_{k,\theta}(a,b))_{m} = (\mathcal{C}(a,b)_{\theta})_{m,k}$ for all $m,k\geq 0$, $\theta\in \Theta$ and $a,b\in \Ob(\mathcal{C})$.

Then the classical homotopy coherent nerve $N^{hc}$ and its left-adjoint $\mathfrak{C}$ induce an adjunction $\mathfrak{C}_{*}: \SSet^{\simp^{op}\times \Theta^{op}}\leftrightarrows \Cat_{\Delta}^{\simp^{op}\times \Theta^{op}}: N^{hc}_{*}$ by post-composition. This restricts to
$$
\mathfrak{C}_{*}: S_{\times}(\SSet^{\Theta^{op}})\leftrightarrows \mathrm{SSSet}^{\Theta^{op}}\Cat: N^{hc}_{*}
$$
Finally, the diagonal functor $\simp\rightarrow \simp\times \simp$ induces a left-adjoint strong monoidal functor $diag: \mathrm{SSSet}^{\Theta^{op}}\twoheadrightarrow \SSet^{\Theta^{op}}$ by pre-composition. Then the nerve $\mathfrak{N}$ is defined as the right-adjoint of the following composite, where we applied $diag$ to hom-objects:
\begin{equation}\label{diagram: left-adjoint of hc nerve for (infty,n)-cats.}
S_{\times}(\SSet^{\Theta^{op}})\xrightarrow{\mathfrak{C}_{*}} \mathrm{SSSet}^{\Theta^{op}}\Cat\xrightarrow{diag} \SSet^{\Theta^{op}}\Cat
\end{equation}

In the following proposition, we identify $\SSet^{\Theta^{op}}\simeq S(\Set^{\Theta^{op}})$. Then for all $T\in \nec$, $\hc(T) = FN(\mathcal{P}_{T})$ where $F: \SSet\rightarrow \SSet^{\Theta^{op}}$ sends every simplicial set $K$ to the constant functor $\Theta^{op}\rightarrow \SSet$ on $K$.

\begin{Prop}\label{proposition: (infty,n)-hc-nerve}
The nerve functor $\SSet^{\Theta^{op}}\Cat\rightarrow S_{\times}(\SSet^{\Theta^{op}})$ generated by $\hc: \nec\rightarrow \SSet^{\Theta^{op}}$ is naturally isomorphic to the homotopy coherent nerve $\mathfrak{N}$ of \cite{moser2024homotopy}.
\end{Prop}
\begin{proof}
Consider the projections
$$
\pi_{1}: \simp\times \simp\rightarrow \simp: ([m],[k])\mapsto [m]\quad \text{and}\quad \pi_{2}: \simp\times \simp\rightarrow \simp: ([m],[k])\mapsto [k]
$$
which induce strong monoidal left-adjoints $\iota_{1},\iota_{2}: \SSet^{\Theta^{op}}\hookrightarrow \mathrm{SSSet}^{\Theta^{op}}$ by pre-composition. Explicitly, $(\iota_{1}(X)_{\theta})_{m,k} = (X_{\theta})_{m}$ and $(\iota_{2}(X)_{\theta})_{m,k} = (X_{\theta})_{k}$ for all $X\in \SSet^{\Theta^{op}}$ and $m,k\geq 0$. We can consider $\mathrm{SSSet}^{\Theta^{op}}$ as tensored over $\SSet^{\Theta^{op}}$ via the inclusion $\iota_{2}$. Since, $diag\circ \iota_{1} = \id$, it suffices by Proposition \ref{proposition: basic comparison of nerves} to show that $N^{\iota_{1}\circ \hc}_{\SSet^{\Theta^{op}}}$ is naturally isomorphic to $N^{hc}_{*}: \mathrm{SSSet}^{\Theta^{op}}\Cat\rightarrow S_{\times}(\SSet^{\Theta^{op}})$. We proceed by comparing their left-adjoints.

Take $X\in S_{\times}(\SSet^{\Theta^{op}})$ with vertices $a$ and $b$. Consider the functors $\hat{L^{\iota_{1}\circ \hc}_{\SSet^{\Theta^{op}}}(X)}$, $\mathfrak{C}_{*}[\hat{X}]: \simp^{op}\times \Theta^{op}\rightarrow \Cat_{\Delta}$. Then observe that for all $k\geq 0$ and $\theta\in \Theta$:
\begin{align*}
\hat{L^{\iota_{1}\circ \hc}_{\SSet^{\Theta^{op}}}(X)}_{k,\theta}(a,b) &= (L^{\iota_{1}\circ \hc}_{\SSet^{\Theta^{op}}}(X)(a,b)_{\theta})_{\bullet,k} = {\colim_{T\in \nec}}^{(\iota_{2}(X_{T}(a,b))_{\theta})_{\bullet,k}}(\iota_{1}\hc(T))_{\theta})_{\bullet,k}\\
&= {\colim_{T\in \nec}}^{(X_{T}(a,b)_{\theta})_{k}}\hc(T)_{\theta} = {\colim_{T\in \nec}}^{(X_{T}(a,b)_{\theta})_{k}}N\mathcal{P}_{T}\\
&= {\colim_{T\in \nec}}^{(\hat{X}_{k,\theta})_{T}(a,b)}N\mathcal{P}_{T}\simeq \mathfrak{C}[\hat{X}_{k,\theta}](a,b) = \mathfrak{C}_{*}[\hat{X}]_{k,\theta}(a,b)
\end{align*}
It is clear that this induces an isomorphism $\hat{L^{\iota_{1}\circ \hc}_{\SSet^{\Theta^{op}}}(X)}_{k,\theta}\simeq \mathfrak{C}_{*}[\hat{X}]_{k,\theta}$ of simplicial categories which is natural in $k$, $\theta$ and $X$.
\end{proof}

Statement 2 of the following corollary of course already appears in \cite{moser2024homotopy} by construction. Here we see it also follows from Section \ref{section: Enriched nerves of enriched categories}.

\begin{Cor}
The following statements are true.
\begin{enumerate}[1.]
\item There is a natural isomorphism $\tilde{U}\circ \mathfrak{N}\simeq N^{hc}$
\item $\mathfrak{N}$ has a left-adjoint.
\item Let $\mathcal{C}$ be a small $\SSet^{\Theta^{op}}$-category. Then $\mathfrak{N}(\mathcal{C})$ has a Frobenius structure.
\end{enumerate}
\end{Cor}
\begin{proof}
\begin{enumerate}[1.]
\item This follows from Theorem \ref{theorem: underlying D-nerve is cat. nerve assoc. to D}.
\item This follows from Proposition \ref{proposition: nerve gen. by strong mon. diagram has left-adj.}.
\item This follows from Corollary \ref{corollary: cubical nerve has Frob. structure}.
\end{enumerate}
\end{proof}

\begin{Rem}
Note that we cannot apply Theorem \ref{theorem: explicitation of left-adjoint for free temp. obj.} to describe the left-adjoint of $\mathfrak{N}$ because there is no functor $D': \nec_{-}\rightarrow \SSet^{\Theta^{op}}$ such that $\Lan_{\iota}D'\simeq hc$. Further, while Theorem \ref{theorem: D-nerve is quasi-cat.} is technically speaking applicable here, a templicial object is a quasi-category in $\SSet^{\Theta^{op}}$ if and only if its underlying simplicial set is a quasi-category, which is not very meaningful.
\end{Rem}

\subsection{The differential graded nerve}\label{subsection: The differential graded nerve}

Fix a commutative unital ring $k$ and let $\Ch(k)$ denote the category of chain complexes over $k$. The differential graded (dg) nerve implicitly goes back to \cite{block2009Riemann} and was named and studied in \cite{lurie2016higher}.
$$
L^{\dg}: \SSet\leftrightarrows k\Cat_{\dg}: N^{\dg}
$$
where $k\Cat_{dg}$ denotes the category of small differential graded (that is $\Ch(k)$-enriched) categories over $k$. Let us see how it fits in the general procedure of \S\ref{subsection: A general procedure}.

Define a strong monoidal functor $\dg$ as the composite
\begin{equation}\label{diagram: dg}
\dg: \nec\xrightarrow{\dim} \square\xrightarrow{\yo} \CSet\xrightarrow{F} C\Mod(k)\xrightarrow{N^{\square}_{\bullet}} \Ch(k)
\end{equation}
where $\yo$ is the Yoneda embedding and $N^{\square}_{\bullet}$ is the cubical normalized chain functor \cite{antolini2002geometric}\cite{rivera2018cubical}, which is strong monoidal. Let us describe the diagram $\dg$ in a little more detail.

\begin{enumerate}[1.]
\item For any $(T,p)\in \nec$, $\dg(T)_{\bullet}$ is the chain complex given by
$$
\dg(T)_{n}\simeq \bigoplus_{\substack{g: U\hookrightarrow T\\\text{in }\nec_{+}\\ \dim(U) = n}}k.g
{}$$
for all integers $n$. The differential $\partial$ on $\dg(T)_{\bullet}$ is given by
$$
\partial_{n}(g) = \sum_{j=1}^{n}(-1)^{j-1}(g\delta_{i_{j}} - g\nu_{i_{j},q-i_{j}})
$$
for all integers $n$ and all injective necklace maps $g: (U,q)\hookrightarrow (T,p)$ with $\dim(U) = n$, where we have written $U^{c} = \{i_{1} < ... < i_{n}\}$.
\item For any necklace map $f: T\rightarrow T'$, the induced chain map $\dg(f): \dg(T)_{\bullet}\rightarrow \dg(T')_{\bullet}$ is as follows. For any injective necklace map $g: U\hookrightarrow T$, factor $fg$ as an active surjective map $\sigma: U\rightarrow U'$ followed by an injective map $g': U'\rightarrow T'$. Then
$$
\dg(f)_{n}(g) =
\begin{cases}
g' & \text{if }\dim(U) = \dim(U')\\
0 & \text{otherwise}
\end{cases}
$$
\end{enumerate}

\begin{Rem}
Given a necklace $(T,p)$, note that $\dg(T)_{n}$ is concentrated in degrees $0\leq n\leq \dim(T)$. In particular, we have
$$
\dg(T)_{0}\simeq \bigoplus_{\substack{g: [q]\hookrightarrow [p]\\ \text{in }\fint^{inj}}}\!\!\!\!k\quad \text{and}\quad \dg(T)_{\dim(T)}\simeq k.g
$$
\end{Rem}

In \cite[Definition 3.15]{lowen2023frobenius}, we constructed a templicial lift $N^{dg}_{k}$ of the classical dg-nerve:
$$
N^{dg}_{k}: k\Cat_{dg}\rightarrow \ts\Mod(k)
$$
We then have the following proposition, the proof of which we postpone to Appendix \ref{section: The generating diagram of the differential graded nerve} because it is rather technical.

\begin{Prop}\label{proposition: dg-nerve}
The nerve $k\Cat_{dg}\rightarrow \ts\Mod(k)$ generated by $\dg: \nec\rightarrow \Ch(k)$ is naturally isomorphic to the templicial dg-nerve of \cite{lowen2023frobenius}.
\end{Prop}

Statements $1$, $3$ and $4$ of the following corollary were already shown in \cite{lowen2023frobenius}, but they are now simple consequences of the results from Section \ref{section: Enriched nerves of enriched categories}. Statements 1 and 3 also recover the fact that $N^{dg}(\mathcal{C})$ is an ordinary quasi-category for any small dg-category $\mathcal{C}$ (\cite[Proposition 1.3.1.10]{lurie2016higher}). The description of the left-adjoint in statement 2 is novel. Note that it also applies to the left-adjoint of $N^{dg}$ by choosing $X = \tilde{F}(K)$ for $K\in \SSet$.

\begin{Cor}\label{corollary: results for the dg-nerve}
The following statements are true.
\begin{enumerate}[1.]
\item There is a natural isomorphism $\tilde{U}\circ N^{\dg}_{k}\simeq N^{\dg}$.
\item $N^{dg}_{k}$ has a left-adjoint $L^{dg}_{k}: \ts\Mod(k)\rightarrow k\Cat_{dg}$. Moreover, for any templicial $k$-module $(X,S)$ with non-degenerate simplices, and $a,b\in S$ and $n\in \mathbb{Z}$, we have an isomorphism of $k$-modules
$$
L^{\dg}_{k}(X)_{n}(a,b)\simeq \bigoplus_{\substack{T\in \nec\\ \dim(T) = n}}X^{nd}_{T}(a,b)
$$
\item Let $\mathcal{C}$ be a small dg-category. Then $N^{\dg}_{k}(\mathcal{C})$ is a quasi-category in $\Mod(k)$.
\item Let $\mathcal{C}$ be a small dg-category. Then $N^{\dg}_{k}(\mathcal{C})$ has a Frobenius structure.
\end{enumerate}
\end{Cor}
\begin{proof}
\begin{enumerate}[1.]
\item In \cite[Theorem 6.1]{rivera2018cubical}, it was shown that the left-adjoint $L^{\dg}$ of the dg-nerve $N^{\dg}$ is given by $L^{\dg}\simeq (-)^{N^{\square}_{\bullet}}\circ L^{\cub}$ where $L^{\cub}$ is the left-adjoint of the cubical nerve $N^{\cub}$ (see \S\ref{subsection: Necklaces versus cubes}), and $(-)^{N^{\square}_{\bullet}}$ applies the cubical normalized chains functor to hom-objects. Denoting $\Gamma^{\square}: \Ch(k)\rightarrow C\Mod(k)$ for the right-ajoint of $N^{\square}_{\bullet}$, we thus find $N^{\dg}\simeq N^{\cub}\circ (-)^{\Gamma^{\square}}$. On the other hand, by \eqref{diagram: dg} and Proposition \ref{proposition: basic comparison of nerves}, the necklicial nerve $k\Cat_{dg}\rightarrow \SSet$ generated by $\dg$ is also isomorphic to $N^{\cub}\circ (-)^{\Gamma^{\square}}$. Thus the result follows from Theorem \ref{theorem: underlying D-nerve is cat. nerve assoc. to D} and Proposition \ref{proposition: dg-nerve}.
\item The left-adjoint exists by Proposition \ref{proposition: nerve gen. by strong mon. diagram has left-adj.} since $\dg$ is strong monoidal. Given $n\in \mathbb{Z}$, consider $\pi_{n}: \Ch(k)\rightarrow \Mod(k): C_{\bullet}\mapsto C_{n}$ and
$$
D': \nec_{-}\rightarrow \Mod(k): T\mapsto
\begin{cases}
k & \text{if }\dim(T) = n\\
0 & \text{if }\dim(T) = 0
\end{cases}
$$
which sends any map in $\nec_{-}$ to the identity if it is spine collapsing, and to the zero map otherwise, which is well-defined by Lemma \ref{lemma: spine collapsing equiv.}. Then in view of Lemma \ref{lemma: left Kan ext. of surj. necklace maps}, we have a canonical isomorphism $\Lan_{\iota}D'\simeq \pi_{n}\dg$. Thus the result follows from Theorem \ref{theorem: explicitation of left-adjoint for free temp. obj.}.
\item Take $0 < j < n$. For the lifting problem of Theorem \ref{theorem: D-nerve is quasi-cat.} to have a solution, it suffices to show that the inclusion $\mathfrak{l}^{\dg}((\Lambda^{n}_{j})_{\bullet}(0,n))\hookrightarrow \dg_{\bullet}(\Delta^{n})$ splits. Denote $C_{\bullet} = \mathfrak{l}^{\dg}((\Lambda^{n}_{j})_{\bullet}(0,n))$. Then note that $C_{\bullet}$ is the subcomplex of $\dg_{\bullet}(\Delta^{n})$ generated by all injective necklace maps $g: U\hookrightarrow \Delta^{n}$ different from $\id_{\Delta^{n}}$ and $\delta_{j}$. Now the quotient $\dg_{\bullet}(\Delta^{n})/C_{\bullet}$ is the complex $k\xrightarrow{=} k$ concentrated in degrees $n-1$ and $n-2$, which is acyclic and degreewise free. Consequently, the following exact sequence splits in $\Ch(k)$:
$$
0\rightarrow C_{\bullet}\hookrightarrow \dg_{\bullet}(\Delta^{n})\twoheadrightarrow \dg_{\bullet}(\Delta^{n})/C_{\bullet}\rightarrow 0
$$

\item This is immediate from Corollary \ref{corollary: cubical nerve has Frob. structure}.
\end{enumerate}
\end{proof}

\begin{Rem}
In the proof of Corollary \ref{corollary: results for the dg-nerve}.3, a retraction $\pi: \dg_{\bullet}(\Delta^{n})\rightarrow C_{\bullet}$ of $C_{\bullet}\hookrightarrow \dg_{\bullet}(\Delta^{n})$ can also easily be defined as follows:
$$
\pi(g) = g\text{ if }g\neq \id_{\Delta^{n}}, g\neq \delta_{j},\qquad \pi(\id_{\Delta^{n}}) = 0,\qquad \pi(\delta_{j}) = \sum_{i=1, i\neq j}^{n-1}(-1)^{i+j-1}\left(\delta_{i}-\nu_{i,n-i}\right).
$$
\end{Rem}

We end the subsection by comparing the templicial dg-nerve with the templicial homotopy coherent nerve of \S\ref{subsection: Homotopy coherent nerves}. Faonte showed in \cite[Proposition 3.3.2]{faonte2015simplicial} that the dg-nerve of a dg-category is equivalent to the homotopy coherent nerve of its associated simplicial category. This equivalence was moreover strengthened to a trivial Kan fibration in \cite[Tag 00SV]{kerodon}. In Corollary \ref{corollary: comparisons with dg-nerve}, we lift this to a trivial fibration of templicial modules, by comparing the diagrams which generate both nerves.

Recall the normalized chain functor $N_{\bullet}: S\Mod(k)\rightarrow \Ch(k)$ and its right-adjoint $\Gamma$ \cite{dold1958homology}. The functor $N_{\bullet}$ is colax monoidal with comultiplications given by the Alexander-Whitney maps \cite[Definition 29.7]{may1967simplicial}, whereby $\Gamma$ has a canonical lax structure. Further, we have the diagram $\hc: \nec\rightarrow S\Mod(k)$ \eqref{diagram: hc}. Thus $N_{\bullet}\circ \hc: \nec\rightarrow \Ch(k)$ is a colax monoidal diagram as well. Note that for $(T,p)\in \nec$, $N_{\bullet}(\hc(T))$ is the normalized chain complex of a simplicial cube of dimension $\dim(T)$, and is thus concentrated in degrees $0,\dots,\dim(T)$. We call a flag $\hat{T} = (T_{0}\subseteq T_{1}\subseteq \dots\subseteq T_{m})$ on $T$ (see \S\ref{subsection: Homotopy coherent nerves}) \emph{non-degenerate} if it is non-degenerate as an $m$-simplex of $N(\mathcal{P}_{T})$. That is, all of the inclusions $T_{i-1}\subsetneq T_{i}$ are strict. Then for every $m\in \mathbb{Z}$, $N_{m}(\hc(T))$ is freely generated by the set of non-degenerate flags of length $m$ on $T$.

Let $(T,p)$ be a necklace with $T^{c} = \{i_{1} < \dots < i_{n}\}$. Following \cite[Tag 00SJ]{kerodon}, define
$$
[\square^{T}] = (-1)^{n}\sum_{\tau\in \mathbb{S}(T)}\sgn(\tau)\hat{\tau}(T)\quad \in N_{n}(\hc(T))
$$
as the \emph{fundamental chain of $T$}, where $\mathbb{S}(T)$ is the group of bijections $T^{c}\xrightarrow{\simeq} T^{c}$, $\sgn(\tau) = \pm 1$ is the sign of a permutation $\tau\in \mathbb{S}(T)$ and $\hat{\tau}(T)$ is the non-degenerate flanked flag
$$
(T\subsetneq T\cup \{\tau(i_{1})\}\subsetneq T\cup \{\tau(i_{1}),\tau(i_{2})\}\subsetneq ... \subsetneq [p])
$$
of length $n$ on $T$. Note that we put an extra sign $(-1)^{n}$ compared to \cite{kerodon} to accommodate for the difference in convention for $\mathcal{P}_{T}$ (also see Remark \ref{remark: convention for partitions}).

\begin{Prop}\label{proposition: comparison of dg and hc nerve}
There is a unique monoidal natural transformation
$$
\mathfrak{z}: \dg\rightarrow N_{\bullet}\circ \hc
$$
For all necklaces $T$, the chain map $\mathfrak{z}_{T}: \dg_{\bullet}(T)\rightarrow N_{\bullet}(\hc(T))$ sends the generator $\id_{T}$ to the fundmental chain $[\square^{T}]$.
\end{Prop}
\begin{proof}
We reduce the data of a monoidal natural transformation $\mathfrak{z}: \dg\rightarrow N_{\bullet}\circ \hc$ in a couple of steps. First note that $\mathfrak{z}$ is completely determined by a collection of chains $\mathfrak{z}(g)\in N_{\dim(U)}(\hc(T))$ for all injective necklace maps $g: U\hookrightarrow T$. By the naturality of $\mathfrak{z}$ in $T$ it follows in particular that for all injective necklace maps $g: U\hookrightarrow T$, we have $\mathfrak{z}(g) = N_{\bullet}(\hc(g))(\mathfrak{z}(\id_{U}))$. Hence $\mathfrak{z}$ is completely determined by $\mathfrak{z}(\id_{T})\in N_{\dim(T)}(\hc(T))$ for all $T\in \nec$. Now note that any non-degenerate flag on $T$ of length $n = \dim(T)$ is necessarily flanked and thus
$$
\mathfrak{z}(\id_{T}) = \sum_{\tau\in \mathbb{S}_{n}}\lambda^{T}_{\tau}\hat{\tau}(T)
$$
for some unique $\lambda^{T}_{\tau}\in k$. Now, carefully going through the definitions, yields the following.
\begin{itemize}
\item The naturality of $\mathfrak{z}$ in $T\in \nec$ is equivalent to having, for any spine collapsing necklace map $\sigma: T\twoheadrightarrow T'$ and $\tau\in \mathbb{S}(T)$:
\begin{equation}\label{equation: comparison natural}
\lambda^{T}_{\tau} = \lambda^{T'}_{\sigma\tau\sigma^{-1}}
\end{equation}
where we used that $\sigma$ induces a bijection $T^{c}\xrightarrow{\simeq} (T')^{c}$ (Lemma \ref{lemma: spine collapsing equiv.}).

\item The monoidality of $\mathfrak{z}$ is equivalent to having, for all $T_{1},T_{2}\in \nec$ and all $\tau_{1}\in \mathbb{S}(T_{1})$ and $\tau_{2}\in \mathbb{S}(T_{2})$:
\begin{equation}\label{equation: comparison monoidal}
\lambda^{T_{1}\vee T_{2}}_{\tau_{1}*\tau_{2}} = \lambda^{T_{1}}_{\tau_{1}}\lambda^{T_{2}}_{\tau_{2}}
\qquad \text{and}\quad \lambda^{\Delta^{0}}_{\id} = 1 
\end{equation}
where we used the canonical group embedding $*: \mathbb{S}(T_{1})\times \mathbb{S}(T_{2})\hookrightarrow \mathbb{S}(T_{1}\vee T_{2})$.

\item The fact that $\mathfrak{z}_{T}: \dg_{\bullet}(T)\rightarrow N_{\bullet}(\hc(T))$ is a chain map for all $T\in \nec$ with $T^{c} = \{i_{1} < \dots < i_{n}\}$, is equivalent to having, for all $\tau\in \mathbb{S}(T^{c})$, $j\in \{1,\dots,n\}$, $\rho\in \mathbb{S}(\delta_{i_{j}}^{-1}(T))$, $\theta\in \mathbb{S}(T\cup \{i_{j}\})$ and $l\in \{1,\dots,n-2\}$:
\begin{equation}\label{equation: comparison chain map}
\lambda^{T}_{(\rho\vert j)} = (-1)^{d+j-1}\lambda^{\delta^{-1}_{i_{j}}(T)}_{\rho},\quad \lambda^{T}_{(j\vert \theta)} = (-1)^{j}\lambda^{T\cup \{i_{j}\}}_{\theta}\quad \text{and}\quad \lambda^{T}_{\tau(i_{l},i_{l+1}
)} = -\lambda^{T}_{\tau}
\end{equation}
where $(\rho\vert j)(i_{k}) = \delta_{i_{j}}\rho\delta^{-1}_{i_{j}}(i_{\delta_{j}(k)})$ if $k < n$ and $(\rho\vert j)(i_{n}) = i_{j}$, and $(j\vert\theta)(i_{1}) = i_{j}$ and $(j\vert\theta)(i_{k}) = \theta(i_{\delta_{j}(k-1)})$ if $k > 1$. Then $\sgn(\rho\vert j) = (-1)^{n-j}\sgn(\rho)$ and $\sgn(j\vert\theta) = (-1)^{j-1}\sgn(\theta)$. Further, $(i_{l},i_{l+1})$ is the transposition swapping $i_{l}$ and $i_{l+1}$.
\end{itemize}
Note that by \eqref{equation: comparison natural} and \eqref{equation: comparison monoidal}, $\lambda^{T}_{\id} = 1$ for any spine $T = \Delta^{1}\vee \dots \vee \Delta^{1}$. Then it follows by the first equation of \eqref{equation: comparison chain map} that $\lambda^{T}_{\id} = (-1)^{\dim(T)}$ for any necklace $T$. Further by \eqref{equation: comparison chain map}, we have $\lambda^{T}_{\tau} = (-1)^{\dim(T)}\sgn(\tau)$ for any necklace $T$ and $\tau\in \mathbb{S}(T)$. Now it is clear that this formula indeed satisfies \eqref{equation: comparison natural}, \eqref{equation: comparison monoidal} and \eqref{equation: comparison chain map} so that we have a unique solution.
\end{proof}

\begin{Cor}\label{corollary: comparisons with dg-nerve}
The following statements are true.
\begin{enumerate}[1.]
\item There is a natural isomorphism
$$
N^{\dg}_{k}\circ (-)^{\iota_{0}}\simeq N_{k}
$$
where $\iota_{0}: \Mod(k)\rightarrow \Ch(k)_{\geq 0}$ places every module in degree $0$.
\item There is a pointwise trivial fibration of templicial objects
$$
\mathfrak{Z}: N^{\hc}_{k}\circ (-)^{\Gamma}\rightarrow N^{\dg}_{k}
$$
where $\Gamma: \Ch(k)\rightarrow S\Mod(k)$ is the right-adjoint of the normalized chain functor $N_{\bullet}: S\Mod(k)\rightarrow \Ch(k)$. In particular, $\tilde{U}\mathfrak{Z}: N^{\hc}\circ \mathcal{U}\circ(-)^{\Gamma}\rightarrow N^{\dg}$ is a pointwise trivial Kan fibration between quasi-categories.
\end{enumerate}
\end{Cor}
\begin{proof}
\begin{enumerate}[1.]
\item By Proposition \ref{proposition: basic comparison of nerves}, it suffices to note that $\iota_{0}$ is right-adjoint to $H_{0}: \Ch(k)_{\geq 0}\rightarrow \Mod(k)$ and that $H_{0}\circ dg$ is constant on $k$.
\item By Proposition \ref{proposition: basic comparison of nerves}, the nerve generated by $N_{\bullet}\circ \hc$ is given by the composite $N^{hc}_{k}\circ (-)^{\Gamma}$. The monoidal natural transformation $\mathfrak{z}$ of Proposition \ref{proposition: comparison of dg and hc nerve} thus induces a comparison map $\mathfrak{Z}: N^{\hc}_{k}\circ (-)^{\Gamma}\rightarrow N^{\dg}_{k}$. Take $n > 0$, and define
$$
C_{\bullet} = \mathfrak{l}^{N_{\bullet}\hc}(\partial\Delta^{n}_{\bullet}(0,n))\amalg_{\mathfrak{l}^{\dg}(\partial\Delta^{n}_{\bullet}(0,n))}\dg_{\bullet}(\Delta^{n})
$$
In order for the lifting problem of Corollary \ref{corollary: trivial fib. between nerves}.2 to have a solution for all small dg-categories, it suffices to show that the map $C_{\bullet}\rightarrow N_{\bullet}(\hc(\Delta^{n}))$ is a split monomorphism. This can be done by noting that the quotient is acyclic and degreewise free, as was done in the proof of \cite[Tag 00SV]{kerodon}.
\end{enumerate}
\end{proof}

\subsection{The cubical nerve}\label{subsection: The cubical nerve}

Recall the cubical nerve and its left-adjoint by Le Grignou \cite[Definition 28]{legrignou2020cubical}, which we already discussed in \S\ref{subsection: Necklaces versus cubes}:
$$
L^{\cub}: \SSet\leftrightarrows \Cat_{\square}: N^{\cub}
$$
In Proposition \ref{proposition: cubical nerve is necklicial}, we showed that it is necklicial and generated by the diagram $\yo\circ \dim: \nec\rightarrow \CSet$. We can then easily construct an enriched version by considering the diagram
\begin{equation}\label{diagram: cub}
\cub: \nec\xrightarrow{\dim} \square\xrightarrow{\yo} \CSet\xrightarrow{F} C\mathcal{V}
\end{equation}
where $C\mathcal{V} = \mathcal{V}^{\square^{op}}$ denotes the category of cubical objects in $\mathcal{V}$ equipped with the Day convolution so that $\cub$ is a strong monoidal diagram. We denote by $\mathcal{V}\Cat_{\square}$ the category of small $C\mathcal{V}$-enriched categories.

\begin{Def}
Define the \emph{templicial cubical nerve} as the nerve functor $N^{\cub}_{\mathcal{V}}$ generated by $\cub$. In other words, it is the composite
$$
N^{\cub}_{\mathcal{V}}: \mathcal{V}\Cat_{\square}\xrightarrow{\mathfrak{n}^{\cub}_{\mathcal{V}}} \mathcal{V}\Cat_{\nec}\xrightarrow{(-)^{temp}} \ts\mathcal{V}
$$
\end{Def}

Let us call a map $[1]^{m}\rightarrow [1]^{n}$ in $\square$ \emph{injective} (resp. \emph{surjective}) if it is injective (resp. surjective) on objects. Note that every map in $\square$ can be uniquely decomposed as a surjective map followed by an injective map. This yields an orthogonal factorization system $(\square_{-},\square_{+})$ on $\square$.

\begin{Lem}\label{lemma: dim. discrete fibration on injective maps}
The functor $\dim: \nec_{+}\rightarrow \square_{+}$ is a discrete fibration. That is, for any injective map $g: [1]^{n}\hookrightarrow [1]^{\dim(T)}$ in $\square$ with $(T,p)\in \nec$, there exists a unique injective necklace map $f: U\hookrightarrow T$ such that $\dim(U) = n$ and $\dim(f) = g$.
\end{Lem}
\begin{proof}
Let us first show existence. Since $g$ is a composite of coface maps $\delta_{j}^{\epsilon}$, we may assume that $g = \delta^{\epsilon}_{j}$ and $\dim(T) = n + 1$. Write $T^{c} = \{i_{1} < \dots < i_{n+1}\}$ We distinguish two cases:
\begin{itemize}
\item If $\epsilon = 0$, then $\dim(\delta_{i_{j}}) = \delta^{0}_{j}$ for the necklace map $\delta_{i_{j}}: \delta_{i_{j}}^{-1}(T)\hookrightarrow T$.
\item If $\epsilon = 1$, then $\dim(\nu_{i_{j},p-i_{j}}) = \delta^{1}_{j}$ for the necklace map $\nu_{i_{j},p-i_{j}}: T\cup \{i_{j}\}\hookrightarrow T$. 
\end{itemize}

Further, to show uniqueness, let $f: U\hookrightarrow T$ be an injective necklace map such that $\dim(f) = g$. Let us write $T^{c} = \{i_{1} < \dots < i_{dim(T)}\}$. By the relations in $\square$ (see \cite{brown1981algebra}), there exist unique $1\leq j_{1} < \dots < j_{m}\leq \dim(T)$ and $\epsilon_{1},\dots ,\epsilon_{m}\in \{0,1\}$ such that $g = \delta^{\epsilon_{1}}_{j_{m}}\dots \delta^{\epsilon_{m}}_{j_{1}}$. Then it follows from the definition of $\dim$ that precisely
$$
[p]\setminus \mathrm{Im}(f) = \{i_{j_{s}}\mid 1\leq s\leq m, \epsilon_{s} = 0\}\quad \text{and}\quad f(U)\setminus T = \{i_{j_{s}}\mid 1\leq s\leq m, \epsilon_{s} = 1\}
$$
The first set completely determines the underlying map of $f$ in $\fint$, while the second set determines $U$. Thus the necklace map $f: U\hookrightarrow T$ is determined by $g$ and $T$.
\end{proof}

Recall the cubical sets $\sqcap^{n}_{j,\epsilon}\in \CSet$ for $j\in \{1,\dots,n\}$ and $\epsilon\in \{0,1\}$, which is the union of all faces of $\square^{n}$ except the face $\delta^{\epsilon}_{j}$. A precise definition can be found in \cite[Example 1]{legrignou2020cubical} or \cite[Tag 00LN]{kerodon}. For the case $\mathcal{V} = \Set$, statement $3$ of the next corollary is also a consequence of \cite[Corollary 4]{legrignou2020cubical}.

\begin{Cor}\label{corollary: results for the cubical nerve}
The following statements are true.
\begin{enumerate}[1.]
\item There is a natural isomorphism $\tilde{U}\circ N^{\cub}_{\mathcal{V}}\simeq N^{\cub}\circ \mathcal{U}$ where $\mathcal{U}: \mathcal{V}\Cat_{\square}\rightarrow \Cat_{\square}$ is the forgetful functor. In particular, if $\mathcal{V} = \Set$, then $N^{\cub}_{\mathcal{V}}$ coincides with $N^{\cub}$.
\item $N^{\cub}_{\mathcal{V}}$ has a left-adjoint $L^{\cub}_{\mathcal{V}}: \ts\mathcal{V}\rightarrow \mathcal{V}\Cat_{\square}$. For any templicial object $(X,S)$ that has non-degenerate simplices, $a,b\in S$ and $n\geq 0$, we have
$$
L^{\cub}_{\mathcal{V}}(X)_{n}(a,b)\simeq \coprod_{\substack{T\in \nec\\ [1]^{n}\twoheadrightarrow [1]^{\dim(T)}\\ \text{in }\square_{-}}}X^{nd}_{T}(a,b)
$$
\item Let $\mathcal{C}$ be small $C\mathcal{V}$-category such that for all $A,B\in \mathcal{C}$ and $1\leq j\leq n$, the following lifting problem admits a solution in $\CSet$:
\[\begin{tikzcd}
	{\sqcap^{n}_{j,1}} & {U(\mathcal{C}(A,B))} \\
	{\square^{n}}
	\arrow[from=1-1, to=1-2]
	\arrow[hook, from=1-1, to=2-1]
	\arrow[dashed, from=2-1, to=1-2]
\end{tikzcd}\]
Then $N^{\cub}_{\mathcal{V}}(\mathcal{C})$ is a quasi-category in $\mathcal{V}$.
\item Let $\mathcal{C}$ be small $C\mathcal{V}$-category. Then $N^{\cub}_{\mathcal{V}}(\mathcal{C})$ has a Frobenius structure.
\end{enumerate}
\end{Cor}
\begin{proof}
\begin{enumerate}[1.]
\item In view of Theorem \ref{theorem: underlying D-nerve is cat. nerve assoc. to D}, it suffices to show that $\Phi(\cub)\simeq \mathcal{F}(W_{c})$ where $\mathcal{F}$ is induced by the free functor $F:\Set\rightarrow \mathcal{V}$. But this follows from Proposition \ref{proposition: cubical nerve is necklicial}.
\item The left-adjoint exists by Proposition \ref{proposition: nerve gen. by strong mon. diagram has left-adj.} since $\cub$ is strong monoidal. Given $n\geq 0$, consider the functor $\pi_{n}: C\mathcal{V}\rightarrow \mathcal{V}: Y\mapsto Y_{n}$. In view of Theorem \ref{theorem: explicitation of left-adjoint for free temp. obj.}, it suffices to show that $\pi_{n}\circ \cub\simeq \Lan_{\iota}D'$ where $\iota: \nec_{-}\hookrightarrow \nec$ is the inclusion and
$$
D' = F(\square_{-}([1]^{n},[1]^{\dim(-)})): \nec_{-}\rightarrow \mathcal{V}
$$
This now easily follows from Lemmas \ref{lemma: left Kan ext. of surj. necklace maps} and \ref{lemma: dim. discrete fibration on injective maps} since for all $T\in \nec$:
\begin{align*}
\pi_{n}(\cub(T)) &= F(\square^{\dim(T)}_{n}) = F(\square([1]^{n},[1]^{\dim(T)}))\\
&\simeq \coprod_{m\geq 0}F(\square_{-}([1]^{n},[1]^{m}))\otimes F(\square_{+}([1]^{m},[1]^{\dim(T)}))\simeq \coprod_{\substack{U\hookrightarrow T\\ \text{in }\nec_{+}}}D'(U)
\end{align*}
and this isomorphism is natural in $T$.
\item In view of Theorem \ref{theorem: D-nerve is quasi-cat.}, it suffices to verify that $\mathfrak{l}^{cub}((\Lambda^{n}_{j})_{\bullet}(0,n))\simeq \sqcap^{n-1}_{j,1}$, which follows from \eqref{equation: horn as colimit}.
\item This immediately follows from Corollary \ref{corollary: cubical nerve has Frob. structure}.
\end{enumerate}
\end{proof}

In the following corollary, we consider the adjunctions:
\begin{itemize}
\item $N^{\square}_{\bullet}: C\Mod(k)\leftrightarrows \Ch(k): \Gamma^{\square}$\\
with $N^{\square}_{\bullet}$ the cubical normalized chains functor.
\item $tr: C\mathcal{V}\leftrightarrows S\mathcal{V}: sq$\\
with $tr$ the left Kan extension of $FN: \square\subseteq \Cat\rightarrow S\mathcal{V}$ along $F\yo: \square\hookrightarrow C\mathcal{V}$.
\item $h: \SSet\leftrightarrows \Cat: N$\\
with $N$ the classical nerve functor.
\item $\pi_{0}: C\mathcal{V}\leftrightarrows \mathcal{V}: \const$\\
with $\const$ sending every $V\in \mathcal{V}$ to the constant functor $\square^{op}\rightarrow \mathcal{V}$ on $V$.
\end{itemize}  

\begin{Cor}
There exist natural isomorphisms
\begin{enumerate}[1.]
\item $N^{\cub}_{k}\circ (-)^{\Gamma^{\square}}\simeq N^{\dg}_{k}$ (for $\mathcal{V} = \Mod(k)$)
\item $N^{\cub}_{\mathcal{V}}\circ (-)^{sq}\simeq N^{\hc}_{\mathcal{V}}$
\item $N^{\cub}\circ (-)^{sq N}\simeq N^{\Dusk}$ (for $\mathcal{V} = \Set$)
\item $N^{\cub}_{\mathcal{V}}\circ (-)^{\const}\simeq N_{\mathcal{V}}$
\end{enumerate}
\end{Cor}
\begin{proof}
All of these follows from Proposition \ref{proposition: basic comparison of nerves} by noting that $\hc\simeq tr\circ \cub$, $\Dusk\simeq h\circ tr\circ \cub$ and $\pi_{0}\circ \cub\simeq \const_{I}$.
\end{proof}

\subsection{Change of enriching category}\label{subsection: Change of enriching category}

Let $(\mathcal{V}',\otimes',I')$ be another cocomplete and finitely complete symmetric monoidal closed category. In this subsection, we show how changing the enriching category $\ts\mathcal{V}\rightarrow \ts\mathcal{V}'$ from $\mathcal{V}$ to $\mathcal{V'}$ can be realized by the general procedure of \S\ref{subsection: A general procedure}. An important disctinction with previous examples is that the generating diagram \eqref{diagram: change of enriching category} does not factor through $\dim: \nec\rightarrow \square$ in this case.

Suppose we have a monoidal adjunction
$$
L: \mathcal{V}\leftrightarrows \mathcal{V}': R
$$
Then since $L$ is strong monoidal, post-composition with $L$ induces a functor
$$
\tilde{L}: \ts\mathcal{V}\rightarrow \ts\mathcal{V}'
$$
This functor always has a right-adjoint, which we will construct as follows. First note that $(\mathcal{V}')^{\nec^{op}}$ is canonically tensored and enriched over $\mathcal{V}'$. We denote its tensoring by $-\cdot-$ and its $\mathcal{V}'$-enrichment by $[-,-]_{\mathcal{V}'}$. Then we can also consider it as tensored over $\mathcal{V}$ by defining its tensoring and $\mathcal{V}$-enrichment $[-,-]_{\mathcal{V}}$ via the adjunction $L\dashv R$:
$$
V\cdot Y = L(V)\cdot Y\qquad \text{and}\qquad [X,Y]_{\mathcal{V}} = R([X,Y]_{\mathcal{V}'})
$$
for all $V\in \mathcal{V}$ and $X,Y\in (\mathcal{V}')^{\nec^{op}}$. Now consider the following diagram
\begin{equation}\label{diagram: change of enriching category}
\nec\xrightarrow{\yo} \Set^{\nec^{op}}\xrightarrow{F} (\mathcal{V}')^{\nec^{op}}
\end{equation}
with $\yo$ the Yoneda embedding, which is strong monoidal for the Day convolution.

\begin{Prop}\label{proposition: change of enriching category}
Let $N^{F\yo}_{\mathcal{V}}: \mathcal{V}'\Cat_{\nec}\rightarrow \ts\mathcal{V}$ be the nerve generated by \eqref{diagram: change of enriching category}. Then $\tilde{R} = N^{F\yo}_{\mathcal{V}}\circ (-)^{nec}: \ts\mathcal{V}'\rightarrow \ts\mathcal{V}$ is right-adjoint to $\tilde{L}$.
\end{Prop}
\begin{proof}
Note that for all $X\in (\mathcal{V}')^{\nec^{op}}$ and $T\in \nec$, we have a natural isomorphism
$$
\mathfrak{n}^{F\yo}_{\mathcal{V}}(X)_{T} = [F\yo(T),X]_{\mathcal{V}} = R([F\yo(T),X]_{\mathcal{V'}})\simeq R(X(T))
$$
so that $\mathfrak{n}^{F\yo}_{\mathcal{V}}: (\mathcal{V'})^{\nec^{op}}\rightarrow \mathcal{V}^{\nec^{op}}$ is simply given by post-composition with $R$. Therefore, its left-adjoint $\mathfrak{l}^{F\yo}_{\mathcal{V}}$ is given by post-composition with $L$. Then since $L$ is strong monoidal and preserves colimits, we have $(-)^{nec}\circ \tilde{L}\simeq \mathfrak{l}^{F\yo}_{\mathcal{V}}\circ (-)^{nec}$. Now for templicial objects $X\in \ts\mathcal{V}$ and $Y\in \ts\mathcal{V}'$, we have natural isomorphisms:
$$
\ts\mathcal{V}(X,\tilde{R}(Y))\simeq \mathcal{V}'\Cat_{\nec}(\mathfrak{l}^{F\yo}_{\mathcal{V}}(X^{nec}),Y^{nec})\simeq \mathcal{V}'\Cat_{\nec}(\tilde{L}(X)^{nec},Y^{nec})\simeq \ts\mathcal{V}'(\tilde{L}(X),Y)
$$
where we used the fact that $(-)^{nec}$ is fully faithful (Theorem \ref{theorem: nec-temp adjunction}).
\end{proof}

\begin{Cor}\label{corollary: results for change of enriching category}
The following statements are true.
\begin{enumerate}[1.]
\item There is a natural isomorphism $\tilde{U}_{\mathcal{V}}\circ \tilde{R}\simeq \tilde{U}_{\mathcal{V}'}$ where $\tilde{U}_{\mathcal{V}}: \ts\mathcal{V}\rightarrow \SSet$ and $\tilde{U}_{\mathcal{V}'}: \ts\mathcal{V}'\rightarrow \SSet$ denote the forgetful functors.
\item Let $(X,S)$ be a quasi-category in $\mathcal{V}'$, then $\tilde{R}(X)$ is a quasi-category in $\mathcal{V}'$.
\item The adjunction $\tilde{L}\dashv \tilde{R}$ lifts to an adjunction $\Fs\mathcal{V}\leftrightarrows \Fs\mathcal{V}'$ along the forgetful functors $\Fs\mathcal{V}\rightarrow \ts\mathcal{V}$ and $\Fs\mathcal{V}'\rightarrow \ts\mathcal{V}'$.
\end{enumerate}
\end{Cor}
\begin{proof}
\begin{enumerate}[1.]
\item By Theorem \ref{theorem: underlying D-nerve is cat. nerve assoc. to D}, we have $\tilde{U}_{\mathcal{V}}\circ \tilde{R}\simeq N^{F\yo}\circ (-)^{nec}$. Now as before, $\mathfrak{n}^{F\yo}: (\mathcal{V}')^{\nec^{op}}\rightarrow \Set^{\nec^{op}}$ is given by post-composition with $U: \mathcal{V}'\rightarrow \Set$  and thus $\tilde{U}\circ \tilde{R}\simeq (-)^{temp}\circ \mathcal{U}\circ (-)^{nec}$ which coincides with $\tilde{U}_{\mathcal{V}'}$ by \cite[Proposition 3.14]{lowen2024enriched}.
\item By Theorem \ref{theorem: D-nerve is quasi-cat.} and Proposition \ref{proposition: change of enriching category}, it suffices to show that for all $0 < j < n$ and $a,b\in S$, any morphism $\mathfrak{l}^{F\yo}((\Lambda^{n}_{j})_{\bullet}(0,n))\rightarrow X_{\bullet}(a,b)$ extends to a morphism $F(\Delta^{n}_{\bullet}(0,n))$. But since $\mathfrak{l}^{F\yo}$ is simply given by post-composing with $F: \Set\rightarrow \mathcal{V}'$, this is exactly the condition that $X$ is a quasi-category in $\mathcal{V'}$.
\item Since $L$ is strong monoidal, $\tilde{L}: \ts\mathcal{V}\rightarrow \ts\mathcal{V}'$ clearly lifts to a functor $\Fs\mathcal{V}\rightarrow \Fs\mathcal{V}'$. Similar to the proof of Proposition \ref{proposition: change of enriching category}, this lift is left-adjoint to the functor $(-)^{temp}\circ \mathcal{R}\circ (-)^{nec}$ where $\mathcal{R}: (\mathcal{V}')^{\overline{\nec}^{op}}\Cat\rightarrow \mathcal{V}^{\overline{\nec}^{op}}\Cat$ is given by post-composition with $R$ on the level of hom-objects. It then follows from Theorem \ref{theorem: Frob. temp. nec. adjunction} that $(-)^{temp}\circ \mathcal{R}\circ (-)^{nec}$ also lifts $\tilde{R}$.
\end{enumerate}
\end{proof}

\begin{Rem}
While Theorem \ref{theorem: explicitation of left-adjoint for free temp. obj.} is applicable in this case, it will not return a more explicit description of $L$. We already have from the definition that $\tilde{L}(X)_{n}(a,b) = L(X_{n}(a,b))$ for all $(X,S)\in \ts\mathcal{V}$, $n\geq 0$ and $a,b\in S$.
\end{Rem}

\begin{Ex}
If we choose the adjunction $L\dashv R$ to be the free-forgetful adjunction $F: \Set\leftrightarrows \mathcal{V}: U$, then Proposition \ref{proposition: change of enriching category} recovers the adjunction $\tilde{F}: \SSet\leftrightarrows \ts\mathcal{V}: \tilde{U}$ of Proposition \ref{proposition: temp. obj. results}.2. In particular, Corollary \ref{corollary: results for change of enriching category}.2 recovers \cite[Corollary 5.13]{lowen2024enriched}.
\end{Ex}

\subsection{Free Frobenius structures}\label{subsection: Free Frobenius structures}

Consider the functor forgetting Frobenius structures, which we'll denote in this subsection by
$$
\mathfrak{u}: \Fs\mathcal{V}\rightarrow \ts\mathcal{V}
$$
We make use of the general procedure of \S\ref{subsection: A general procedure} to show that it has a left-adjoint $(-)^{Frob}: \ts\mathcal{V}\rightarrow \Fs\mathcal{V}$ and describe it more explicitly.

Note that $\mathcal{V}^{\overline{\nec}^{op}}$ is canonically tensored over $\mathcal{V}$. Consider the following diagram
\begin{equation}\label{diagram: free Frobenius structures}
\nec\xrightarrow{i} \overline{\nec}\xrightarrow{\yo} \Set^{\overline{\nec}^{op}}\xrightarrow{F} \mathcal{V}^{\overline{\nec}^{op}}
\end{equation}
where $i$ is the inclusion and $\yo$ is the Yoneda embedding, which are both strong monoidal. Note that again, this diagram does not factor through $\dim: \nec\rightarrow \square$.

\begin{Prop}\label{proposition: free Frobenius structure}
Let $N^{F\yo i}_{\mathcal{V}}: \mathcal{V}^{\overline{\nec}^{op}}\Cat\rightarrow \ts\mathcal{V}$ be the nerve functor generated by \eqref{diagram: free Frobenius structures}. Then $N^{F\yo i}_{\mathcal{V}}\circ (-)^{nec}: \Fs\mathcal{V}\rightarrow \ts\mathcal{V}$ is naturally isomorphic to $\mathfrak{u}$.
\end{Prop}
\begin{proof}
Note that for all $X\in \Fs\mathcal{V}$ and $T\in \nec$, we have a natural isomorphism $\mathfrak{n}^{F\yo i}_{\mathcal{V}}(X)_{T} = [F\yo(T),X]\simeq X(T)$ and thus $\mathfrak{n}^{F\yo i}_{\mathcal{V}}$ coincides with the restriction functor $res_{i}: \mathcal{V}^{\overline{\nec}^{op}}\rightarrow \mathcal{V}^{\nec^{op}}$. The result then follows from Theorem \ref{theorem: Frob. temp. nec. adjunction}.
\end{proof}

Recall the unique factorization of morphisms in $\overline{\nec}$ presented in Remark \ref{remark: unique decomp. of extended necklace map}. In order to describe the left-adjoint of the forgetful functor $\mathfrak{u}: \Fs\mathcal{V}\rightarrow \ts\mathcal{V}$ by means of Theorem \ref{theorem: explicitation of left-adjoint for free temp. obj.}, we require a different factorization in $\overline{\nec}$.

\begin{Def}\label{definition: extended necklaces factorization system 2}
We denote
\begin{enumerate}[1.]
\item $\overline{\nec}_{+} = \nec_{+}$, the monoidal subcategory of $\overline{\nec}$ of all injective necklace maps.
\item $\overline{\nec}_{-}$ for the monoidal subcategory of all maps $(f,U'): (T,p)\rightarrow (U,q)$ in $\overline{\nec}$ such that $U' = U$ and $U\cup f([p]) = [q]$.
\end{enumerate}
Note that $\overline{\nec}_{+}$ contains all active injective and inert necklace maps, while $\overline{\nec}_{-}$ contains all active surjective necklace maps and all coinert maps. In addition, it also contains the maps $\nu^{co}_{p,q}\delta_{p}: \Delta^{p+q-1}\rightarrow \Delta^{p}\vee \Delta^{q}$ for all $p,q > 0$.

It is an easy verification to see that both $\overline{\nec}_{+}$ and $\overline{\nec}_{-}$ are indeed closed under composition and taking wedges $\vee$.
\end{Def}

\begin{Prop}\label{proposition: extended necklace factorization}
The subcategories $(\overline{\nec}_{-},\overline{\nec}_{+})$ form an (orthogonal) factorization system on $\overline{\nec}$.
\end{Prop}
\begin{proof}
Let $(f,U'): (T,p)\rightarrow (U,q)$ be a map in $\overline{\nec}$. We wish to show that $(f,U')$ factors uniquely as a map in $\overline{\nec}_{-}$ followed by a map in $\overline{\nec}_{+}$. Since any active surjective necklace map belongs to $\overline{\nec}_{-}$ and any inert map belongs to $\overline{\nec}_{+}$, we may assume by Remark \ref{remark: unique decomp. of extended necklace map} that $f$ is the composite of an active injective necklace map and a coinert map, so that $f: [p]\rightarrow [q]$ is injective and $U' = U$. Let us first show uniqueness. Suppose we have maps $g_{i}: (T,p)\rightarrow (V_{i},r_{i})$ in $\overline{\nec}_{-}$, and $h_{i}: (V_{i},r_{i})\rightarrow (U,q)$ in $\overline{\nec}_{+}$ such that $f = h_{i}\circ g_{i}$ in $\overline{\nec}$ for $i\in \{1,2\}$. This implies that $h_{i}(V_{i}) = U$. Then applying $h_{i}$ to the equation $V_{i}\cup g_{i}([p]) = [r_{i}]$, we find $U\cup f([p]) = h_{i}(V_{i})\cup f([p]) = h_{i}([r_{i}])$. As $h_{i}$ is injective, this implies that $r_{1} = r_{2}$ and $h_{1} = h_{2}$. It follows that also $g_{1} = g_{2}$.

To show existence, set $r = \vert f([p])\cup U\vert - 1$ and define $h: [r]\hookrightarrow [q]$ to be the unique injective morphism in $\fint$ such that $h([r]) = f([p])\cup U$. Then there is clearly a unique morphism $g: [p]\rightarrow [r]$ such that $f = h\circ g$ in $\fint$. Now consider the necklace $(V,r)$ with $V = h^{-1}(U)$. As $U$ belongs to the image of $h$, we have $h(V) = U$ and thus $h: (V,r)\rightarrow (U,q)$ is a necklace map. Further, $(g,V): (T,p)\rightarrow (V,r)$ belongs to $\overline{\nec}_{-}$ since $V\cup g([p]) = h^{-1}(U\cup f([p])) = [r]$. Finally, it follows from $h(U) = V$ that also $h\circ g = f$ in $\overline{\nec}$.
\end{proof}

\begin{Prop}\label{proposition: extended necklace cat. generators}
The monoidal subcategory $\overline{\nec}_{-}$ of $\overline{\nec}$ is generated by the maps $\sigma_{i}: \Delta^{n+1}\rightarrow \Delta^{n}$, $\nu^{co}_{p,q}: \Delta^{p+q}\rightarrow \Delta^{p}\vee \Delta^{q}$ and $\nu^{co}_{p,q}\delta_{p}: \Delta^{p+q-1}\rightarrow \Delta^{p}\vee \Delta^{q}$, for $0\leq i\leq n$ and $p,q > 0$.
\end{Prop}
\begin{proof}
Let $(f,U): (T,p)\rightarrow (U,q)$ be a map in $\overline{\nec}_{-}$, so $f(T)\subseteq U$ and $U\cup f([p]) = [q]$. Factor the map $f: [p]\rightarrow [q]$ in $\fint$ as $\delta\circ \sigma$ with $\sigma: [p]\twoheadrightarrow [r]$ surjective and $\delta: [r]\hookrightarrow [q]$ injective. Then we can factor $(f,U)$ in $\overline{\nec}$ as
$$
(f,U): (T,p)\xrightarrow{(\sigma,\sigma(T))} (\sigma(T),q)\xrightarrow{(\delta,V)} (V,q)\xrightarrow{(\id,U)} (U,q) 
$$
where $V = f(T)\cup ([q]\setminus f([p]))\subseteq U$. Then $(\sigma,\sigma(T))$ is an active surjective necklace map and $(\id_{[q]},U)$ is coinert, which are monoidally generated by the maps $\sigma_{i}$ and $\nu^{co}_{p,q}$ respectively. Then it remains to show that $(\delta,V)$ is monoidally generated by the maps $\nu^{co}_{p,q}\delta_{p}$. Since $\delta(\sigma(T))\subseteq V$, we can write $(\delta,V)$ as a wedge sum of maps $\Delta^{n}\rightarrow V'$, so we may assume that $\sigma(T) = \{0 < n\}$. In that case, $[q]\setminus \delta([r]) = V\setminus \{0 < q\}$. Writing $V = \{0 = v_{0} < v_{1} < \dots < v_{k} = q\}$, we thus have $\delta = \delta_{v_{k-1}}\dots\delta_{v_{1}}$ in $\fint$ and therefore
\begin{align*}
(\delta,V) &= \nu^{co}_{v_{1},v_{2}-v_{1},\dots,q-v_{k-1}}\delta_{v_{k-1}}\dots \delta_{v_{1}}\\
&= \left(\id\vee\dots \vee \id\vee \nu^{co}_{v_{k-1}-v_{k-2},q-v_{k-1}}\delta_{v_{k-1}-v_{k-2}}\right)\dots (\id\vee \nu^{co}_{v_{2}-v_{1},q-v_{2}}\delta_{v_{2}-v_{1}})\nu^{co}_{v_{1},q-v_{1}}\delta_{v_{1}}
\end{align*}
\end{proof}

\begin{Lem}\label{lemma: splitting of necklace maps}
Let $T_{1}$, $T_{2}$ and $U$ be necklaces and $f: T_{1}\vee T_{2}\rightarrow U$ a map in $\overline{\nec}$. Then there exist unique necklaces $U_{i}$ and maps $f_{i}: T_{i}\rightarrow U_{i}$ in $\overline{\nec}$ for $i\in \{1,2\}$ such that $U_{1}\vee U_{2} = U\cup \{p\}$ and $f = \nu(f_{1}\vee f_{2})$ with $\nu: U\cup\{p\}\hookrightarrow U$ the inert map. Moreover, if $f$ is a necklace map, then so are $f_{1}$ and $f_{2}$.
\end{Lem}
\begin{proof}
Consider necklaces $(T_{1},k)$, $(T_{2},l)$ and $(U,n)$ and a map $(f,U'): T_{1}\vee T_{2}\rightarrow U$ in $\overline{\nec}$ with $f: [k+l]\rightarrow [n]$ in $\fint$ and $f(T_{1}\vee T_{2})\cup U\subseteq U'\subseteq [n]$. There exist unique morphisms $f_{1}: [k]\rightarrow [p]$ and $f_{2}: [l]\rightarrow [q]$ in $\fint$ such that $f_{1} + f_{2}$, where $p = f(k)$ and $q = n-p$. Then there exist unique necklaces $(U_{1},p)$ and $(U_{2},q)$ such that $U_{1}\vee U_{2} = U\cup\{p\}$, as well as $(U'_{1},p)$ and $(U'_{2},q)$ such that $U'_{1}\vee U'_{2} = U'\cup\{p\}$. It follows that $(f_{i},U'_{i}): T_{i}\rightarrow U_{i}$ is a map in $\overline{\nec}$ for $i\in \{1,2\}$. To verify that $\nu(f_{1}\vee f_{2}) = f$, it suffices to note that
$$
(U_{1}\vee U_{2})\cup (U'_{1}\vee U'_{2}) = U\cup U'\cup \{p\} = U'
$$
Finally, suppose that $f$ is a necklace map, i.e. $U' = f(T_{1}\vee T_{2})$. Then for $i\in \{1,2\}$ we necessarily have $U'_{i} = f_{i}(T_{i})$ and thus $f_{i}$ is a necklace map as well.
\end{proof}

\begin{Cor}\label{corollary: free Frobenius structure}
The functor $\mathfrak{u}$ has a left-adjoint $(-)^{Frob}: \ts\mathcal{V}\rightarrow \Fs\mathcal{V}$. For any templicial object $(X,S)$ that has non-degenerate simplices, $a,b\in S$ and $n\geq 0$, we have
$$
X^{Frob}_{n}(a,b)\simeq \coprod_{\substack{T\in \nec\\ f: \Delta^{n}\rightarrow T\\ \text{in }\overline{\nec}_{-}}}X^{nd}_{T}(a,b)
$$
\end{Cor}
\begin{proof}
In view of Proposition \ref{proposition: free Frobenius structure}, we will show that the left-adjoint $L^{F\yo i}_{\mathcal{V}}$ factors as
$$
\ts\mathcal{V}\xrightarrow{(-)^{Frob}} \Fs\mathcal{V}\xrightarrow{(-)^{nec}} \mathcal{V}^{\overline{\nec}^{op}}\Cat
$$
Then $(-)^{Frob}$ must necessarily be the left-adjoint of $\mathfrak{u}$ since $(-)^{nec}$ is fully faithful. In other words, given a templicial object $(X,S)$, we show that for all $T,U\in \nec$, the quiver morphism $L^{F\yo i}_{\mathcal{V}}(X)_{T}\otimes_{S} L^{F\yo i}_{\mathcal{V}}(X)_{U}\rightarrow L^{F\yo i}_{\mathcal{V}}(X)_{T\vee U}$ is an isomorphism.

As before, $\mathfrak{n}^{F\yo i}_{\mathcal{V}}$ coincides with the restriction $res_{i}: \mathcal{V}^{\overline{\nec}^{op}}\rightarrow \mathcal{V}^{\nec^{op}}$, which is left-adjoint to $\Lan_{i}$, the left Kan extension along $i$. Thus we have $L^{F\yo i}_{\mathcal{V}}\simeq (-)^{\Lan_{i}}\circ (-)^{nec}$. Hence, the quiver morphism above is
\begin{equation}\label{equation: free Frob. structure functor preserves temp. obj.}
\colim_{\substack{(T\rightarrow T')\in (T\downarrow i)\\ (U\rightarrow U')\in (U\downarrow i)}}X_{T'\vee U'}\rightarrow \colim_{\substack{(T\vee U\rightarrow V)\\ \in (T\vee U\downarrow i)}}X_{V}
\end{equation}
induced by $\vee: (T\downarrow i)\times (U\downarrow i)\rightarrow (T\vee U\downarrow i): (T\rightarrow T',U\rightarrow U')\mapsto (T\vee U\rightarrow T'\vee U')$. Given a map $f: T\vee U\rightarrow V$ in $\overline{\nec}$, Lemma \ref{lemma: splitting of necklace maps} provides unique maps $f_{1}: T\rightarrow V_{1}$ and $T_{2}\rightarrow V_{2}$ in $\overline{\nec}$ such that $V_{1}\vee V_{2} = V\cup \{p\}$ and $f = \nu(f_{1}\vee f_{2})$ with $\nu$ inert. It follows that the assignment $f\mapsto (f_{1},f_{2})$ extends to a functor $(T\vee U\downarrow i)\rightarrow (T\downarrow i)\times (U\downarrow i)$ which is right-adjoint to $\vee$. Hence, between opposite categories, $\vee$ is a right-adjoint and thus a final functor. Consequently, the quiver morphism \eqref{equation: free Frob. structure functor preserves temp. obj.} is an isomorphism.

Now assume $X$ has non-degenerate simplices. Given $T\in \nec$, let $\pi_{T}: \mathcal{V}^{\overline{\nec}^{op}}\rightarrow \mathcal{V}: Y\mapsto Y_{T}$. Since $\nec_{-}\subseteq \overline{\nec}_{-}$ the functor $F(\overline{\nec}_{-}(T,-)): \nec_{-}\rightarrow \mathcal{V}$ is well-defined. Moreover, it directly follows from Proposition \ref{proposition: extended necklace factorization} that $\pi_{T}F\yo i = F(\overline{\nec}(T,-))\simeq \Lan_{\iota}F(\overline{\nec}_{-}(T,-))$ where $\iota: \nec_{-}\hookrightarrow \nec$ is the inclusion. Hence, the description of $X^{Frob}_{n}(a,b)$ for $n\geq 0$ follows from Theorem \ref{theorem: explicitation of left-adjoint for free temp. obj.}.
\end{proof}

\begin{Rem}
While Theorem \ref{theorem: underlying D-nerve is cat. nerve assoc. to D} is applicable here, it will just returns Corollary \ref{corollary: results for change of enriching category}.3. Further Theorem \ref{theorem: D-nerve is quasi-cat.} doesn't provide a very informative condition for a Frobenius templicial object to be a quasi-category in $\mathcal{V}$.
\end{Rem}

\begin{Exs}
\begin{enumerate}[1.]
\item Given a necklace $(T,p)$, we have $T^{Frob}_{n}(0,p)\simeq \overline{\nec}(\Delta^{n},T)$ for all $n > 0$. This can be seen through Corollary \ref{corollary: free Frobenius structure} or more directly from the definition.
\item Consider the simplicial circle $S^{1}$, defined here as the coequalizer of the coface maps $\delta_{0},\delta_{1}: \Delta^{0}\rightrightarrows \Delta^{1}$. Recall that $S^{1}$ has exactly one vertex $*$ and one edge $e$ which are its only non-degenerate simplices. We have by Corollary \ref{corollary: free Frobenius structure} that
$$
(S^{1})_{n}^{Frob}(*,*)\simeq \coprod_{\substack{T\in \nec\\ f: \Delta^{n}\rightarrow T\\ \text{in }\overline{\nec}_{-}}}(S^{1})^{nd}_{T}(*,*)\simeq \coprod_{p\geq 0}\fint([n],[p])
$$
for all $n > 0$, where we used that $(S^{1})^{nd}_{T}(*,*)$ is only non-empty when $T$ is a spine, in which case $(S^{1})^{nd}_{T}(*,*)$ is a singleton. For a morphism $[m]\rightarrow [n]$ in $\fint$, the induced map $(S^{1})^{Frob}_{n}\rightarrow (S^{1})^{Frob}_{m}$ is given by pre-composition in the obvious way. This defines the degeneracy and inner face maps. Note that while $S^{1}$ is a finite simplicial set, $(S^{1})^{Frob}$ has infinitely many non-degenerate simplices in each dimension.
\end{enumerate}
\end{Exs}

\appendix

\section{The generating diagram of the differential graded nerve}\label{section: The generating diagram of the differential graded nerve}

We return to the proof of Proposition \ref{proposition: dg-nerve}, showing that the templicial dg-nerve of \cite[Definition 3.15]{lowen2023frobenius} is generated by $\dg: \nec\rightarrow \Ch(k)$ \eqref{diagram: dg}. The templicial dg-nerve is defined by a factorization through $\Fs\Mod(k)$:
$$
k\Cat_{dg}\xrightarrow{\overline{N}^{dg}_{k}} \Fs\Mod(k)\rightarrow \ts\Mod(k)
$$
and $\overline{N}^{dg}_{k}$ is even an equivalence when restricted to non-negatively graded dg-categories $k\Cat_{dg,\geq 0}$. On the other hand, the nerve generated by $\dg$ also factors through $\Fs\Mod(k)$ by Corollary \ref{corollary: cubical nerve has Frob. structure}. We will prove Proposition \ref{proposition: dg-nerve} by characterizing the Frobenius templicial maps from an arbitrary Frobenius templicial module into both factorizations.

Fix a Frobenius templicial $k$-module $(X,S)$ and a small dg-category $\mathcal{C}$ with object set $S$. We denote the comultiplication and counit of $X$ by $\mu$ and $\epsilon$ respectively, and we denote the composition law and identities of $\mathcal{C}$ by $m$ and $u$ respectively. Define
\begin{itemize}
\item the set $\mathcal{S}_{1}$ of all collections of morphisms in $k\Quiv_{S}$:
$$
\left(\beta_{n}: X_{n}\rightarrow \mathcal{C}_{n-1}\right)_{n > 0}
$$
\item the set $\mathcal{S}_{2}$ of all collections of morphisms in $k\Quiv_{S}$:
$$
\left(H_{g}: X_{T}\rightarrow \mathcal{C}_{\dim U}\right)_{\substack{g: U\hookrightarrow T\text{ inj.}\\ \text{in }\nec}}
$$
such that
\begin{enumerate}[(a)]
\item\label{item: dg-nerve lemma condition 1} for all injective necklace maps $g_{1}$ and $g_{2}$, we have
$$
m(H_{g_{1}}\otimes H_{g_{2}}) = H_{g_{1}\vee g_{2}}
$$
\item\label{item: dg-nerve lemma condition 2} we have
$$
u\epsilon = H_{\id_{\Delta^{0}}}
$$
\item\label{item: dg-nerve lemma condition 3} for all injective necklace maps $g: U\hookrightarrow T$ and $f: T\hookrightarrow T'$, we have
$$
H_{g}\circ X(f) = H_{fg}
$$
\end{enumerate}
\end{itemize}

\begin{Lem}\label{lemma: dg-nerve lemma 1}
The following maps are inverse bijections:
$$
\mathcal{S}_{2}\rightarrow \mathcal{S}_{1}: (H_{g})_{g}\mapsto (H_{\id_{\Delta^{n}}})_{n > 0}\quad \text{and}\quad \mathcal{S}_{1}\rightarrow \mathcal{S}_{2}: (\beta_{n})_{n}\mapsto (m\beta_{U}X(g))_{g: U\hookrightarrow T}
$$
where $\beta_{U} = \beta_{n_{1}}\otimes ...\otimes \beta_{q-n_{l-1}}$ for any necklace $U = \{0 = u_{0} < u_{1} < ... < u_{l} = q\}$.
\end{Lem}
\begin{proof}
Given $(\beta_{n})_{n > 0}\in \mathcal{S}_{1}$, note that $(m\beta_{U}X(g))_{g: U\hookrightarrow T}$ clearly satisfies conditions \eqref{item: dg-nerve lemma condition 1}-\eqref{item: dg-nerve lemma condition 3} above. Moreover, $m\beta_{\{0 < n\}}X(\id_{\Delta^{n}}) = \beta_{n}$ for all $n > 0$ so that the map $\mathcal{S}_{2}\rightarrow \mathcal{S}_{1}$ is surjective. Further, note that for every injective necklace map $g: U\hookrightarrow T_{1}\vee T_{2}$, there exist unique injective necklace maps $g_{i}: U_{i}\hookrightarrow T_{i}$ such that $g = g_{1}\vee g_{2}$. So it follows from conditions \eqref{item: dg-nerve lemma condition 1}-\eqref{item: dg-nerve lemma condition 3} that a collection $(H_{g})_{g}\in \mathcal{S}_{2}$ is completely determined by $H_{\id_{\Delta^{n}}}$ for $n > 0$. Hence, the map $\mathcal{S}_{2}\rightarrow \mathcal{S}_{1}$ is injective as well.
\end{proof}

\begin{Lem}\label{lemma: dg-nerve lemma 2}
For any $(H_{g})_{g}\in \mathcal{S}_{2}$ and $(\beta_{n})_{n > 0} = (H_{\id_{\Delta^{n}}})_{n > 0}$, the following statements are equivalent:
\begin{enumerate}
\item for all injective necklace maps $g: U\hookrightarrow T$ with $U^{c} = \{i_{1} < ... < i_{n}\}$,
$$
\partial H_{g} = \sum_{j=1}^{n}(-1)^{j-1}\left(H_{g\delta_{i_{j}}} - H_{g\nu_{i_{j},n-i_{j}}}\right)
$$
\item for all $n > 0$,
$$
\partial\beta_{n} = \sum_{j=1}^{n-1}(-1)^{j-1}\left(\beta_{n-1}d^{X}_{j} - m(\beta_{j}\otimes \beta_{n-j})\mu^{X}_{j,n-j}\right)
$$
\end{enumerate}
\end{Lem}
\begin{proof}
The implication $(1)\Rightarrow (2)$ is immediate from Lemma \ref{lemma: dg-nerve lemma 1}. Conversely, assume that $(2)$ holds and take $g: (U,q)\hookrightarrow (T,p)$ an injective necklace map. Write $U^{c} = \{i_{1} < ... < i_{n}\}$ and $U = \{0 = u_{0} < u_{1} < ... < u_{l} = q\}$. Then
\begin{align*}
&\partial H_{g} = \partial m(\beta_{U})X(g) = \sum_{i=1}^{l}(-1)^{u_{i-1}-i+1}m(\beta_{u_{1}}\otimes ...\otimes \partial\beta_{u_{i}-u_{i-1}}\otimes ...\beta_{q-u_{l-1}})X(g)\\
&\, = \sum_{i=1}^{l}\sum_{j=1}^{u_{i}-u_{i-1}-1}(-1)^{u_{i-1}-i+j}\Big[ m(\beta_{u_{1}}\otimes ...\otimes \beta_{u_{i}-u_{i-1}-1}\otimes ...\otimes \beta_{q-u_{l-1}})X(g\delta_{u_{i-1}+j})\\
&\,\quad - m(\beta_{u_{1}}\otimes ...\otimes \beta_{j}\otimes \beta_{u_{i}-u_{i-1}-j}\otimes ...\otimes \beta_{q-u_{l-1}})X(g\nu_{u_{i-1}+j,q-u_{i-1}-j})\Big]\\
&\, = \sum_{k=1}^{n}(-1)^{k-1}\left(H_{g\delta_{i_{k}}} - H_{g\nu_{i_{k},q-i_{k}}}\right)
\end{align*}
where we used that for all $k\in \{1,...,n\}$, $i_{k} = u_{i-1} + j$ for a unique $i\in \{1,...,l\}$ and $0 < j < u_{i}-u_{i-1}$. Moreover, note that for this $i$, we have $i_{k} = i + k - 1$.
\end{proof}

\begin{Lem}\label{lemma: dg-nerve lemma 3}
For any $(H_{g})_{g}\in \mathcal{S}_{2}$ and $(\beta_{n})_{n > 0} = (H_{\id_{\Delta^{n}}})_{n > 0}$, the following statements are equivalent:
\begin{enumerate}
\item for any injective necklace map $g: U\hookrightarrow T$ and any map $f: T\rightarrow T'$ in $\overline{\nec}_{-}$, factor $fg = \sigma g'$ with $g: U'\hookrightarrow T'$ an injective necklace map and $\sigma$ in $\overline{\nec}_{-}$ (by Proposition \ref{proposition: extended necklace factorization}). Then
$$
H_{g}X(f) =
\begin{cases}
H_{g'} & \text{if } \dim(U) = \dim(U')\\
0 & \text{otherwise}
\end{cases}
$$
\item for all $n\geq 0$ and $0\leq i\leq n$,
$$
\beta_{n+1}s^{X}_{i} =
\begin{cases}
u\epsilon & \text{if }n = i = 0\\
0 & \text{otherwise}
\end{cases}
$$
and for all $p,q > 0$,
$$
\beta_{p+q}Z^{p,q}_{X} = 0\quad \text{and}\quad \beta_{p+q-1}d^{X}_{p}Z^{p,q}_{X} = m(\beta_{p}\otimes \beta_{q})
$$
\end{enumerate}
\end{Lem}
\begin{proof}

Note that $\sigma_{i}: \Delta^{n+1}\twoheadrightarrow \Delta^{n}$ for $n > 0$ and $\nu^{co}_{p,q}: \Delta^{p+q}\rightarrow \Delta^{p}\vee \Delta^{q}$ strictly decrease dimenions, while $\nu^{co}_{p,q}\delta_{p}$ and $\sigma_{0}$ preserve dimensions. Thus by Lemma \ref{lemma: dg-nerve lemma 1}, statement $(1)$ specializes to $(2)$ by choosing $g$ to be the identity and $f$ to be $\sigma_{i}$, $\nu^{co}_{p,q}$ and $\nu^{co}_{p,q}\delta_{p}$ respectively. By Proposition \ref{proposition: extended necklace cat. generators}, $\overline{\nec}_{-}$ is generated by these maps as a monoidal category and thus it follows by conditions \eqref{item: dg-nerve lemma condition 1}-\eqref{item: dg-nerve lemma condition 3} that $(2)\Rightarrow (1)$ also holds. 
\end{proof}

\begin{proof}[Proof of Proposition \ref{proposition: dg-nerve}]
Let $\overline{\dg}: \overline{\nec}\rightarrow \Ch(k)$ be the extension of $\dg$ from Corollary \ref{corollary: cubical nerve has Frob. structure} and consider the induced adjunction:
$$
\mathfrak{l}^{\overline{\dg}}_{k}: \Ch(k)\leftrightarrows \Mod(k)^{\overline{\nec}^{op}}: \mathfrak{n}^{\overline{\dg}}_{k}$$
Then the factorization of the nerve generated by $\dg$ through $\Fs\Mod(k)\rightarrow \ts\Mod(k)$ is given by $(-)^{temp}\circ \mathfrak{n}^{\overline{\dg}}_{k}$ and the statement will follow by providing a natural bijection:
$$
k\Cat_{dg}(\mathfrak{l}^{\overline{\dg}}_{k}(X^{nec}),\mathcal{C})\simeq \Fs\Mod(k)(X,\overline{N}^{\dg}_{k}(\mathcal{C}))
$$
A dg-functor $\mathfrak{l}^{\overline{\dg}}_{k}(X^{nec})\rightarrow \mathcal{C}$ consists of a map of sets $f: S\rightarrow \Ob(\mathcal{C})$ along with a quiver chain map $H: \colim_{T\in \nec}^{X_{T}}\overline{dg}(T)_{\bullet}\rightarrow f^{*}(\mathcal{C}_{\bullet})$ compatible with the identities and composition laws. Replacing $\mathcal{C}$ by $f^{*}(\mathcal{C})$, we may safely assume that $f = \id_{S}$. Then we see that $H$ is determined by a collection of quiver linear maps $H_{g}: X_{T}\rightarrow \mathcal{C}_{\dim U}$ satisfying condition $(1)$ of Lemma \ref{lemma: dg-nerve lemma 3}. The fact that $H$ is compatible with the composition laws and identities precisely translates to conditions \eqref{item: dg-nerve lemma condition 1} and \eqref{item: dg-nerve lemma condition 2} above. Moreover, the fact that $H$ is compatible with the differentials exactly comes down to condition $(1)$ of Lemma \ref{lemma: dg-nerve lemma 2}. Hence the dg-functor $\mathfrak{l}^{\overline{\dg}}(X^{nec})\rightarrow \mathcal{C}$ is equivalent to a collection of quiver maps $(\beta_{n}: X_{n}\rightarrow \mathcal{C}_{n-1})_{n > 0}$ satisfying conditions $(2)$ of the same lemmas. By \cite[Corollary 3.22]{lowen2023frobenius}, this data is exactly equivalent to a Frobenius templicial map $X\rightarrow \overline{N}^{\dg}_{k}(\mathcal{C})$. Clearly this bijection is natural in $X$ and $\mathcal{C}$.
\end{proof}

\printbibliography

\end{document}